\documentclass[reqno,10pt]{amsart}
\pdfoutput=1

\usepackage[utf8]{inputenc}
\usepackage[english]{babel}
\usepackage{amsfonts}
\usepackage{amssymb}
\usepackage{bbm}
\usepackage{tikz}
\usepackage{svg}
\usepackage{scalerel}
\usepackage[right=2.5cm,bottom=2.5cm,footskip=30pt,left=2.5cm]{geometry}
\usepackage[autostyle]{csquotes}
\MakeOuterQuote{"}
\usepackage{color}
\usepackage{mathrsfs}
\usepackage{mathtools,leftidx}
\usepackage{enumitem}
\usepackage{amsmath}
\usepackage{mathrsfs}
\usepackage{accents}
\usepackage{scalerel}
\usepackage{aligned-overset}
\setlist[enumerate,2]{label={(\theenumi.\theenumii)},ref={(\theenumi.\theenumii)},leftmargin=1cm}
\setlist[enumerate,1]{label={(\theenumi)},ref={(\theenumi)},leftmargin=1cm}

\usetikzlibrary{svg.path,arrows,calc}
\tikzset{>=latex'}

\usepackage[toc]{appendix}
\usepackage{etoolbox}
\cslet{blx@noerroretextools}\empty
\usepackage[maxbibnames=9,giveninits=true,style=alphabetic,sorting=nyt,backend=biber]{biblatex}
\DeclareFieldFormat{titlecase:title}{\MakeSentenceCase*{#1}}
\renewbibmacro*{url+urldate}{%
\iffieldundef{urlyear}
  {}
  {}
\usebibmacro{url}}
\def\restrict#1{\raise-.5ex\hbox{\ensuremath|}_{#1}}

\usepackage[colorlinks, citecolor=citegreen, linkcolor=refred,urlcolor=blue, unicode,psdextra,hypertexnames=false]{hyperref}
\usepackage{cleveref}
\usepackage{orcidlink}
\crefname{enumi}{}{}
\crefname{enumii}{}{}
\expandafter\def\csname ver@etex.sty\endcsname{3000/12/31}

\makeatletter
 \def\author@andify{
 \nxandlist {\unskip{\kern.3cm} \penalty-2}
 {\unskip {\kern.3cm} \penalty-2}
 {\unskip {\kern.3cm} \penalty-2}}
\makeatother

\usepackage{accents}
\usepackage{autonum}
\usepackage{gasymbols}
\usepackage{nicefrac}

\definecolor{citegreen}{rgb}{0,0.3,0}
\definecolor{refred}{rgb}{0.5,0,0}

\def\mathring#1{\accentset{\circ}{#1}}
\usepackage[equation=section]{theorems}

\let\oldemail\email
\let\email\relax
\def\email#1{\oldemail{\href{mailto:#1}{\textcolor{black}{#1}}}}

\title[Fine properties of nonlinear potentials]{Fine properties of nonlinear potentials and a unified perspective on monotonicity formulas}

\author[L.~Benatti]{Luca Benatti \orcidlink{0000-0002-4685-7443}}
\address{L.~Benatti, Universit\`a di Pisa,
Largo Bruno Pontecorvo 5, 56127 Pisa, Italy}
\email{luca.benatti@dm.unipi.it}

\author[A.~Pluda]{Alessandra Pluda \orcidlink{0000-0003-4714-4119}}
\address{A.~Pluda, Universit\`a di Pisa,
Largo Bruno Pontecorvo 5, 56127 Pisa, Italy}
\email{alessandra.pluda@unipi.it}

\author[M.~Pozzetta]{Marco Pozzetta \orcidlink{0000-0002-2757-0826}}
\address{M.~Pozzetta (Corresponding Author), Politecnico di Milano, Via Bonardi 9, 20133 Milano, Italy}
\email{marco.pozzetta@polimi.it}

\renewcommand{\ncapa}{\mathfrak{c}}


\newcommand{\MF}{\mathscr{F}}
\newcommand{\IF}{\mathscr{G}}
\newcommand{\SQ}{\mathscr{Q}}

\newcommand{\Err}{\theta}

\newcommand{\sml}[1]{\scaleobj{.9}{(#1)}}

\begin{document}
 
\begin{abstract}
We rigorously show that a large family of monotone quantities along the weak inverse mean curvature flow is the limit case of the corresponding ones along the level sets of $p$-capacitary potentials. Such monotone quantities include  Willmore and Minkowski-type functionals on Riemannian manifolds with nonnegative Ricci curvature. In $3$-dimensional manifolds with nonnegative scalar curvature, we also recover the monotonicity of the Hawking mass and its nonlinear potential theoretic counterparts. This unified view is built on a refined analysis of $p$-capacitary potentials. We prove that they strongly converge in $W^{1,q}_{\loc}$ as $p\to 1^+$ to the inverse mean curvature flow and their level sets are curvature varifolds. Finally, we also deduce a Gauss--Bonnet-type theorem for level sets of $p$-capacitary potentials.
\end{abstract}  
\maketitle


\smallskip

\noindent \underline{\smash{Keywords}}: monotonicity formulas, geometric inequalities, nonlinear potential theory, inverse mean curvature flow, curvature varifolds.

\section{Introduction}

The inverse mean curvature flow (IMCF for short) is a geometric flow in which the speed of the evolved hypersurfaces is the inverse of the mean curvature. Its weak formulation is given by the level sets of a  locally Lipschitz solution $w_1$ of the following problem
    \begin{equation}\label{1-Laplaciano}
        \begin{cases}
            \div\left( \dfrac{\nabla w_1}{\abs{\nabla w_1}}\right)& =& \abs{\nabla w_1 } &\text{on $ M \smallsetminus \Omega $,}\\
         w_1&=&0& \text{on $\partial \Omega$.}
        \end{cases}
    \end{equation}
    Here, $\Omega$ is a compact domain in a noncompact Riemannian manifold $(M,g)$. 
    The solution to \cref{1-Laplaciano} should be understood in the nonstandard variational sense introduced by Huisken and Ilmanen \cite{huisken_inversemeancurvatureflow_2001}. In this groundbreaking paper, the monotonicity of the Hawking mass along the weak IMCF is the cornerstone of the proof of the Riemannian Penrose inequality. Although such monotonicity was well-known along the smooth IMCF \cite{geroch_energyextraction_1973}, the major issues in \cite{huisken_inversemeancurvatureflow_2001} are due to the mild regularity of solutions to \cref{1-Laplaciano} and the presence of fattening of its level sets.

Geometric inequalities proved through monotonicity formulas can be found in the seemingly distinct framework of linear potential theory. An early example of this approach is Colding's breakthrough \cite{colding_newmonotonicityformulasricci_2012}. Here, monotone quantities along the level sets of harmonic functions are used to estimate the distance between manifolds with nonnegative Ricci curvature and the spaces of cones. Later on, in \cite{colding_uniquenesstangentconeseinstein_2014}, both the monotonicity and the convergence rate of a suitably defined quantity imply the uniqueness of smooth tangent cones at infinity in Ricci-flat manifolds. In \cite{agostiniani_riemannianaspectspotentialtheory_2015}, a sphere theorem is proved for subsets of $\R^n$. A sharp Willmore-type inequality for bounded subsets of Riemannian manifolds with nonnegative Ricci curvature is established in \cite{agostiniani_monotonicityformulaspotentialtheory_2020, agostiniani_sharpgeometricinequalitiesclosed_2020}. This approach has also been effectively applied in $3$-dimensional manifolds with nonnegative scalar curvature. For instance, in \cite{munteanu_comparisontheorems3dmanifolds_2023}, a comparison theorem regarding the growth of Green’s functions is proved. In the same context, \cite{agostiniani_greenfunctionproofpositive_2024} provides a new proof of the Positive Mass theorem \cite{schoen_proofpositivemassconjecture_1979}, using monotonicity along Green function level sets. In \cite{Miao_2023}, the author established a lower bound of the mass-capacity ratio in terms of the Willmore functional at the boundary.

Moser \cite{moser_inversemeancurvatureflow_2007} (and the subsequent \cite{moser_inversemeancurvatureflow_2008,kotschwar_localgradientestimatesharmonic_2009,mari_flowlaplaceapproximationnew_2022}) has drawn the link between the two worlds of IMCF and harmonic potential. 
Taking $p \in [1,2]$, proper solutions $w_p$ to the family of problems involving the $p$-Laplace equation
    \begin{equation}\label{p-Laplaciano}
    \begin{cases}
        \div \left(\dfrac{\nabla w_p}{\abs{\nabla w_p}^{2-p}}\right)& =& \abs{\nabla w_p}^p &\text{on $M\smallsetminus \Omega $,}\\
     w_p&=&0& \text{on $\partial \Omega$,}
    \end{cases}
    \end{equation}
serve as an interpolation between the harmonic potential of $\Omega$ (which is given by $\ee^{-w_2}$) and the weak IMCF (which corresponds to the case $p=1$). More precisely, it can be proved that solutions to \cref{p-Laplaciano} converge to the weak IMCF locally uniformly as $p\to1^+$. 

As for $p=1$ and $p=2$, it is possible to build monotone quantities along the level sets of solutions to \cref{p-Laplaciano} for $p\in (1,2)$, and in turn to prove geometric and functional inequalities. For instance, a Minkowski-type inequality is provided in \cite{fogagnolo_geometricaspectscapacitarypotentials_2019,agostiniani_minkowskiinequalitiesnonlinearpotential_2022,benatti_minkowskiinequalitycompleteriemannian_2022} in Riemannian manifolds with nonnegative Ricci curvature. In a $3$-dimensional Riemannian manifold with nonnegative scalar curvature, one can find proofs of the Riemannian Penrose inequality via nonlinear potential theory \cite{agostiniani_riemannianpenroseinequalitynonlinear_2022}, its $p$-capacitary version \cite{xia_newmonotonicitycapacitaryfunctions_2024}, and also the growth comparison for the $p$-Green function \cite{chan_monotonicitygreenfunctions_2022}. In \cite{hirsch_monotonequantitiesharmonicfunctions_2024}, the results in \cite{Miao_2023} were extended to the case $p\in(1,2)$. 

Geometric and functional inequalities derived using the weak IMCF can be obtained as limits of their counterparts in nonlinear potential theory. This suggests that monotonicity formulas corresponding to the two approaches should be similarly related. Although this is formally true, it has never been rigorously proved in the literature. Only the paper by Moser \cite{moser_gerochmonotonicity_2015}  moves in this direction, attempting to derive the Huisken--Ilmanen's monotonicity for the Hawking mass using solutions to \cref{p-Laplaciano}. However, the corresponding monotonicity formulas for the solutions to \cref{p-Laplaciano} were only introduced subsequently in \cite{agostiniani_riemannianpenroseinequalitynonlinear_2022}.

In this paper, we propose a unified view of all these monotonicity formulas. Our first result can be stated as follows.

 \begin{theorem}[Monotonicity formulas]\label{main}
     Let $(M,g)$ be a noncompact Riemannian manifold of dimension $n\ge 3$.
     Let $\Omega \subset M$ be a compact set with $\CS^{1,1}$-boundary. Let $w_p$ be a solution to \cref{p-Laplaciano} for $p \in [1,2]$ and let $\Omega^{\sml{p}}_t \coloneqq \set{w_p \leq t}$. Assume there exists $\overline{T}>0$ such that $\Omega^{\sml{p}}_t \Subset M$ for every $t \in [0,\overline{T})$. Then, the following statements hold:
     \begin{enumerate}
         \item  $\partial \Omega^{\sml{p}}_t$ is a curvature varifold with square integrable second fundamental form for almost every $t \in (0,\overline{T})$;
         \item for every $\alpha \geq (n-p)/(n-1)$ the function $[0,\overline{T}) \ni t \mapsto \MF_{p}(t)$ defined as 
        \begin{equation}\label{funz-monotoni}
        \begin{split}
        \MF_{p}(t) \coloneqq\begin{multlined}[t][.6\textwidth]
            \ee^{\left(\frac{\alpha}{n-p} - 1\right)t} \int_{\partial \Omega^{\sml{p}}_t}\abs{\nabla w_p }^{\alpha+p-2} \left(\abs{\nabla w_p}\left(\frac{n-1}{n-p}- \frac{1}{\alpha}\right) -\H\right) \dif \Hff^{n-1} \\-\int_0^t \ee^{\left(\frac{\alpha}{n-p} - 1\right)s}\int_{\partial \Omega^{\sml{p}}_s} \abs{\nabla w_p}^{\alpha+p-3} \Ric(\nu,\nu)\dif \Hff^{n-1} \dif s
            \end{multlined}
            \end{split}
        \end{equation}
            is essentially monotone nondecreasing; and
     \item for every  nonnegative function $\psi \in \Lip_c([0,\overline{T}))$ we have
        \begin{equation}\label{derivata_F_p}
            - \int_0^{+\infty} \psi' (t)  \MF_{p}(t) \dif t - \psi(0) \MF_{p}(0)\geq   \int_0^{+\infty} \psi(t) \ee^{\left(\frac{\alpha}{n-p} - 1 \right)t} \int_{\partial \Omega^{\sml{p}}_t} \abs{\nabla w_p}^{\alpha+p-3} \SQ_{p}\dif \Hff^{n-1} \dif t,
        \end{equation}
        where 
        \begin{equation}\label{Qdef-intro}
        \SQ_{p}\coloneqq 
            \left[\alpha-(2-p)\right] \frac{\abs{\nabla^\top \abs{\nabla w_p}}^2}{\abs{\nabla w_p}^2} + \abs{ \mathring{\h}}^2 + \frac{1}{p-1} \left[ \alpha- \frac{n-p}{n-1}\right] \left[ {\H} - \frac{n-1}{n-p}\abs{ \nabla w_p}\right]^2.
        \end{equation}
     \end{enumerate}
 \end{theorem}
 
In the above statement, $\H$, $\mathring{\h}$ and $\nu$ respectively are the mean curvature, the traceless second fundamental form and the unit normal vector field of $\partial \Omega^{\sml{p}}_t$, while $\nabla^\top$ is the tangential gradient along $\partial \Omega^{\sml{p}}_t$. For $p=1$, the last term in $\SQ_{p}$ is understood to be zero. For $p>1$, we further get that $\MF_p$ belongs to $W^{1,1}_{\loc}$. The precise statements are given in \cref{thm:monotonicity-NPT} and \cref{thm:monotonicty_formula_IMCF}.

The function $\MF_p$ presents similarities with the expressions in \cite{agostiniani_sharpgeometricinequalitiesclosed_2020,agostiniani_riemannianpenroseinequalitynonlinear_2022,benatti_minkowskiinequalitycompleteriemannian_2022, hirsch_monotonequantitiesharmonicfunctions_2024,xia_newmonotonicitycapacitaryfunctions_2024,huisken_inversemeancurvatureflow_2001}. However, each of these papers assumes geometric or dimensional restrictions. Here, we avoid such hypotheses by incorporating in $\MF_p$ an additional term involving the Ricci tensor. Nevertheless, all the aforementioned results can be obtained by tailoring \cref{derivata_F_p} to each specific setting. Taking for example $p>1$ and $(M,g)$ with nonnegative Ricci curvature, the last integral in \cref{funz-monotoni} is manifestly nonpositive and one recovers the quantity appearing in \cite{benatti_minkowskiinequalitycompleteriemannian_2022}. This link is also technical. The proof of the monotonicity for $p>1$ (\cref{thm:monotonicity-NPT}) follows very closely the proof of \cite[Theorem 3.1]{benatti_minkowskiinequalitycompleteriemannian_2022}. As in that paper, we only exploit the properties of solutions to \cref{p-Laplaciano}. Differently from \cite{agostiniani_riemannianpenroseinequalitynonlinear_2022,xia_newmonotonicitycapacitaryfunctions_2024}, we do not exploit the quasi-monotonicity along the solution of the $\varepsilon$-regularized equation (see \cref{eq:eps_moser_potential} below).

In the proof of \cref{main}, $\varepsilon$-regularized approximations are employed exclusively for their original purpose of proving regularity. Indeed, this family of functions nicely converges to $w_p$ as $\varepsilon \to 0^+$ and almost all of their level sets are smooth with a uniform $L^2$ bound for their second fundamental form. These properties imply that almost every $\partial \Omega^{\sml{p}}_t$ admits a notion of second fundamental form. In other words, almost every level set is a curvature varifold. For the weak IMCF, such property already appears in some form in \cite{huisken_inversemeancurvatureflow_2001}. Here, the authors exploit the regularity of the level sets $\partial \Omega^{\sml{1}}_t$ of a weak IMCF to define (weak) mean curvature and second fundamental form. Indeed, $\partial \Omega^{\sml{1}}_t$ is of class $\CS^{1,\beta}$ (in dimension up to $7$) although the solution to \cref{1-Laplaciano} is merely Lipschitz. For $p>1$, solutions to \cref{p-Laplaciano} are known to be of class $\CS^{1,\beta}_{\loc}$ and smooth out of the critical set. Although the function has this additional regularity, it is not known whether this guarantees the $\CS^{1}$ regularity of its level sets. The only exception is the case $p=2$ where almost every level set is smooth by Sard's theorem. For $p \in (1,2)$, $\partial \Omega^{\sml{p}}_t$ is smooth outside the critical set, which is only $\Hff^{n-1}$-negligible for almost every $t$. This information is still not enough to globally define the notions of either the second fundamental or the mean curvature of the level sets. In \cite{agostiniani_minkowskiinequalitiesnonlinearpotential_2022,benatti_minkowskiinequalitycompleteriemannian_2022,benatti_nonlinearisocapacitaryconceptsmass_2023a,xia_newmonotonicitycapacitaryfunctions_2024}, the lack of these notions does not pose a real obstacle. The mean curvature (and the second fundamental form) is replaced with the analytic expression that one can derive from the equation, which is
\begin{equation}\label{eq:curvatura_p_pot}
    \H =  \abs{\nabla w_p } - (p-1) \frac{ \ip{\nabla \abs{\nabla w_p } |\nabla w_p }}{\abs{\nabla w_p}^2 }
\end{equation}
at every point where $w_p$ is smooth. As a consequence of our result, \cref{eq:curvatura_p_pot} is an identity between two \textit{a priori} distinct quantities rather than a definition: the generalized mean curvature of the varifold $\partial \Omega^{\sml{p}}_t$ is almost everywhere equivalent to the expression on the right-hand side of \cref{eq:curvatura_p_pot}. Through this lens, the expressions of $\MF_p$ in \cref{funz-monotoni} and $\SQ_p$ in \cref{Qdef-intro} are to be understood in this new, yet equivalent, geometric sense.

This point of view is not merely superficial. The dual perspective in \cref{eq:curvatura_p_pot} helps us in studying the behavior of \cref{derivata_F_p} in the limit as $p\to 1^+$, which is how we ultimately prove \cref{thm:monotonicty_formula_IMCF}. Following this path clearly requires a deeper analysis of the relation between the solutions to \cref{p-Laplaciano} and \cref{1-Laplaciano}, as they shape the family of quantities in \cref{funz-monotoni}. Indeed, only the local uniform convergence $w_p \to w_1$ is addressed in the literature and even the limit behavior of the gradient in the expression of $\MF_p$ is unknown. To make matters worse, $\SQ_p$ also contains terms with second-order derivatives of $w_p$. Monotonicity formulas \cref{derivata_F_p} help us to strengthen the convergence of $w_p\to w_1$ as $p\to 1^+$. Indeed, the boundedness of $\MF_p$ implies that all terms appearing in the expression of $\SQ_p$ are controlled on almost every level set of $w_p$. Consequently, since the traceless second fundamental form appears in the expression \cref{Qdef-intro}, its $L^2$ norm is controlled on the same level sets. Moreover, $(\H- \abs{\nabla w_p})^2$ is divided by $(p-1)$ in the last term of 
\cref{Qdef-intro}. To preserve the boundedness for $p$ close to $1$, the expression  $(\H- \abs{\nabla w_p})^2$ must vanish, hence the mean curvature behaves like the gradient, when $p\to 1^+$. All in all, the $L^2$ norm of the second fundamental form is controlled. By compactness, $\partial \Omega^{\sml{p}}_t$ converges to $\partial \Omega^{\sml{1}}_t$ in the sense of varifolds. In particular, the mean curvature $\vec{\H}^{\sml{p}}$ of $\partial \Omega^{\sml{p}}$, thus the gradient $\nabla w_p$, converges to the mean curvature $\vec{\H}^{\sml{1}}= \nabla w_1$ of $\partial \Omega^{\sml{1}}_t$. Summing up, we obtain the following second main result.

\begin{theorem}[Improved convergence]\label{main-convergences}
    Let $(M,g)$ be a noncompact Riemannian manifold of dimension $n\ge 3$ and $\Omega \subseteq M$ be a compact subset with $\CS^{1,1}$-boundary. Let $w_1$ be a solution to \cref{1-Laplaciano} and let $\Omega^{\sml{1}}_t \coloneqq \set{w_1 \leq t}$. Assume there exists $\overline{T}>0$ such that $\Omega^{\sml{1}}_t \Subset M$ for every $t \in [0,\overline{T})$. Let $D$ be an open precompact subset with smooth boundary containing $\Omega$ such that $\sup_{\partial D} w_1 <\overline{T}$. Then, there exists a family $(\Phi_p)_{p>1}$, with $\Phi_p\in \Lip(\overline D)$, such that if $w_p$ denotes the solution of  
    \begin{equation}\label{p-Laplaciano-inD} 
        \begin{cases} 
           \Delta_p w_p &=& \abs{\nabla w_p}^p & \text{on $D\smallsetminus \Omega$,}\\ 
            w_p &=& 0 & \text{on $\partial\Omega$,}\\ 
            w_p &=& \Phi_p & \text{on $\partial D$,} 
        \end{cases} 
    \end{equation} 
    the following properties hold as $p\to 1^+$: 
    \begin{enumerate}
       \item the solutions $w_p$ converge to $w_1$ uniformly on $\overline{D\smallsetminus \Omega}$;
       \item\label{stronggradients} their gradients $\nabla w_p$ strongly converge to $\nabla w_1$ in $L^q(D\smallsetminus \Omega)$, for $q<+\infty$;
       \item for any subsequence of $p$'s converging to $1^+$ there exists a subsequence $(p_k)_{k \in \N}$ such that $\partial \Omega^{\sml{p_k}}_t$ converges to $\partial \Omega^{\sml{1}}_t$ in the sense of varifolds and the corresponding generalized second fundamental forms weakly converge to that of $\partial \Omega^{\sml{1}}_t$ (in the sense of \eqref{eq:ConvDeboleSFFmanifoldgenerale}) as $k\to+\infty$, for almost every $t\in[0, \inf_{\partial D} w_1)$.
    \end{enumerate}
\end{theorem}

Limited to the convergence aspect, the statement above enhances all the results in the literature \cite{moser_inversemeancurvatureflow_2007,moser_inversemeancurvatureflow_2008,kotschwar_localgradientestimatesharmonic_2009,mari_flowlaplaceapproximationnew_2022}. The convergence of the gradients cannot improve to $L^\infty$ since $w_p$ is of class $\CS^{1,\beta}$ while $w_1$ is just Lipschitz in general.

\cref{main-convergences} comes with all the tools required to send $p\to 1^+$ in \cref{derivata_F_p}. The strong convergence of the gradient and the weak convergence of the mean curvature imply that $\MF_p(t)$ converges to $\MF_1(t)$. The $L^\alpha$ norm of the mean curvature is lower semicontinuous yielding $\liminf_p \MF_p(0)\leq \MF_1(0) $. As far as $\SQ_p$ is concerned, the $L^2$ norm of the traceless second fundamental form is lower semicontinuous by the convergence of curvature varifolds. The lower semicontinuity of the tangential gradient of $\abs{\nabla w_p}$ requires additional work, built again on varifold convergence and the strong convergence of the gradients.

\medskip

As mentioned above, \cref{main} unifies a rich family of monotonicity formulas in the literature and their geometric consequences. In Riemannian manifolds with nonnegative Ricci curvature, one can derive the Willmore-type inequality from \cite[Theorem 1.1]{agostiniani_sharpgeometricinequalitiesclosed_2020}, by setting $\alpha = n-p$, and the Minkowski-type inequality from \cite[Theorem 1.1]{benatti_minkowskiinequalitycompleteriemannian_2022}, taking $p=\alpha=1$. In the latter case, we also improve the rigidity statement \cite[Theorem 1.2]{benatti_minkowskiinequalitycompleteriemannian_2022} by removing some of the assumptions, which are deduced as a consequence of the equality instead. Such an improvement is possible since the Minkowski-type inequality directly follows from its monotonicity formula rather than through approximation by its counterparts for $p>1$. 

In Riemannian $3$-manifolds with nonnegative scalar curvature, \cref{main} also implies Huisken--Ilmanen's Geroch monotonicity formula \cite{huisken_inversemeancurvatureflow_2001} and its nonlinear potential theoretic generalizations \cite{agostiniani_riemannianpenroseinequalitynonlinear_2022,xia_newmonotonicitycapacitaryfunctions_2024}. Indeed, these monotone quantities are obtained by combining $\MF_p$ for $\alpha =3-p$ with an exponentially growing term. The Ricci curvature in the expression of $\MF_p$ can be replaced using the Gauss equation
\begin{equation}\label{gauss}
    \sca^\top = \sca - 2 \Ric(\nu,\nu) + \H^2 - \abs{\h}^2.
\end{equation}
Here, $\sca$ and $\sca^\top$ are the scalar curvatures of $M$ and $\partial \Omega^{\sml{p}}_t$ respectively. Almost every $\partial \Omega^{\sml{p}}_t$ admits a square integrable second fundamental form. Hence, the integral of either side in \cref{gauss} over these level sets is well-defined. To conclude the proof of the monotonicity, we now have to apply the Gauss--Bonnet theorem. For $p=2$, level sets are smooth, so this theorem applies. Gauss--Bonnet theorem also holds for the level sets of the weak IMCF (see \cite[Weak Gauss--Bonnet Formula 5.4]{huisken_inversemeancurvatureflow_2001}). For all other values of $p\in (1,2)$, one could define the Euler characteristic of the level sets as the integral of the induced scalar curvature (up to a constant). However, such a definition does not encompass the topological properties of the level sets and does not lead to the desired monotonicity. To overcome this issue, we prove the following Gauss--Bonnet-type theorem.

\begin{theorem}[Gauss--Bonnet-type theorem]\label{GB}
    Let $(M,g)$ be a noncompact $3$-dimensional Riemannian manifold and $\Omega \subseteq M$ be a compact subset with $\CS^{1,1}$ boundary. Let $w_p$ be a solution to \cref{p-Laplaciano} for $p \in [1,2]$ and let $\Omega^{\sml{p}}_t \coloneqq \set{w_p \leq t}$. Assume there exists $\overline{T}>0$ such that $\Omega^{\sml{p}}_t \Subset M$ for every $t \in [0,\overline{T})$. Then,
    \begin{equation}
        \int_{\partial \Omega^{\sml{p}}_t} \sca^{\top} \dif \Hff^{n-1} \in 8 \pi \Z
    \end{equation}
    for almost every $t \in [0,\overline{T})$.    
\end{theorem}

In the previous statement, the set $8 \pi \Z$ represents the possible values allowed by the classical Gauss--Bonnet theorem if $\partial \Omega^{\sml{p}}_t$ were smooth.
We get \cref{GB} by approximation with smooth surfaces. Suitable candidates are, as always, the $\varepsilon$-regularized level sets. The main difficulty here is to show the continuity of the integrals of $\H^2$ and $\abs{\h}^2$ in the limit as $\varepsilon \to 0^+$. Varifold convergence only implies the lower semicontinuity of these quantities. However, the equality in \cref{thm:monotonicity-NPT} prevents the loss of energy during the process.

In conclusion, we believe that one of the merits of this paper is to give a rigorous meaning to formal computations one performs on \cref{p-Laplaciano} so that one can avoid the repetition of the $\varepsilon$-regularization process in future uses.

\subsection{Structure of the paper}
In \cref{sec:preliminaries} we introduce the notation we are using throughout the paper and we revise the known regularity theory of solutions to \cref{p-Laplaciano}. We pass then to the weak inverse mean curvature flow and we prove a local approximation result of IMCF in terms of $p$-harmonic functions (\cref{thm:local-approximation}). In \cref{sec:monotonocity_nonlinear_potential}, after collecting all essential computations, we dive into the improved regularity for the level sets of the solutions to \cref{p-Laplaciano}. Subsequently, we carry on our analysis on the monotone quantities $\MF_p$. The section culminates in the proof of \cref{thm:monotonicity-NPT}.  \cref{sec:IMCF} is devoted to the IMCF counterpart of \cref{sec:monotonocity_nonlinear_potential}.
\cref{sec:geometric_consequences} contains the geometric consequences of the previous sections. First, we prove the Gauss--Bonnet type \cref{GB}. Then, we prove \cref{thm:minkowski} that covers at the same time Willmore and Minkowski inequalities with a new rigidity result. We conclude with the Geroch monotonicity formula. 
Finally, the paper includes \cref{sec:AppendiceVarifolds}, which systematically gathers various results on curvature varifolds in Riemannian manifolds and collects some technical results on the convergence of level sets of sequences of functions.

\subsection{Acknowledgments} The authors warmly thank Mattia Fogagnolo for many stimulating discussions on the topic of this paper. 
The authors are also grateful to the anonymous referee for the valuable comments and suggestions provided in their report.

The authors are members of INdAM - GNAMPA. Luca Benatti and Alessandra Pluda are partially supported by the BIHO Project ``NEWS - NEtWorks and Surfaces evolving by curvature'' and by the MUR Excellence Department Project awarded to the Department of Mathematics of University of Pisa. Luca Benatti is partially supported by the PRIN Project 2022PJ9EFL ``Geometric Measure Theory: Structure of Singular Measures, Regularity Theory and Applications in the Calculus of Variations''. Alessandra Pluda is partially supported by the PRIN Project 2022R537CS ``$\rm{NO}^3$ - Nodal Optimization, NOnlinear elliptic equations, NOnlocal geometric problems, with a focus on regularity''. Marco Pozzetta is partially supported by the PRIN Project 2022E9CF89 ``GEPSO - Geometric Evolution Problems and Shape Optimization'' -- PNRR Italia Domani, funded by EU Program NextGenerationEU.

\section{Preliminaries and notation}\label{sec:preliminaries}

This section sets the notation we will use throughout the paper. We also recall some results already known in the literature about the nonlinear potential theory and the inverse mean curvature flow. 
In what follows, $(M,g)$ will be a connected noncompact smooth Riemannian manifold of dimension $n\geq 3$. Whenever not specified, we assume that $M$ has no boundary. We stress that $(M,g)$ is not assumed to be complete.

\smallskip

If $E \subset M$ is a subset with $\CS^1$ boundary, we denote by $\nu$ the outward pointing unit normal along $\partial E$. If $E$ has $\CS^2$ boundary, we denote by $\vec{\H}$ the mean curvature vector of $\partial E$ and we set $\H\coloneqq- \ip{\vec{\H}| \nu }$, and we denote by $\h$ (resp. $\mathring{\h}$) the second fundamental form (resp. traceless second fundamental form) of $\partial E$. We will avoid subscripts and superscripts when geometric quantities like $\nu,\H,\h,\mathring{\h}$ are clearly referred to some given submanifold. In particular, when a geometric quantity appears without indices in an integral along the boundary of a set $\partial E$, it is understood that such quantity is referred to $\partial E$. We will also make use of the theory of varifolds on Riemannian manifolds. The definitions and the results we will use are recalled in \cref{sec:AppendiceVarifolds}. A notation completely analogous to the one introduced before is adopted for the weak notions of mean curvature and (traceless) second fundamental form defined on varifolds, see \cref{sec:AppendiceVarifolds}. Finally, we shall denote by $\Hff^d$ the $d$-dimensional Hausdorff measure induced by the Riemannian distance on $(M,g)$, for any $d\ge0$, and we will simply denote $\abs{E}\coloneqq \Hff^n(E)$ for any Borel set $E\subset M$ and $\abs{\partial E} \coloneqq \Hff^{n-1}(\partial E)$ for any set $E\subset M$ such that $\partial E$ is countably $(n-1)$-rectifiable. Given a set $E\subset M$ with $\CS^1$ boundary we will denote by
\begin{equation}
\begin{split}
    &\nabla^\perp f = \ip{\nabla f | \nu } \,\nu,\\
    &\nabla^\top f = \nabla f - \nabla^\perp f.  
\end{split}
\end{equation}
the projection of the gradient on the normal and the tangent space to $\partial E$, respectively.

\smallskip

We shall fix a compact subset $\Omega$ of $M$ with $\CS^{1,\beta}$-boundary, for some $\beta>0$. The letter $D$ will denote a precompact open subset with smooth boundary containing $\Omega$. For any compact set $K\subseteq D \smallsetminus \Omega^\circ$, for any function $u:K\to\R$ with $u \in \CS^1(K)$ we denote
\begin{equation}
\Crit(u) \coloneqq \Omega^\circ \cup \set{x \in K  : \abs{\nabla u}(x)=0}.
\end{equation}
We use the same notation for functions that are only Sobolev, with the obvious modification. If $u$ is a Lipschitz function and has compact sublevels in $D$, we shall denote $\Omega_t\coloneqq \set{u \leq t } \cup \Omega$ and $U_t\coloneqq \set{u <t} \cup \Omega$, hence $\Omega \subseteq U_t \subseteq  \Omega_t $. When referring to the unit normal to the level sets of a function $u$, we will also employ the symbol $\nu$ for the vector field equal to $\nabla u/\abs{\nabla u}$ almost everywhere outside $\Crit(u)$ and equal to $0$ almost everywhere on $\Crit(u)$. In case of possible confusion, suitable superscripts will be added to the symbols $\Omega_t, \nu$ in order to refer to some given function.

\subsection{Nonlinear potential theory}
Let $\Phi \in \Lip(\overline{D})$ with $\Phi>0$ on $\partial D$. For $\varepsilon>0$ and $1<p\leq 2 $, let $w_p^\varepsilon$ be the solution to the following boundary value problem
\begin{equation}\label{eq:eps_moser_potential}
    \begin{cases}
        \Delta_{p}^\varepsilon \ee^{-\frac{w_p^\varepsilon}{p-1}}& =& 0 &\text{on $D \smallsetminus \Omega $,}\\
     w_p^\varepsilon&=&0& \text{on $\partial \Omega$,}\\
     w_p^\varepsilon&=&\Phi & \text{on $\partial D$,}
    \end{cases}
\end{equation}
 where 
\begin{align}
    \Delta_{p}^\varepsilon f \coloneqq \div\left( \abs{ \nabla f}_\varepsilon^{p-2} \nabla f\right) \text{ and } {\abs{}}_\varepsilon\coloneqq \sqrt{ \abs{}^2 +\varepsilon^2}.
\end{align}
The PDE in \cref{eq:eps_moser_potential} is a second-order elliptic equation, hence the solution $w^\varepsilon_p$ is smooth in $D\smallsetminus \Omega$ by classical elliptic regularity theory. We now recall some regularity estimates that are stable with respect to $\varepsilon\to0^+$. The following proposition collects results from \cite{dibenedetto_alphalocalregularityweak_1983,tolksdorf_dirichletproblemquasilinearequations_1983,lieberman_boundaryregularitysolutionsdegenerate_1988,lou_singularsetslocalsolutions_2008} (see also \cite{benatti_monotonicityformulasnonlinearpotential_2022} for a dissertation on the topic).

\begin{proposition}\label{prop:StimeBaseWepsilonP}
    Let $(M,g)$ be a Riemannian manifold and let $\Omega \subset M$ be compact with $\CS^{1,\beta}$-boundary, for some $\beta>0$. Let $w^\varepsilon_p$ be the solution to the problem \cref{eq:eps_moser_potential} for some $\Phi \in \Lip(\overline{D})$, where $D$ is a precompact open subset with smooth boundary containing $\Omega$. Then
    \begin{enumerate}
        \item for any compact $K\subset D \smallsetminus \Omega^\circ$ there exists a constant $\kst>0$ independent of $\varepsilon$ such that $\norm{w^\varepsilon_p}_{\CS^{1,\beta'}(K)} \leq \kst$ for some $\beta'\le   \beta$,
        \item $\abs{\nabla w^\varepsilon_p} \in W^{1,2}_{\loc}( D \smallsetminus \Omega)$ and for any compact $K'\subset D \smallsetminus \Omega$ there exists a constant $\kst'>0$ independent of $\varepsilon$ such that $\norm{\abs{\nabla w^\varepsilon_p} }_{W^{1,2}(K')} \leq \kst'$.
    \end{enumerate}
\end{proposition}
For $w^\varepsilon_p$, we have that $\Omega_t^{\sml\varepsilon} \subset D$ for any $t \in[0,\inf_{\partial D} \Phi)$. Moreover, since $\partial \Omega$ is $\CS^{1,\beta}$, by Hopf maximum principle, we have that $\partial \Omega \cap \Crit(w^\varepsilon_p) = \varnothing$. 

\smallskip

We now pass to the analysis of the level set flow of $w^\varepsilon_p$. By the classical Sard theorem, $\partial \Omega^{\sml\varepsilon}_t$ is a smooth hypersurface for almost every $t \in [0,\inf_{\partial D} \Phi)$. We also have that $\nu = \nabla w^\varepsilon_p/ \abs{\nabla w^\varepsilon_p}$ at any point of $\partial \Omega^{\sml{\varepsilon}}_t \smallsetminus \Crit(w^\varepsilon_p)$. Moreover, for every $\varepsilon>0$ we denote
\begin{equation}\label{eq:DefinizioneThetaEpsilon}
    \Err_\varepsilon \coloneqq \frac{\varepsilon^2}{\displaystyle\abs{\nabla \ee^{-\frac{w^\varepsilon_p}{(p-1)}}}^2_\varepsilon}=1- \frac{\displaystyle\abs{\nabla \ee^{-\frac{w^\varepsilon_p}{(p-1)}}}^2}{\displaystyle\abs{\nabla \ee^{-\frac{w^\varepsilon_p}{(p-1)}}}^2_\varepsilon}\in [0,1].
\end{equation}
Exploiting the equation in \cref{eq:eps_moser_potential}, a direct computation shows that the mean curvature $\H$ of $\partial \Omega^{\sml\varepsilon}_t$ can be written as
\begin{equation}\label{eq:MeanCurvatureLevelWepsP}
    \H = \left( \abs{\nabla w^\varepsilon_p } - (p-1) \frac{ \ip{\nabla \abs{\nabla w^\varepsilon_p }|\nabla w^\varepsilon_p }}{\abs{\nabla w^\varepsilon_p}^2 }\right) \left(1+ \frac{2-p}{p-1} \Err_\varepsilon\right).
\end{equation}

\bigskip

We will be interested in the limit of the functions $w^\varepsilon_p$ as $\varepsilon\to 0^+$. We denote by $w_p$ the solution to the problem
\begin{equation}\label{eq:moser_potential}
    \begin{cases}
        \Delta_{p}w_p& =& \abs{\nabla w_p }^p &\text{on $ D \smallsetminus \Omega $,}\\
     w_p&=&0& \text{on $\partial \Omega$,}\\
     w_p&=&\Phi & \text{on $\partial D$,}
    \end{cases}
\end{equation}
where $\Delta_p f \coloneqq \div\left( \abs{ \nabla f}^{p-2} \nabla f\right)$ is the usual $p$-Laplace operator.

As a consequence of \cref{prop:StimeBaseWepsilonP} and classical regularity estimates, we have the following convergence and regularity result.
\begin{proposition}\label{prop:RegolaritaWp}
    Let $(M,g)$ be a Riemannian manifold and let $\Omega \subset M$ be compact with $\CS^{1,\beta}$-boundary, for some $\beta>0$. Let $w^\varepsilon_p$ (resp. $w_p$) be the solution to the problem \cref{eq:eps_moser_potential} (resp. \cref{eq:moser_potential}) for some $\Phi \in \Lip(\overline{D})$, where $D$ is an open precompact subset with smooth boundary containing $\Omega$. Then $w^\varepsilon_p$ converges to $w_p$ in $\CS^1_{\loc}(D \smallsetminus \Omega^\circ)$ and $\abs{\nabla w^\varepsilon_p }\rightharpoonup \abs{\nabla w_p}$ weakly in $W^{1,2}_{\loc} (D \smallsetminus \Omega)$. In particular, $w_p \in \CS^{1,\beta'}_{\loc}(D \smallsetminus \Omega^\circ)$ for some $\beta'\leq  \beta$ and $\abs{\nabla w_p}^{p-1} \in W^{1,2}_{\loc}(D\smallsetminus \Omega)$. Moreover, the function $w_p$ is smooth in $D \smallsetminus \Crit(w_p)$.
\end{proposition}

We remark that $\Crit(w_p)$ in \cref{prop:RegolaritaWp} may have positive $\Hff^n$-measure. Differently from $w^\varepsilon_p$ we can not apply Sard's theorem to $w_p$. Hence, the set of $t$'s in $[0,\inf_{\partial D} \Phi)$ such that $\partial \Omega^{\sml{p}}_t$ is not smooth could be of positive measure. Nevertheless, \cref{rmk:weakSard} ensures that for almost every $t \in [0,\inf_{\partial D}\Phi)$ it holds that $\partial \Omega^{\sml{p}}_t \cap \Crit(w_p)$ is $\Hff^{n-1}$-negligible. Also, the set $\partial \Omega^{\sml{p}}_t$ is smooth outside $\Crit(w_p)$. To be consistent with \cref{eq:DefinizioneThetaEpsilon} we impose $\Err_0\equiv 0$. Exploiting the equation \cref{eq:moser_potential}, a direct computation shows that the mean curvature $\H$ of $\partial \Omega^{\sml{p}}_t$ at regular points can be written as
\begin{equation}\label{eq:MeanCurvatureLevelWP}
    \H =  \abs{\nabla w_p } - (p-1) \frac{ \ip{\nabla \abs{\nabla w_p } |\nabla w_p }}{\abs{\nabla w_p}^2 },
\end{equation}
which is \cref{eq:MeanCurvatureLevelWepsP} for $\varepsilon=0$.

\smallskip
Solutions to \cref{eq:moser_potential} naturally come with the notion of $p$-capacity. For any compact $K \subseteq D$ and for $p>1$, we define the $p$-capacity relative to $D$, as
\begin{equation}\label{eq:capacitydefinition}
\ncapa_p(K, D) = \inf\set{\left(\frac{p-1}{n-p}\right)^{p-1} \frac{1}{\abs{\S^{n-1}}}\int_{D \smallsetminus K}\abs{\nabla \psi}^p \dif \Hff^n\st \psi \in \CS^\infty_c(D), \psi \geq 1 \text{ on }K}.
\end{equation}

We have the following classic result whose proof can be found in \cite[Lemma 3.8]{holopainen_nonlinearpotentialtheoryquasiregular_1990}.

\begin{lemma}\label{lem:p-cap_levelset}
Let $p>1$ and let $w_p$ be a solution to \cref{eq:moser_potential}. For every $\tau,t,T\in [0,\inf_{\partial D} \Phi)$ with $t<T$ we have
\begin{equation}
    \left(\ee^{-\frac{t}{p-1}}-\ee^{-\frac{T}{p-1}}\right)^{p-1}\ncapa_p(\Omega_t^{\sml{p}}, U_T^{\sml{p}}) = \frac{\ee^{-\tau}}{\abs{\S^{n-1}}} \int_{\partial \Omega_\tau^{\sml{p}}} \left( \frac{\abs{\nabla w_p}}{n-p}\right)^{p-1} \dif \Hff^{n-1},
\end{equation}
for every $\tau \in (t,T)$.
\end{lemma}

\subsection{Inverse mean curvature flow}
    We will assume the existence of a locally Lipschitz solution $w_1$ to
    \begin{equation}\label{eq:IMCF}
        \begin{cases}
            \div\left( \dfrac{\nabla w_1}{\abs{\nabla w_1}}\right)& =& \abs{\nabla w_1 } &\text{on $ M \smallsetminus \Omega $,}\\
         w_1&=&0& \text{on $\partial \Omega$,}
        \end{cases}
    \end{equation}
    in the sense of \cite{huisken_inversemeancurvatureflow_2001}. We also assume there exists $\overline{T}>0$ such that $\Omega^{\sml{1}}_t \Subset M$ for every $t \in [0,\overline{T})$.  Although we can have many locally Lipschitz solutions to \cref{eq:IMCF}, we recall that they coincide as soon as level sets remain precompact (see \cite[Uniqueness Theorem 2.2]{huisken_inversemeancurvatureflow_2001}). 
    
    Since $w_1$ is locally Lipschitz, \cref{rmk:weakSard} implies that $\partial \Omega^{\sml{1}}_t\cap \Crit(w_1)$ is $\Hff^{n-1}$-negligible for almost every $t \in [0,\overline{T})$. However, differently from $w_p$, we cannot say that $\partial \Omega_t^{\sml{1}}$ is smooth outside the critical set. A regularity result for the level sets $\partial \Omega^{\sml{1}}_t$ still holds. We will exploit the properties in the following proposition in which we combine various results that appeared in the literature (see for example \cite[Regularity Theorem 1.3]{huisken_inversemeancurvatureflow_2001}, \cite[Lemma 2.8]{xu_isoperimetrypropernessweakinverse_2023} and the references therein).
    \begin{proposition}\label{prop:regularityLevelsetIMCF}
          Let $(M,g)$ be a noncompact Riemannian manifold and let $\Omega \subset M$ be compact with $\CS^{1,\beta}$-boundary, for some $\beta>0$. Let $w_1$ be a locally Lipschitz solution to \cref{eq:IMCF}. Assume there exists $\overline{T}>0$ such that $\Omega^{\sml{1}}_t \Subset M$ for every $t \in [0,\overline{T})$. Then,
          \begin{enumerate}
              \item for each $t \in [0,\overline{T})$, $\partial \Omega^{\sml{1}}_t$ is locally a $\CS^{1,\alpha}$, for $\alpha <1/2$, hypersurface out of a closed set $S_t$ of Hausdorff dimension at most $n-8$;
              \item for each $t \in [0,\overline{T})$, the set $\partial \Omega^{\sml{1}}_t\smallsetminus(\partial U^{\sml{1}}_t \cup S_t)$ is a smooth minimal hypersurface;
              \item \label{item:convergenceC1bIMCF}for every $t \in [0,\overline{T})$, $\partial U^{\sml{1}}_s$ converges in $\CS^{1,\beta'}$ to $\partial\Omega^{\sml{1}}_t$ around each point of $\partial \Omega^{\sml{1}}_t\smallsetminus S_t$ as $s \to t^+$, for $\beta'<\beta$;
              \item for every $t \in (0,\overline{T})$, $\partial U^{\sml{1}}_s$ converges in $\CS^{1,\beta'}$ to $\partial U^{\sml{1}}_t$ around each point of $\partial \Omega^{\sml{1}}_t\smallsetminus S_t$ as $s \to t^-$, for $\beta'<\beta$.
          \end{enumerate}
    \end{proposition}
    For the previous regularity result, the reader is referred to \cite[Proposition 2]{tamanini_boundariesCaccioppoliSets_1982} (see also \cite{miranda_FrontiereMinimaliOstacoli_1971}). The convergence assertions can be obtained as in \cite[Proposition 1]{tamanini_boundariesCaccioppoliSets_1982} by replacing \cite[(3.4)]{tamanini_boundariesCaccioppoliSets_1982} with the minimality property enjoyed by $U^{\sml{1}}_{s}$. Indeed, since the weak IMCF $w_1$ is locally Lipschitz, the estimate \cite[($*$)]{tamanini_boundariesCaccioppoliSets_1982} can be proved at a uniform scale for every $\partial U^{\sml{1}}_{s}$ with a uniform constant independent of $s$ around any $x \in \partial\Omega^{\sml{1}}_t\smallsetminus S_t$.

    The $\CS^{1,\beta}$-regularity is not sufficient in general to introduce a notion of mean curvature $\H$, even in a weak sense. However, this is possible for $\partial \Omega^{\sml{1}}_t$ since $w_1$ is a minimizer of a functional and its gradient is locally bounded. The mean curvature of the level sets of $w_1$ was already introduced in the work of Huisken and Ilmanen (see \cref{rmk:equivalent_formulation_monotonicity_IMCF} below). Assuming the function $w_1$ is smooth with nonvanishing gradient, the equation in \cref{eq:IMCF} gives
    \begin{equation}\label{eq:MeanCurvatureLevelW1}
        \H = \abs{\nabla w_1}.
    \end{equation}
    One can recognize \cref{eq:MeanCurvatureLevelWP} for $p=1$. We will show that this identity holds on $\partial \Omega^{\sml{1}}_t$ out of a $\Hff^{n-1}$-negligible set for almost every $t \in [0,\overline{T})$. 

    Given $\Omega\subseteq M$ we will denote by $\Omega^*$ the strictly outward minimizing hull of $\Omega$ in $M$, which is the maximal volume solution to the least area problem with obstacle $\Omega$. A set $\Omega$ is strictly outward minimizing provided $\Omega=\Omega^*$. The reader is referred to \cite{fogagnolo_minimisinghullscapacityisoperimetric_2022} for a complete dissertation on this topic (see also \cite[Proposition 2]{tamanini_boundariesCaccioppoliSets_1982}). 
    \begin{proposition}\label{prop:outwardminimizinghull}
          Let $(M,g)$ be a noncompact Riemannian manifold and let $\Omega \subset M$ be compact with $\CS^{1,\beta}$-boundary, for some $\beta>0$. Let $w_1$ be the locally Lipschitz solution to \cref{eq:IMCF}. Assume there exists $\overline{T}>0$ such that $\Omega^{\sml{1}}_t \Subset M$ for every $t \in [0,\overline{T})$. Then
          \begin{enumerate}
              \item  $\Omega^{\sml{1}}_t$ is strictly outward minimizing for every $t \in [0,\overline{T})$;
              \item\label{item:hull0}  for $t=0$ we have  that $\Omega^{\sml{1}}_0 = \Omega^*$ and the singular set $S_0\subset \partial \Omega^{\sml{1}}_0$ is disjoint from $\partial \Omega$;
              \item for any $t \in [0,\overline{T})$, $\abs{\partial \Omega^{\sml{1}}_t} = \abs{\partial U^{\sml{1}}_t}$; $\abs{\partial \Omega^{\sml{1}}_0} \leq \abs{\partial \Omega}$ and the equality holds provided $\partial \Omega$ is outward minimizing.
          \end{enumerate}
    \end{proposition}

    The $p$-capacity \cref{eq:capacitydefinition} can be extended to $p=1$, with the obvious convention that $(p-1)^{p-1}=1$ for $p=1$. The $1$-capacity is intimately connected with the area minimization problem. Indeed,
    \begin{equation}
        \abs{\S^{n-1}}\ncapa_1(K,D) = \inf\set{\abs{\partial E} \st K\subseteq E \subseteq D, E\text{ with smooth boundary}}
    \end{equation}
    by \cite[§2.2.5]{mazya_sobolevspacesapplicationselliptic_2011}. Observe that if $w_1 >0$ on $\partial D$ then $\Omega^* = \Omega^{\sml{1}}_0 \subseteq D$. Then, for a set $\Omega$ with $\CS^{1,\beta}$ boundary we have
    \begin{equation}\label{eq:somh_and_1cap}
       \abs{\S^{n-1}}\ncapa_1(\Omega,D)= \abs{\partial \Omega^*},
    \end{equation}
    for every $D$ containing $\Omega^*$. Indeed, $\Omega^*$ minimizes the perimeter among a larger class of sets, hence one inequality is trivial. The other one follows by \cite[Theorem 2.23]{fogagnolo_minimisinghullscapacityisoperimetric_2022}, taking an approximation of $\Omega^*$ by sets with smooth boundary entirely contained in $D$. In view of \cref{eq:somh_and_1cap}, the following result follows (see \cite[Exponential Growth Lemma 1.6]{huisken_inversemeancurvatureflow_2001}).
    \begin{lemma}\label{lem:exp_areagrowth}
        Let $w_1$ be a locally Lipschitz solution to \cref{eq:IMCF}. Assume there exists $\overline{T}>0$ such that $\Omega^{\sml{1}}_t \Subset M$ for every $t \in [0,\overline{T})$. Then, for every $t \in [0,\overline{T})$
        \begin{equation}
            \abs{\partial \Omega^{\sml{1}}_t} = \ee^t \abs{\partial \Omega^*}.
        \end{equation}
    \end{lemma}

    \subsection{Local approximation of IMCF}
    We want to show that $w_1$ can be locally uniformly approximated on a given open set $D$ by the solutions $(w_p)_{p\geq 1}$ to \cref{eq:moser_potential} as $p\to 1^+$. We firstly need a uniform gradient bound for $(w_p)_{p \geq 1}$.

    \begin{proposition}[Uniform gradient bound]\label{prop:uniform_gradient}
    Let $(M,g)$ be a Riemannian manifold and let $\Omega \subseteq M$ be a compact set with $\CS^{1,1}$-boundary. Let $D$ be an open precompact set with smooth boundary containing $\Omega$ and let $w_p$ be the solution to \cref{eq:moser_potential} for $p\in(1,2]$, for some $\Phi \in \Lip(\overline{D})$. Then for any compact set $K\subset D$ there exists a positive constant $\kst$ depending only on the boundary $\partial \Omega$, $n$, $K$ and a lower bound on the sectional curvature on some tubular neighborhood of $K$, such that
    \begin{equation}
        \norm{\nabla w_p}_{L^\infty(K\smallsetminus\Omega^\circ)} \leq \kst.
    \end{equation}
    \end{proposition}
    
    \begin{proof}
        Without loss of generality, we can assume that $\Omega^\circ\subset K$ and that $K$ has smooth boundary. We have $\ee^{-w_p/(p-1)}>0$ on $K\smallsetminus \Omega^\circ$. Hence \cite[Theorem 2.22]{mari_flowlaplaceapproximationnew_2022} applies, which yields
        \begin{equation}
            \sup_{K\smallsetminus\Omega^\circ} \abs{\nabla w_p} \leq \max\set{(n-1)\kappa, \limsup_{x\to \partial K \cup \partial \Omega} \abs{\nabla w_p}},
        \end{equation}
        where $\kappa$ is such that $\Ric \geq - (n-1)\kappa^2$ on some tubular neighborhood of $K$. By \cite[Proposition 3.3]{kotschwar_localgradientestimatesharmonic_2009}, since $\abs{\nabla w_p}$ is continuous up to the boundary $\partial \Omega$, then $\limsup_{x\to \partial \Omega} \abs{\nabla w_p}$ is bounded by a constant depending on $\Omega$ and $\kappa$. On the other hand, $\partial K$ is in the interior of $D$. Hence, we can use \cite[Theorem 1.1]{kotschwar_localgradientestimatesharmonic_2009} to bound $\norm{\abs{\nabla w_p}}_{L^\infty(\partial K)}$ with a constant with the same dependencies claimed in the statement.
    \end{proof}

    The following theorem is a local version of the Moser approximation \cite{moser_inversemeancurvatureflow_2007}. The statement says that solutions to \cref{eq:IMCF} can be locally approximated by solutions to \cref{eq:moser_potential} on some bounded domain. The idea of the proof below was suggested to the authors by Mattia Fogagnolo.

    \begin{theorem}[Local approximation of IMCF]\label{thm:local-approximation}
    Let $(M,g)$ be a noncompact Riemannian manifold. Let $\Omega \subseteq M$ be a compact subset with $\CS^{1,1}$-boundary. Let $w_1$ be a weak IMCF on $M\smallsetminus \Omega$ starting from $\Omega$. Assume there exists $\overline{T}>0$ such that $\Omega^{\sml{1}}_t \Subset M$ for every $t \in [0,\overline{T})$. Let $D$ be an open precompact subset with smooth boundary containing $\Omega$ such that $\sup_{\partial D} w_1 <\overline{T}$.  Then, there exists a family $(\Phi_p)_{p> 1}$ in $\Lip(\overline{D})$ such that the solutions $(w_p)_{p>1}$ to \cref{eq:moser_potential} with $\Phi =\Phi_p$ converge uniformly to $w_1$ on $\overline{D\smallsetminus \Omega}$ as $p \to 1^+$.
    \end{theorem}
    
    \begin{proof}
    Let $0<\tau<\overline{T}$ such that $\overline{D} \subseteq \Omega^{(1)}_\tau$. Let $K_\tau$ be an open precompact domain with smooth boundary which contains $\Omega^{\sml{1}}_\tau$. As observed in \cite[Theorem 3.1]{huisken_inversemeancurvatureflow_2001} (we refer to \cite[§4.1]{xu_isoperimetrypropernessweakinverse_2023} for the explicit construction), it is possible to attach to $K_\tau$ finitely many conical ends, thus generating a new Riemannian manifold formally given by $M_\tau\coloneqq K_\tau \cup (\partial K_\tau \times [0, +\infty))$ (understanding the obvious identifications) endowed with a smooth metric $g_\tau$ with the following properties: $g_\tau=g$ on $\overline{K_\tau}$; $g_\tau> g$ on $\partial K_\tau \times (0,\delta_\tau]$ for some $\delta_\tau>0$; $g_\tau$ is a cone metric on $\partial K_\tau \times[\delta_\tau, +\infty)$. One can check that the function $\overline{w}_1$ which is $\overline{w}_1 =\min\set{w_1, \tau}$ on $K_\tau\smallsetminus \Omega$ and $\overline{w}_1=\tau$ on $M_\tau \smallsetminus K_\tau$ is a subsolution to the IMCF on $M_\tau\smallsetminus \Omega$.

    Let $E_0 \coloneqq K_\tau \cup (\partial K_\tau \times [0,\delta_\tau])$. The function $v \coloneqq (n-1) \ln(r/\delta_\tau)$, where $r$ denotes the radial coordinate on $\partial K_\tau \times [\delta_\tau,+\infty)$, is the unique smooth IMCF with never vanishing gradient on $M_\tau\smallsetminus E_0$ with initial condition $E_0$. We set $v=0$ in $E_0$. For every $m\in \N$, let $E_m\coloneqq \set{ v\leq m}$, and for $p \in (1,2]$ let $w_{p,m}$ be the solution to the equation
        \begin{equation}\label{zzEq:AuxiliaryPB}
            \begin{cases}
                \Delta_{p}w_{p,m}& =& \abs{\nabla w_{p,m} }^p &\text{on $ E_m \smallsetminus \Omega $,}\\
             w_{p,m}&=&0& \text{on $\partial \Omega$,}\\
             w_{p,m}&=&m & \text{on $\partial E_m$.}
            \end{cases}
        \end{equation}     
    Denoting $v_m\coloneqq [(m-1)/m] v$, we have
        \begin{align}\label{eq:zzzcomparison_at_boundary}
            w_{p,m}\geq 0 = v_{m} \text{ on $\partial E_0$,} &&
            w_{p,m}=m \geq  v_{m} \text{ on $\partial E_{m}$.}
        \end{align}
    Moreover, on $E_m \smallsetminus \overline{E}_0$ we can compute
        \begin{align}
            \Delta_p v_m - \abs{\nabla v_m}^p &= \abs{\nabla v_m}^{p-1} \left( \Delta_1 v_m - \abs{\nabla v_m} - (p-1) \frac{\ip{\nabla \abs{\nabla v_m}| \nabla v_m}}{\abs{\nabla v_m}^2}\right)\\&\geq 
            \abs{\nabla v_m}^{p-1} \left( \frac{m}{m-1}\abs{\nabla v_m} - (p-1) \frac{\ip{\nabla \abs{\nabla v_m}| \nabla v_m}}{\abs{\nabla v_m}^2}\right).
        \end{align}
    Hence, for any $m\in\N$ there exists $p(m)>1$ such that $\Delta_p v_m- \abs{\nabla v_m}^p\geq0$ on $E_m \smallsetminus \overline{E}_0$, for any $p\leq p(m)$, i.e., $v_m$ is a subsolution for \cref{zzEq:AuxiliaryPB} for any $p\leq p(m)$ on $E_m \smallsetminus \overline{E}_0$. Without loss of generality, we can arrange that $p(m)$ is a sequence that decreases to $1$ as $m\to\infty$. For $p \in (1,2]$ let $m(p)=\sup\set{k\in \N \st p\leq p(k)}<+\infty$. In particular, $m(p)\to +\infty $ as $p\to 1$. Recalling \cref{eq:zzzcomparison_at_boundary}, by the comparison principle $w_{p,m(p)} \geq v_{m(p)}$ on $E_{m(p)}\smallsetminus E_0$. By \cref{prop:uniform_gradient}, there exists a (not relabeled) subsequence $p\to 1^+$ such that $u_p \coloneqq w_{p,m(p)}$ converges uniformly on any compact set. By \cite{moser_inversemeancurvatureflow_2007}, a local uniform limit $u_1$ of $u_p$ is a weak IMCF starting from $\Omega$. By construction $u_1\geq v$ on $M_\tau \smallsetminus E_0$, thus $u_1$ is the unique proper solution starting from $\Omega$ on $M_\tau$. Therefore, the whole family $u_p$ must converge to $u_1$ locally uniformly as $p\to1^+$.    
    Comparison principle yields $u_1 \geq \overline{w}_1$ on $M_\tau$ since the latter is a subsolution to the IMCF on $M_\tau \smallsetminus \Omega$. In particular, $u_1 \geq \min\set{w_1,\tau}$ on $K_\tau$ and $\set{u_1 \leq t} \Subset K_\tau$ for every $t < \tau$. Now $u_1$ and $\min{w_1, \tau}$ are weak IMCF in $K_\tau$ with the same initial data. Both have precompact subsets for $t < \tau$.
    By \cite[Uniqueness Theorem 2.2]{huisken_inversemeancurvatureflow_2001}, they must coincide on $K_\tau$. Since $\min\set{w_1, \tau}= w_1$ on $\Omega^{\sml{1}}_\tau$, $u_p \to w_1$ on $\overline{D\smallsetminus \Omega}$ uniformly as $p\to 1^+$. The statement follows by taking $\Phi_p$ as the restriction to $\overline{D\smallsetminus \Omega}$ of $u_p$ (and extending it to $0$ inside $\Omega$).
    \end{proof}
    
     Differently from \cite{moser_inversemeancurvatureflow_2007,moser_inversemeancurvatureflow_2008,kotschwar_localgradientestimatesharmonic_2009,mari_flowlaplaceapproximationnew_2022}, we do not prove the existence of a proper IMCF, but we assume it. This approach has the advantage of not requiring that $(M,g)$ is strongly $p$-nonparabolic, i.e. it admits a solution to \cref{eq:moser_potential} with $D=M$ and $\Phi = +\infty$. Indeed, in general, a strongly $1$-nonparabolic manifold can be $p$-parabolic for any $p>1$. In this particular case, it is not even possible to prove that the $p$-capacities of $\Omega$ in $M$ converge to $\abs{\partial \Omega^*}$ as $p\to1^+$, as proved under further geometric conditions on the manifold in \cite[Theorem 1.2]{fogagnolo_minimisinghullscapacityisoperimetric_2022}. Indeed, $\ncapa_p(\Omega, M)$ may be identically $0$ for every $p>1$, while $\abs{\partial \Omega^*} >0$. However, this result remains true if one considers the relative capacities of the sequence $(w_p)_{p>1}$ obtained in \cref{thm:local-approximation}.

    \begin{lemma}\label{lem:continuity_p-capacity}
        Under the same assumptions of \cref{thm:local-approximation}, if $\inf_{\partial D}w_1 >0$ we have
        \begin{equation}\label{eq:capacity_convergence}
            \lim_{p\to 1^+} \ncapa_p(\Omega, U^{\sml{p}}_T) =  \frac{\abs{\partial \Omega^*}}{\abs{\S^{n-1}}}
        \end{equation}
        for every $0<T < \inf_{\partial D} w_1$.
    \end{lemma}
    
    \begin{proof}
    Without loss of generality, we can assume that $D$ has smooth boundary.
        In view of \cref{eq:somh_and_1cap}, \cref{eq:capacity_convergence} follows from
        \begin{equation}\label{eq:zzzcontinuity_p-cap}
            \lim_{p \to 1^+}\ncapa_p(\Omega, U^{\sml{p}}_T) = \ncapa_1(\Omega, U^{\sml{1}}_T).
        \end{equation}
        Let $\psi \in \CS^\infty_c(U_T^{\sml{1}}) $ with $\psi \geq 1 $ on $\Omega$. Let $0< \varepsilon < \inf_{\partial D} w_1 -T$. Since $w_p\to w_1$ uniformly on $D$ by \cref{thm:local-approximation}, we then have $U^{\sml{1}}_T \subseteq U_{T+ \varepsilon}^{\sml{p}}$ for any $p$ sufficiently close to $1$. In particular,
        \begin{equation}
            \ncapa_p(\Omega, U^{\sml{p}}_{T+ \varepsilon})\leq \left(\frac{p-1}{n-p}\right)^{p-1} \frac{1}{\abs{\S^{n-1}}} \int_{D \smallsetminus\Omega} \abs{\nabla \psi}^p \dif \Hff^n.
        \end{equation}
        Taking the limit superior and recalling \cref{lem:p-cap_levelset}, we have
        \begin{equation}
            \limsup_{p \to 1^+} \ncapa_p(\Omega, U^{\sml{p}}_T) = \limsup_{p\to 1^+} \ncapa_p(\Omega, U^{\sml{p}}_{T+\varepsilon})\leq \frac{1}{\abs{\S^{n-1}}} \int_{D\smallsetminus \Omega} \abs{\nabla \psi} \dif \Hff^n.
        \end{equation}
        Taking the limit inferior, among all $\psi$ we conclude that
        \begin{equation}\label{eq:zzzusc_pcap}
            \limsup_{p \to 1^+} \ncapa_p(\Omega, U^{\sml{p}}_T) \leq \ncapa_1(\Omega, U^{\sml{1}}_T).
        \end{equation}
    
        \smallskip
    
        Let now $0< \varepsilon < T$. \cref{thm:local-approximation} imply again that $U_{T- \varepsilon}^{\sml{p}} \subseteq U^{\sml{1}}_T $ for any $p$ sufficiently close to $1$. Replacing in \cite[Theorem 1.2]{fogagnolo_minimisinghullscapacityisoperimetric_2022} the global Sobolev inequality with the local Sobolev inequality in $D$, we get that
        \begin{equation}
            \ncapa_1(\Omega, U^{\sml{1}}_T) \leq c_{n,p, D}\, \ncapa_p(\Omega, U_{T- \varepsilon}^{\sml{p}})^{\frac{n-1}{n-p}}
        \end{equation}
        for some constant $c_{n,p, D}\to 1$ as $p\to 1^+$ (see in particular \cite[Eq. (4.15)]{fogagnolo_minimisinghullscapacityisoperimetric_2022}). Taking the limit inferior and appealing again to \cref{lem:p-cap_levelset}, we get
        \begin{equation}\label{eq:zzzlsc_pcap}
            \ncapa_1(\Omega, U^{\sml{1}}_T) \leq \liminf_{p\to 1^+}\ncapa_p(\Omega, U_{T- \varepsilon}^{\sml{p}})^{\frac{n-1}{n-p}}= \liminf_{p \to 1^+} \ncapa_p(\Omega, U_{T}^{\sml{p}}).
        \end{equation}
        Since \cref{eq:zzzlsc_pcap,eq:zzzusc_pcap} imply \cref{eq:zzzcontinuity_p-cap}, the proof follows.
        \end{proof}
    
        \begin{corollary}\label{cor:energy_boundedness}
        Under the same assumptions of \cref{thm:local-approximation}, assuming $\inf_{\partial D} w_1 >0$, for every $0<T <\inf_{\partial D} w_1 $ there exists a constant $\kst= \kst(T)>0$ such that
            \begin{equation}
                \int_{\partial \Omega^{\sml{p}}_t} \abs{\nabla w_p}^{p-1} \dif \Hff^{n-1} \leq \kst(T) \ee^t
            \end{equation}
            for every $t\in [0,T]$ and $p\in (1,2]$.
        \end{corollary}
    
        \begin{proof}
            By \cref{lem:p-cap_levelset} we have
            \begin{equation}
                \ee^{-t}\int_{\partial \Omega^{\sml{p}}_t} \abs{\nabla w_p}^{p-1} \dif \Hff^{n-1}  = \abs{\S^{n-1}} (n-p)^{p-1} \left(1- \ee^{-\frac{T}{p-1}}\right)^{p-1}\ncapa_p(\Omega, U^{\sml{p}}_T) \leq \kst(T),
            \end{equation}
            where the last inequality follows from \cref{lem:continuity_p-capacity}.
        \end{proof}

\section{Monotonicity formulas in nonlinear potential theory}\label{sec:monotonocity_nonlinear_potential}

Within this section, we fix a compact set $\Omega \subset M$ with $\CS^{1,\beta}$-boundary, for some $\beta>0$. Let $w_p$ be the solution to the problem \cref{eq:moser_potential} for some $\Phi \in \Lip(\overline{D})$, where $D$ is an open precompact domain with smooth boundary containing $\Omega$, such that $\inf_{\partial D}\Phi>0$ and $1<p\le 2$. We define the function $\MF_{p}$ as
\begin{equation}\label{eq:monotonicty_formula}
\begin{split}
   \MF_{p}(t) \coloneqq\begin{multlined}[t][.6\textwidth]
   \ee^{\left(\frac{\alpha}{n-p} - 1\right)t} \int_{\partial \Omega^{\sml{p}}_t}\abs{\nabla w_p }^{\alpha+p-2} \left(\abs{\nabla w_p}\left(\frac{p-1}{n-p}- \frac{1}{\alpha}\right) +(p-1)\frac{\ip{\nabla \abs{\nabla w_p} | \nabla w_p}}{\abs{\nabla w_p}^2}\right) \dif \Hff^{n-1} \\-\int_0^t \ee^{\left(\frac{\alpha}{n-p} - 1\right)s}\int_{\partial \Omega^{\sml{p}}_s} \abs{\nabla w_p}^{\alpha+p-3} \Ric(\nu,\nu)\dif \Hff^{n-1} \dif s.
    \end{multlined}
\end{split}
\end{equation}
for almost every $t \in [0,\inf_{\partial D}\Phi)$ and by
\begin{equation}
    \MF_{p}(0) \coloneqq
             \int_{\partial \Omega}\abs{\nabla w_p }^{\alpha+p-2} \left[\abs{\nabla w_p}\left(\frac{n-1}{n-p}- \frac{1}{\alpha}\right) -\H\right] \dif \Hff^{n-1}.
\end{equation}
Well-posedness of the previous definition for almost every $t$ will be proved below in \cref{lem:continuity_G}. In this section we prove the following statement.
\begin{theorem}\label{thm:monotonicity-NPT}
    Let $p\in(1,2]$, $\alpha > (n-p)/(n-1)$. Let $\Omega \subset M$ be a compact set with $\CS^{1,\beta}$-boundary, for some $\beta>0$. Let $w_p$ be the solution to the problem \cref{eq:moser_potential} for some $\Phi \in \Lip(\overline{D})$, where $D$ is a precompact subset with smooth boundary containing $\Omega$, and $\inf_{\partial D}\Phi>0$. Fix $T<\inf_{\partial D} \Phi$. Assume that $\partial \Omega$ has generalized mean curvature $\H \in L^2(\partial \Omega)$. Then, the following statements hold:
    \begin{enumerate}
        \item $\partial \Omega^{\sml{p}}_t$ is a curvature varifold with square integrable second fundamental form for almost every $t \in (0,T)$;
        \item the function $\MF_{p}$ defined in \cref{eq:monotonicty_formula} belongs to $W^{1,1}(0,T)$, it is essentially monotone nondecreasing and $\MF_p(t) \searrow\MF_p(0^+) \geq \MF_p(0)$ as $t \to 0^+$; and
        \item for every $\psi \in \Lip_c([0,T))$ we have 
        \begin{equation}\label{eq:derivative-monotonicity_formula-NPT}
            - \int_0^T \psi' (t)  \MF_{p}(t) \dif t - \psi(0) \MF_{p}(0^+)    = \int_0^T \psi(t) \ee^{\left(\frac{\alpha}{n-p} - 1 \right)t} \int_{\partial \Omega^{\sml{p}}_t} \abs{\nabla w_p}^{\alpha+p-3} \SQ_{p}\dif \Hff^{n-1} \dif t,
        \end{equation}
        where 
        \begin{equation}\label{eq:Qpdef}
        \SQ_{p}\coloneqq \begin{multlined}[t]
            \left[\alpha-(2-p)\right] \frac{\abs{\nabla^\top \abs{\nabla w_p}}^2}{\abs{\nabla w_p}^2} + \abs{ \mathring{\h}}^2 + \frac{1}{p-1} \left[ \alpha- \frac{n-p}{n-1}\right] \left[ {\H} - \frac{n-1}{n-p}\abs{ \nabla w_p}\right]^2
        \end{multlined}
        \end{equation}
        on $D \smallsetminus \Crit(w_p)$ and $\SQ_{p}\coloneqq0$ otherwise.
    \end{enumerate}
    Moreover, $\MF_p(0) =\MF_p(0^+)$ if $\partial \Omega$ is smooth.
\end{theorem}

\begin{remark}
    By \cref{rmk:weakSard}, $\Hff^{n-1}(\partial \Omega^{\sml{p}}_t \cap \Crit (w_p))=0 $ for almost every $t \in [0,T)$. Moreover, since $\partial \Omega^{\sml{p}}_t $ is a curvature varifold and $w_p\in \CS^\infty(\partial \Omega^{\sml{p}}_t \smallsetminus \Crit (w_p))$, \cref{thm:MeanCruvatureDivergence} and \cref{eq:MeanCurvatureLevelWP} imply that the generalized mean curvature satisfies
    \begin{equation}
        \vec{\H} = - \H \nu= \left(-\abs{\nabla w_p} + (p-1) \frac{\ip{\nabla \abs{\nabla w_p}|\nabla w_p}}{\abs{\nabla w_p}^2}\right) \frac{\nabla w_p}{\abs{\nabla w_p}}.
    \end{equation}
    Hence, the function \cref{eq:monotonicty_formula} can be rewritten as 
    \begin{equation}
        \begin{split}
        \MF_{p}(t) =\begin{multlined}[t][.6\textwidth]
            \ee^{\left(\frac{\alpha}{n-p} - 1\right)t} \int_{\partial \Omega^{\sml{p}}_t}\abs{\nabla w_p }^{\alpha+p-2} \left[\abs{\nabla w_p}\left(\frac{n-1}{n-p}- \frac{1}{\alpha}\right) -\H\right] \dif \Hff^{n-1} \\-\int_0^t \ee^{\left(\frac{\alpha}{n-p} - 1\right)s}\int_{\partial \Omega^{\sml{p}}_s} \abs{\nabla w_p}^{\alpha+p-3} \Ric(\nu,\nu)\dif \Hff^{n-1} \dif s.
            \end{multlined}
            \end{split}
    \end{equation}
\end{remark}

The monotonicity of \cref{eq:monotonicty_formula} can be proved even without giving the geometric meaning to $\H$ explained in the above remark. This approach is the one followed in the literature so far (see \cite{benatti_minkowskiinequalitycompleteriemannian_2022,agostiniani_riemannianpenroseinequalitynonlinear_2022,benatti_nonlinearisocapacitaryconceptsmass_2023a,xia_newmonotonicitycapacitaryfunctions_2024}). This interpretation is possible since almost every level set $\partial \Omega^{\sml{p}}_t$ is regular around $\Hff^{n-1}$-almost every point and $w_p$ is smooth around these points. However, we will need that $\partial \Omega^{\sml{p}}_t$ is a curvature varifold in the approximation as $p \to 1^+$. 

We employ the classic approximation scheme \textit{via} the regularized problem \cref{eq:eps_moser_potential}. First, we compute the divergence of some vector fields related to the definition of \cref{eq:monotonicty_formula}. We use these formulas to control the $L^2$ norm of the second fundamental form of the level sets of solutions $w^\varepsilon_p$ to \cref{eq:eps_moser_potential}. Second, since this estimate is independent of $\varepsilon$, we pass it to the limit as $\varepsilon \to 0^+$, inferring that $\partial \Omega^{\sml{p}}_t$ is a curvature varifold for almost every $t$. In conclusion, we will show the monotonicity result with the expression of the derivative \cref{eq:derivative-monotonicity_formula-NPT}.

\subsection{Main computations}
We compute the divergence of the vector field that will be the core of the monotonicity formulas in \cref{thm:monotonicity-NPT} and of the regularity results we will prove in the following subsections. 

Given $\Omega\subseteq M$, $w_p$ and $w^\varepsilon_p$ will denote the solutions to \cref{eq:moser_potential} and \cref{eq:eps_moser_potential} associated to $\Omega$ respectively. We define the following vector field
\begin{equation}\label{eq:Xdef}
    J_\varepsilon \coloneqq \abs{\nabla w^\varepsilon_p}^{\alpha+ p-2}\nabla w^\varepsilon_p \in L^\infty(D \smallsetminus \Omega).
\end{equation}
We will write $J$ in place of $J_0$.
\begin{lemma} \label{lem:divX}
Let $J_\varepsilon$ be as in \cref{eq:Xdef}. Then,
    \begin{align}\label{eq:divX}
    &\div(J_\varepsilon) = [\alpha -(2-p)\Err_\varepsilon]  \abs{\nabla w^\varepsilon_p}^{\alpha+ p-3}\ip{\nabla \abs{\nabla w_p^\varepsilon}|\nabla w^\varepsilon_p} + \left(1+ \frac{2-p}{p-1} \Err_\varepsilon\right)\abs{\nabla w^\varepsilon_p}^{\alpha + p } \\
    \label{eq:divEX}
    &\div(\Err_\varepsilon J_\varepsilon) = \Err_\varepsilon[(\alpha-2)+p\Err_\varepsilon] \abs{\nabla w^\varepsilon_p}^{\alpha+ p-3}\ip{\nabla \abs{\nabla w_p^\varepsilon}|\nabla w^\varepsilon_p} + \Err_\varepsilon \left( \frac{p+1}{p-1} - \frac{p}{p-1} \Err_\varepsilon \right) \abs{\nabla w^\varepsilon_p }^{\alpha + p } 
\end{align}
in the sense of distributions for any $\varepsilon\geq0$,  where we set  $\Err_0\equiv 0$. In particular, $\div(J) \in L^2_{\loc}(D \smallsetminus \Omega)$
\end{lemma}

\begin{proof} By straightforward computations, we have
    \begin{align}
        \div(J_\varepsilon) &= (\alpha+ p -1)\abs{\nabla w^\varepsilon_p}^{\alpha+ p-3}\ip{\nabla \abs{\nabla w_p^\varepsilon}|\nabla w^\varepsilon_p} + \abs{\nabla w^\varepsilon_p}^{\alpha + p -1} \H \\
        \overset{\cref{eq:MeanCurvatureLevelWepsP}}&{=}[\alpha -(2-p)\Err_\varepsilon]  \abs{\nabla w^\varepsilon_p}^{\alpha+ p-3}\ip{\nabla \abs{\nabla w_p^\varepsilon}|\nabla w^\varepsilon_p} + \left(1+ \frac{2-p}{p-1} \Err_\varepsilon\right)\abs{\nabla w^\varepsilon_p}^{\alpha + p }.
    \end{align}
    Since $\Err_\varepsilon \in L^\infty \cap W^{1,2}_{\loc}(D \smallsetminus \Omega)$, by Leibniz rule we have
    \begin{align}
        \div(\Err_\varepsilon J_\varepsilon) &= \Err_\varepsilon \div J_\varepsilon + \ip{J_\varepsilon | \nabla \Err_\varepsilon}\\
        &=\Err_\varepsilon \div J_\varepsilon -2\Err_\varepsilon(1- \Err_\varepsilon) \ip{J_\varepsilon | \frac{\nabla \abs{\nabla w^\varepsilon_p}}{\abs{\nabla w^\varepsilon_p}}- \frac{\nabla w^\varepsilon_p}{p-1}}\\
        &= \Err_\varepsilon[(\alpha-2)+p\Err_\varepsilon] \abs{\nabla w^\varepsilon_p}^{\alpha+ p-3}\ip{\nabla \abs{\nabla w_p^\varepsilon}|\nabla w^\varepsilon_p} + \Err_\varepsilon \left( \frac{p+1}{p-1} - \frac{p}{p-1} \Err_\varepsilon \right) \abs{\nabla w^\varepsilon_p }^{\alpha + p }.
    \end{align}
    By \cref{prop:RegolaritaWp}, $\div(J)\in L^2_{\loc}(D \smallsetminus \Omega)$.
\end{proof}
We now define the following vector field
\begin{equation}\label{eq:defY}
    Y_\varepsilon \coloneqq \abs{\nabla w^\varepsilon_p}^{\alpha+p-3}\left(\nabla \abs{\nabla w^\varepsilon_p } +(p-2) (1- \Err_\varepsilon) \nabla^\perp\abs{\nabla w_p^\varepsilon} \right),
\end{equation}
at every point of $D\smallsetminus \Crit(w^\varepsilon_p)$. We will denote $Y$ in place of $Y_0$.

\begin{lemma}\label{lem:divY}
Let $Y_\varepsilon$ be as in \cref{eq:defY}. Then,
    \begin{equation}\label{eq:divY}
    \div(Y_\varepsilon) = \abs{\nabla w^\varepsilon_p}^{\alpha+p-2}(\mathcal{D}_+ + \mathcal{D}_\sigma)
\end{equation}
at every point of $D \smallsetminus \Crit(w^\varepsilon_p)$, where
\begin{align}
\mathcal{D}^{\sml{\varepsilon}}_+&\coloneqq\begin{multlined}[t][.8\textwidth]
         (p-1)^2 \left(1+ \frac{2-p}{p-1} \Err_\varepsilon\right) \left[\frac{\alpha+p-2}{p-1}-\frac{n-2}{n-1}\left(1+ \frac{2-p}{p-1}\Err_\varepsilon\right)\right]\frac{\abs{\nabla^\perp\abs{\nabla w^\varepsilon_p}}^2}{\abs{\nabla w^\varepsilon_p }^2}\\
         +(\alpha+ p -2)\frac{\abs{\nabla^\top \abs{\nabla w^\varepsilon_p}}^2}{\abs{\nabla w^\varepsilon_p }^2} + \abs{\mathring{\h}}^2,
    \end{multlined}\\
\mathcal{D}_{\mathrlap{\sigma}\hphantom{+}}^{\sml{\varepsilon}}&\coloneqq\begin{multlined}[t][.8\textwidth]
\Ric\left(\frac{\nabla w^\varepsilon_p}{\abs{\nabla w^\varepsilon_p }},\frac{\nabla w^\varepsilon_p}{\abs{\nabla w^\varepsilon_p }}\right)+ \left[ \frac{1}{n-1} \left( 1+ \frac{2-p}{p-1} \Err_\varepsilon \right)^2 +2 \frac{2-p}{(p-1)^2} (1- \Err_\varepsilon) \Err_\varepsilon \right]\abs{\nabla w^\varepsilon_p}^2\\
     +\left[2 \frac{(p-1)(n-2)}{n-1} \left( 1+ \frac{2-p}{p-1} \Err_\varepsilon \right)^2 + (2-p) (1- \Err_\varepsilon) \left( 1- \frac{p}{p-1} \Err_\varepsilon \right)\right] \ip{\nabla \abs{\nabla w^\varepsilon_p}| \frac{\nabla w^\varepsilon_p}{\abs{\nabla w^\varepsilon_p }}},
    \end{multlined}
\end{align}
with $\Err_0\equiv 0$. Moreover, $\mathcal{D}_\sigma \in L^2_{\loc}(D \smallsetminus \Omega)$ and, provided $\alpha \geq (3-p) - 1/(n-1)$, $\mathcal{D}_+ \geq 0$.
\end{lemma}

\begin{proof}
    Using the equation of $w^\varepsilon_p$, the vector field $Y_\varepsilon$ can be also rewritten as
\begin{equation}
    Y_\varepsilon= -\abs{\nabla w^\varepsilon_p}^{\alpha+p-4}\left[ \Delta w^\varepsilon_p \nabla w^\varepsilon_p - \frac{1}{2} \nabla \abs{\nabla w^\varepsilon_p}^2\right] + \left(1+ \frac{2-p}{p-1}\Err_\varepsilon\right) J_\varepsilon,
\end{equation}
where $J_\varepsilon$ is defined in \cref{eq:divX}.
Hence, we have 
\begin{equation}
    \div(Y_\varepsilon) =\begin{multlined}[t][.8\textwidth](\alpha + p-4)\abs{\nabla w^\varepsilon_p}^{\alpha+p-5} \left( \abs{\nabla w^\varepsilon_p} \abs{\nabla \abs{\nabla w^\varepsilon_p}}^2-\Delta w^\varepsilon_p\ip{\nabla \abs{\nabla w^\varepsilon_p}| \nabla w^\varepsilon_p}\right)\\+\abs{\nabla w^\varepsilon_p }^{\alpha+p-4}\left( \frac{1}{2} \Delta \abs{\nabla w^\varepsilon_p }^2-\ip{\nabla \Delta w^\varepsilon_p | \nabla w^\varepsilon_p }- (\Delta w^\varepsilon_p)^2\right)  +\div \left(\left(1+ \frac{2-p}{p-1} \Err_\varepsilon\right)J_\varepsilon\right).
    \end{multlined}
\end{equation}
We already computed the divergence of $J_\varepsilon $ and $\Err_\varepsilon J_\varepsilon$ in \cref{eq:divX,eq:divEX}. The Laplacian of $w_\varepsilon^p$ can be computed using
\begin{equation}
    \Delta w^\varepsilon_p = \abs{\nabla w^\varepsilon_p} \H  + \frac{\ip{\nabla \abs{\nabla w^\varepsilon_p}| \nabla w^\varepsilon_p}}{\abs{\nabla w^\varepsilon_p}},
\end{equation}
and replacing $\H$ with its expression in \cref{eq:MeanCurvatureLevelWepsP}. To conclude, by combining Kato's and Bochner's identities, one has
\begin{align}
    \frac{1}{2} \Delta \abs{\nabla w^\varepsilon_p }^2 -\ip{\nabla \Delta w^\varepsilon_p | \nabla w^\varepsilon_p }&= \Ric(\nabla w^\varepsilon_p,\nabla w^\varepsilon_p) +\abs{\nabla^2 w^\varepsilon_p }^2\\& =\abs{\nabla w^\varepsilon_p}^2 \left(\Ric(\nu,\nu) +\abs{\mathring\h}^2 + 2\frac{\abs{\nabla^\top\abs{\nabla w^\varepsilon_p}}^2}{\abs{\nabla w^\varepsilon_p}^2} + \frac{\abs{\nabla^\perp \abs{\nabla w^\varepsilon_p }}^2}{\abs{\nabla w^\varepsilon_p}^2}+ \frac{\H^2}{n-1}\right).
\end{align}
The claimed \cref{eq:divY} follows employing again \cref{eq:MeanCurvatureLevelWepsP}.
\end{proof}

\subsection{Regularity of level sets of \texorpdfstring{$w_p$}{wp}}
We will now use the above computations to get regularity estimates on geometric quantities (namely the mean curvature and the second fundamental form).
Within this section we denote by $\Omega\subseteq M$ the given compact subset of $M$ and by $w_p$ and $w^\varepsilon_p$ the solutions to \cref{eq:moser_potential} and \cref{eq:eps_moser_potential} associated to $\Omega$ respectively.

We will show that for almost every $t \in [0,\inf_{\partial D}\Phi)$ the level set $\partial \Omega^{\sml{p}}_t$ is a curvature varifold (see \cref{lem:ConvergenzaVarifoldEspToZero}).

\begin{proposition}[Regularity estimate]\label{prop:interior_regularity_estimate}
   Assume $\Omega \subseteq  M $ is smooth. For every $0<\tau< T<\inf_{\partial D} \Phi$ and $\tau<\inf_{\partial D} \Phi - T$, it holds
\begin{equation}\label{eq:interior_regularity_eps}
        \int_\tau^{T} \int_{\partial \Omega^{\sml\varepsilon}_s} \abs{{\h}}^2 \dif \Hff^{n-1} \dif s\leq \kst\left(1+ \int_{\tau/2}^{T+\tau} \int_{\partial \Omega_s^{\sml{\varepsilon}}} \abs{\nabla w^\varepsilon_p}^2\dif \Hff^{n-1} \dif s\right),
    \end{equation}
    for some positive constant $\kst$ depending on $\tau$, $T$, $D$, $p$, $n$ and a lower bound on the Ricci tensor.
\end{proposition}

\begin{proof}
Fix $\alpha = 3-p$.   Observe that $\mathcal{D}^{\sml{\varepsilon}}_+\geq 0$ under this assumption. The proof is divided into two steps. First, we prove that
    \begin{equation}\label{eq:zzintegration}
        \begin{multlined}[t][.8\textwidth]
        -\int_0^{T+\tau}\varphi'(t) \int_{\partial \Omega_t^{\sml{\varepsilon}}}\ip{Y_\varepsilon | \nu^{\sml{\varepsilon}}} \dif \Hff^{n-1} \dif t\geq \int_0^{T+\tau}\varphi(t) \int_{\partial \Omega_t^{\sml{\varepsilon}}}(\mathcal{D}_+^{\sml{\varepsilon}} + \mathcal{D}_\sigma^{\sml{\varepsilon}}) \dif \Hff^{n-1} \dif t
        \end{multlined}
    \end{equation}
    for every nonnegative $\varphi \in \Lip_{c}(0,T+\tau)$. Second, we will use \cref{eq:zzintegration} to infer \cref{eq:interior_regularity_eps}.

    \begin{step}[Proof of \cref{eq:zzintegration}]\label{step:interior_estimate_integration_by_part}
    Let $\chi_\delta$ be a smooth cut-off function satisfying 
    \begin{equation}
        \chi_\delta(t) \coloneqq\begin{cases}
            &\chi_\delta(t) =0& & \text{if $t \leq \delta$,}\\
            &0\leq \chi'_\delta(t) \leq 2/\delta&& \text{if $t \in [\delta, 2 \delta]$,}\\
            &\chi_\delta(t)=1& &\text{if $t \geq2 \delta$.}\\
        \end{cases}
    \end{equation}
    Define $Y^{\sml{\delta}}_\varepsilon \coloneqq \chi_\delta(\abs{\nabla w^\varepsilon_p }) Y_\varepsilon$. Then, $Y^{\sml{\delta}}_\varepsilon$ is smooth for every $\delta >0$ and it vanishes identically on $\Crit(w^\varepsilon_p)$. Since $\ip{Y^{\sml{\delta}}_\varepsilon|\nabla [\varphi(w^\varepsilon_p)]}= \varphi'(w^\varepsilon_p) \ip{Y^{\sml{\delta}}_\varepsilon | \nabla w^\varepsilon_p } \in L^2_{\loc}(D \smallsetminus \Omega)$, by dominated convergence theorem we have that
    \begin{align}\label{eq:zz}
       \int_{D \smallsetminus \Omega} \ip{Y_\varepsilon |\nabla [\varphi(w^\varepsilon_p)]}\dif \Hff^n =\lim_{\delta \to 0} \int_{D \smallsetminus \Omega} \ip{Y^{\sml{\delta}}_\varepsilon | \nabla [\varphi(w^\varepsilon_p)]}\dif \Hff^n
    \end{align}
    Moreover, by divergence theorem, we have that
    \begin{equation}
        \int_{D \smallsetminus \Omega} \ip{Y^{\sml{\delta}}_\varepsilon | \nabla[\varphi(w^\varepsilon_p)]} \dif \Hff^n=\int_{D \smallsetminus \Omega} [\chi'_\delta (\abs{\nabla w^\varepsilon_p })\ip{Y_\varepsilon | \nabla \abs{\nabla w^\varepsilon_p }} + \chi_\delta(\abs{\nabla w^\varepsilon_p }) \div(Y_\varepsilon ) ]\varphi(w^\varepsilon_p)\dif \Hff^n.
    \end{equation}
    The very definition of $Y_\varepsilon$ implies
    \begin{equation}\label{eq:zznonegativity}
        \ip{Y_\varepsilon| \nabla \abs{\nabla w^\varepsilon_p}} =  \left[\abs{\nabla^\top \abs{\nabla w^\varepsilon_p }}^2 + (p-1) \left(1+ \frac{2-p}{p-1} \Err\right) \abs{\nabla^\perp \abs{\nabla w^\varepsilon_p }}^2\right]\geq 0 
    \end{equation}
    
    Recalling that $\varphi \abs{\nabla w }\mathcal{D}^{\sml{\varepsilon}}_\sigma \in L^2(\Omega_T \smallsetminus \Omega)$, since $\mathcal{D}^{\sml{\varepsilon}}_+\geq0 $, we have
    \begin{align}\label{eq:limite pijamini divergenza}
    \lim_{\delta \to 0} \int_{D \smallsetminus \Omega}  \chi_\delta(\abs{\nabla w^\varepsilon_p })\varphi(w^\varepsilon_p)\div(Y_\varepsilon)\dif \Hff^n&=\lim_{\delta \to 0}\int_{D \smallsetminus \Omega} \chi_\delta(\abs{\nabla w^\varepsilon_p })\varphi(w^\varepsilon_p)({\mathcal{D}_+^{\sml{\varepsilon}}+ \mathcal{D}^{\sml{\varepsilon}}_\sigma}){\abs{\nabla w^\varepsilon_p}}\dif \Hff^n\\&= \int_{D \smallsetminus \Omega}\varphi(w^\varepsilon_p)({\mathcal{D}^{\sml{\varepsilon}}_++ \mathcal{D}^{\sml{\varepsilon}}_\sigma})(w^\varepsilon_p){\abs{\nabla w^\varepsilon_p }}\dif \Hff^n,
    \end{align}
    where we used the dominated convergence theorem for passing to the limit $\delta\to0$ for the term containing $\mathcal{D}_{\sigma}$, and the monotone convergence theorem for the other one, since $\mathcal{D}_+\geq 0$. Recalling \cref{eq:zz,eq:zznonegativity}, we can conclude that    \begin{equation}\label{eq:integrazione per parti campo X}
        - \int_{D \smallsetminus \Crit(w^\varepsilon_p)} \ip{Y_\varepsilon|\nabla [\varphi(w^\varepsilon_p)]}\dif \Hff^n\geq \int_{D \smallsetminus \Crit(w^\varepsilon_p)}\varphi(w^\varepsilon_p)({\mathcal{D}_+^{\sml{\varepsilon}}+ \mathcal{D}^{\sml{\varepsilon}}_\sigma}){\abs{\nabla w^\varepsilon_p }}\dif \Hff^n,
    \end{equation}
    holds for every nonnegative function $\varphi \in \Lip_c(0,T)$. Thus \cref{eq:zzintegration} follows by coarea formula.
    \end{step}
    \begin{step}[Proof of \cref{eq:interior_regularity_eps}]\label{step:interior_estimate} Take now $\varphi(s)= 1$ for $s\in[\tau,T]$, $\varphi(s)=0$ for $s\in\R\smallsetminus[\tau/2, T+\tau]$ and $\abs{\varphi'(s)}\leq 8/\tau$, and plug $\varphi^2$ into \cref{eq:zzintegration}. Hence, by Young's inequality, for every $\eta>0$ we have
    \begin{equation}
        \ip{Y_\varepsilon | \nabla [\varphi(w^\varepsilon_p)]^2} \geq - \eta^2 [\varphi(w^\varepsilon_p)]^2 \frac{\abs{\nabla \abs{\nabla w^\varepsilon_p }}^2}{\abs{\nabla w^\varepsilon_p}}- \frac{\kst}{\eta^2}[\varphi'(w^\varepsilon_p)]^2 \abs{\nabla w^\varepsilon_p}^3,
    \end{equation}
    and
    \begin{equation}
        [\varphi(w^\varepsilon_p)]^2(\mathcal{D}^{\sml{\varepsilon}}_+ + \mathcal{D}^{\sml{\varepsilon}}_\sigma) \abs{\nabla w^\varepsilon_p} \geq [\varphi(w^\varepsilon_p)]^2\left[\left( \kst - \eta^2 \right)\frac{\abs{\nabla \abs{\nabla w^\varepsilon_p }}^2}{\abs{\nabla w^\varepsilon_p}^2} + \abs{\mathring{\h}}^2 - \frac{\kst}{\eta^2} \abs{\nabla w^\varepsilon_p}^2 - \kappa\right]\abs{\nabla w^\varepsilon_p} ,
    \end{equation}
    where $\kst>0$ is a positive constant depending only on $p$ and $n$, and $\kappa$ is any lower bound on $\Ric$. Moreover, since $w^\varepsilon_p$ is smooth up to the boundary, \cref{eq:MeanCurvatureLevelWepsP} holds on $\partial \Omega$. Choosing $\eta\ll \kst/4$ and using coarea formula, we obtain that
    \begin{equation}
\int_{\tau}^t \int_{\partial\Omega^{\sml{\varepsilon}}_s}\abs{\mathring{\h}}^2 + \frac{\abs{\nabla \abs{\nabla w^\varepsilon_p }}^2}{\abs{\nabla w^\varepsilon_p}^2} \dif \Hff^{n-1} \dif s   \leq \kst \left(1+ \int_{\tau/2}^{T} \int_{\partial\Omega^{\sml{\varepsilon}}_{s}} \abs{\nabla w^\varepsilon_p}^2 \dif \Hff^{n-1} \dif s \right),
    \end{equation}
for some constant $\kst>0$ with the same dependencies as in the statement. The expression \cref{eq:MeanCurvatureLevelWepsP} of the mean curvature gives that
    \begin{equation}\label{eq:zz2ffestimate}
        \abs{\h}^2 \leq \abs{\mathring{\h}}^2+ \frac{1}{n-1} \abs{\H}^2\leq \kst \left(\abs{\mathring{\h}}^2+ \abs{\nabla w^\varepsilon_p}^2 + \frac{\abs{\nabla \abs{\nabla w^\varepsilon_p}}^2}{\abs{\nabla w^\varepsilon_p}^2}\right)
    \end{equation}
    holds at any regular point of $\partial \Omega^{\sml{\varepsilon}}_s$, for some constant $\kst$ with the same dependencies of the statement.  \qedhere  
    \end{step}
\end{proof}

For smooth boundaries \cref{prop:interior_regularity_estimate} can be promoted to a global estimate.

\begin{corollary}\label{cor:global_regularity_estimate}
    Assume $\Omega \subseteq  M $ is smooth.
    Then, for every $0<T< \inf_{\partial D} \Phi$  and $0< \tau<\inf_{\partial D} \Phi - T$ we have
    \begin{equation}\label{eq:global_regularity_eps}
        \int_0^T \int_{\partial \Omega^{\sml\varepsilon}_s} \abs{{\h}}^2 \dif \Hff^{n-1} \dif s \leq \kst\left(1+ \int_{0}^{T+\tau} \int_{\partial \Omega_s^{\sml{\varepsilon}}} \abs{\nabla w^\varepsilon_p}^2\dif \Hff^{n-1} \dif s+\int_{\partial \Omega}\H^2+ \abs{\nabla w^\varepsilon_p}^2 \dif \Hff^{n-1} \right),
    \end{equation}
    for some positive constant $\kst$ depending on $T-t$, $D$, $p$, $n$ and any lower bound on the Ricci tensor.
\end{corollary}

\begin{proof}
    Since $\partial \Omega$ is smooth, by Hopf's maximum principle there exists $\tau>0$ such that $\Crit(w_p^\varepsilon)\cap (\Omega_{\tau}^{\sml{\varepsilon}}\smallsetminus \Omega) = \varnothing$. Therefore, as the vector field $Y_\varepsilon$ is smooth in $\Omega_{\tau}^{\sml{\varepsilon}}\smallsetminus \Omega $, we can apply the divergence theorem coupled with the coarea formula
    \begin{align}
        -\int_0^{T+\tau}\varphi'(s) \int_{\partial \Omega_t^{\sml{\varepsilon}}}\ip{Y_\varepsilon | \nu^{\sml{\varepsilon}}} \dif \Hff^{n-1} \dif s &= -\int_{D\smallsetminus \Omega}\ip{Y_\varepsilon | \nabla [\varphi(w^\varepsilon_p)]}\dif \Hff^n \\\overset{\cref{eq:divY}}&{\geq}\varphi(0) \int_{\partial \Omega} \ip{Y_\varepsilon| \nu} \dif \Hff^{n-1}+ \int_{D\smallsetminus \Omega}\varphi( w^\varepsilon_p)\abs{\nabla w^\varepsilon_p} ({\mathcal{D}_+^{\sml{\varepsilon}}+ \mathcal{D}^{\sml{\varepsilon}}_\sigma})\dif \Hff^n
        \\&=  \varphi(0) \int_{\partial \Omega} \ip{Y_\varepsilon| \nu} \dif \Hff^{n-1}+  \int_0^{T}\varphi(s) \int_{\partial \Omega_s^{\sml{\varepsilon}}}({\mathcal{D}_+^{\sml{\varepsilon}}+ \mathcal{D}^{\sml{\varepsilon}}_\sigma}) \dif \Hff^{n-1} \dif s.
    \end{align}
    for every $\varphi \in \Lip_c([0,\tau))$. By \cref{step:interior_estimate_integration_by_part} of \cref{prop:interior_regularity_estimate}, the above identity holds actually for every $\varphi \in \Lip_c([0,T+\tau))$. The proof of \cref{eq:global_regularity_eps} now follows as in \cref{step:interior_estimate} of \cref{prop:interior_regularity_estimate}. Indeed, one can take $\varphi \in \Lip_c([0,T+\tau)$, $\varphi(s)= 1$ for $s\in[0,T]$, $\varphi(s)=0$ for $s\geq T+\tau$ and $\abs{\varphi'(s)}\leq 8/\tau$, and plug $\varphi^2$ in the above identity. The additional term can be ruled employing Young's inequality, since 
    \begin{equation}
        \int_{\partial \Omega } \ip{Y_\varepsilon| \nu }\dif \Hff^{n-1} \leq \kst\int_{\partial \Omega} \abs{\H}^2 + \abs{\nabla w^\varepsilon_p}^2\dif \Hff^{n-1},
    \end{equation}
    for some positive constant $\kst>0$ depending only on $p$ and $n$. 
\end{proof}

We are now ready to prove that almost every level of $w_p$ is a curvature varifold, addressing the first part of \cref{thm:monotonicity-NPT}.

\begin{theorem}\label{lem:ConvergenzaVarifoldEspToZero}
    Let $p\in(1,2]$ and let $\Omega \subset M$ be a compact set with smooth boundary.
    Let $T>0$ and let $w_p$ (resp. $w^\varepsilon_p$) be the solution to the problem \cref{eq:moser_potential}  (resp. \cref{eq:eps_moser_potential}) for some $\Phi \in \Lip(\overline{D})$, where $D$ is a precompact subset with smooth boundary containing $\Omega$, such that $\inf_{\partial D}\Phi>0$. Fix $T<\inf_{\partial D}\Phi$ and $0<\tau< \inf_{\partial D} \Phi -T$.  Then, there exists a vanishing sequence $(\varepsilon_k)_{k \in \N}$, such that
    \begin{enumerate}
        \item $\partial \Omega^{\sml{\varepsilon_k}}_t$ converges in the sense of varifold to $\partial \Omega^{\sml{p}}_t$ as $k\to +\infty$, for almost every $t \in [0,T]$;
        \item $\partial \Omega^{\sml{p}}_t$ is a curvature varifold satisfying the properties in \cref{lem:ConditionsEquivalentSFF} for almost every $t \in [0,T]$;
        
        \item the generalized second fundamental form of $\partial \Omega^{\sml{\varepsilon_k}}_t$ weakly converges to that of $\partial \Omega^{\sml{p}}_t$ as $k\to+\infty$, for almost every $t \in [0,T]$, in the sense of \eqref{eq:WeakConvergenceComponentsSFF};
        
        \item the generalized second fundamental form of $\partial \Omega^{\sml{p}}_t$ satisfies
        \begin{equation}\label{eq:interior_esitmate_2ff_p}
           \int_0^T \int_{\partial \Omega^{\sml{p}}_t} \abs{{\h}}^2 \dif \Hff^{n-1} \dif t \leq \kst\left(1+ \int_0^{T+\tau} \int_{\partial \Omega_t^{\sml{p}}} \abs{\nabla w_p}^2\dif \Hff^{n-1} \dif t+\int_{\partial \Omega}\H^2+ \abs{\nabla w_p}^2 \dif \Hff^{n-1} \right),
        \end{equation}
        for some constant $\kst>0$ depending on $T$, $\tau$, $D$, $p$, $n$ and any lower bound on the Ricci curvature tensor on $D$.
    \end{enumerate}
\end{theorem}

\begin{proof}
Since $\partial \Omega$ is smooth, then $w^\varepsilon_p \to w_p$ in $\CS^1(D \smallsetminus \Omega^\circ)$. Observe that the right-hand side of \cref{eq:global_regularity_eps} passes to the limit as $\varepsilon\to 0$, yielding the right-hand side of \cref{eq:interior_esitmate_2ff_p}. Fix a vanishing sequence $(\varepsilon_k)_{k \in \N}$. By \cref{eq:global_regularity_eps}, we find a set of full measure in $(0,T)$ such that the level set $\partial \Omega^{\sml{\varepsilon_k}}_t$ is smooth, for any $t$ in such set and any $k$.
By \cref{lem:area_convergence_under_BV_convergence} we can also assume that $\abs{\partial \Omega^{\sml{\varepsilon_k}}_t} \to \abs{\partial \Omega^{\sml{p}}_t}$ for almost every $t\in(0,T)$, up to passing to a subsequence.
Moreover, by \cref{eq:global_regularity_eps}, up to passing to a further subsequence, we can assume that
\begin{equation}\label{eq:zzinferiorlimit2ff}
        \sup_k \int_{\partial \Omega^{\sml{\varepsilon_k}}_t} \abs{\h^{\sml{\varepsilon_k}}}^2 \dif \Hff^{n-1} < +\infty,
\end{equation}
for almost every $t\in(0,T)$. Fix such a $t$. Hence, \cref{thm:ProprietaConvergenzeVarifold,cor:ConvergenzaLevelSetsVersusVarifolds} imply that $\partial \Omega^{\sml{\varepsilon_k}}_t$ converges in the sense of varifolds to $\partial \Omega^{\sml{p}}_t$ and $\partial \Omega^{\sml{p}}_t$ is a curvature varifold.
Finally, by \cref{thm:ProprietaConvergenzeVarifold} and Fatou's lemma, we get that \cref{eq:global_regularity_eps} implies \cref{eq:interior_esitmate_2ff_p}.

\end{proof}

\begin{remark}\label{rem:ConvergenzaVarifoldEspToZeroNONsmoothOmega}
    Both estimates \cref{eq:global_regularity_eps,eq:interior_esitmate_2ff_p} also hold assuming $\partial \Omega \in \CS^{1,\beta}$ with generalized mean curvature $\H\in L^2(\partial \Omega)$. Indeed, one can approximate $\partial \Omega$ with subsets with smooth boundary in such a way that \cref{eq:global_regularity_eps,eq:interior_esitmate_2ff_p} are stable. A more detailed explanation of this approximation procedure will be presented in the proof of \cref{thm:monotonicity-NPT}.
\end{remark}

\subsection{Well-posedness of the monotone quantity}\label{sec:well-posedness}
Before proving \cref{thm:monotonicity-NPT}, we show that the definition of the function $\MF_p$ in \cref{eq:monotonicty_formula} is well-posed for almost every $t\in[0,\inf_{\partial D}\Phi)$. This assertion is clearly true in the presence of regular values. Indeed, everything is smooth as the gradient does not vanish in a tubular neighborhood of the regular level set. Also, taking into account the regularity of $w_p$, the expression of the mean curvature \cref{eq:MeanCurvatureLevelWP}, and \cref{rmk:weakSard}, we notice that the integrand in the definition \cref{eq:monotonicty_formula} of $\MF_p$ is well-defined $\Hff^{n-1}$-a.e. along almost every level set. We wish to prove now that, in fact, $\MF_p$ is an $L^1_{\loc}$ function.

\medskip
We define
    \begin{equation}\label{eq:pezzo-di-gradiente}
        \IF_p(t)\coloneqq  \ee^{\left(\frac{\alpha}{n-p}-1\right) t} \int_{\partial \Omega^{\sml{p}}_t} \abs{\nabla w_p}^{\alpha+p-1} \dif \Hff^{n-1} 
    \end{equation}
for every $t\in [0,\inf_{\partial D}\Phi)$.

\begin{lemma}\label{lem:continuity_G}
    Under the same hypotheses of \cref{thm:monotonicity-NPT}, the function $\IF_p$ is continuous in $(0,T)$ and belongs to $W^{1,1}_{\loc}(0,T)$. Moreover, there holds
    \begin{equation}\label{eq:defivata-di-G}
        \IF_p'(t) = \frac{1}{p-1} \left( \IF_p(t) + \alpha \MF_p(t) + \alpha \int_0^t \ee^{\left(\frac{\alpha}{n-p}-1\right)s}\int_{\partial \Omega_s^{\sml{p}}} \abs{\nabla w_p}^{\alpha + p-3} \Ric(\nu,\nu) \dif \Hff^{n-1} \dif s\right)
    \end{equation}
    for almost every $t\in (0,T)$, and $\MF_p$ belongs to $L^1_{\loc}(0,T)$.
\end{lemma}

\begin{proof}
    Let $\delta>0$ and let $\chi_\delta$ be a smooth cut-off function satisfying 
    \begin{equation}
        \chi_\delta(t) \coloneqq\begin{cases}
            &\chi_\delta(t) =0& & \text{if $t \leq \delta$,}\\
            &0\leq \chi'_\delta(t) \leq 2/\delta&& \text{if $t \in [\delta, 2 \delta]$,}\\
            &\chi_\delta(t)=1& &\text{if $t \geq2 \delta$,}\\
        \end{cases}
    \end{equation}
    and consider the vector field $J^{\sml{\delta}}=\chi_{\delta} (\abs{\nabla w_p})J$, where $J$ is defined in \cref{eq:Xdef}. Observe that $J^{\sml{\delta}}$ is smooth. Hence, we can apply the divergence theorem for $0<\tau<t<T$ to get
    \begin{align}\label{eq:zzzdivergenceJcontinuity}
        \int_{\partial \Omega^{\sml{p}}_t} \ip{J^{\sml{\delta}}| \nu^{\sml{p}}} \dif \Hff^{n-1} - 
        \int_{\partial \Omega^{\sml{p}}_\tau} \ip{J^{\sml{\delta}}| \nu^{\sml{p}}} \dif \Hff^{n-1} &= \int_{\Omega^{\sml{\delta}}_t \smallsetminus \Omega^{\sml{\delta}}_\tau}  \div(J^{\sml{\delta}})\dif \Hff^n
        \end{align}
    Observe that $\ip{J|\nu^{\sml{p}}} = \abs{\nabla w_p}^{\alpha+p-1}$. Hence, the dominated convergence theorem gives
    \begin{equation}
        \lim_{\delta \to 0^+}\int_{\partial \Omega^{\sml{p}}_t} \ip{J^{\sml{\delta}}| \nu^{\sml{p}}} \dif \Hff^{n-1} = \int_{\partial \Omega^{\sml{p}}_t} \ip{J | \nu^{\sml{p}}} \dif \Hff^{n-1} =\ee^{- \left(\frac{\alpha}{n-p}-1\right)t}\IF_p(t).
    \end{equation}
    On the other hand
    \begin{equation}
         \abs{\div(J^{\sml{\delta}})}\leq\abs{\chi'_\delta(\abs{\nabla w_p}) \ip{\nabla \abs{\nabla w_p}|J
         } }+ \abs{\chi_\delta ( \abs{\nabla w_p})\div(J)}\leq  4(2\delta)^{\alpha+p-2} \abs{\nabla \abs{\nabla w_p}}+ \abs{\div(J)}
    \end{equation}
    which belongs to $L^2_{\loc}(D \smallsetminus \Omega)$ by \cref{prop:RegolaritaWp} and \cref{lem:divX}. By dominated convergence theorem, we can now pass to the limit as $\delta \to 0^+$ in \cref{eq:zzzdivergenceJcontinuity} to get
        \begin{align}
       \ee^{- \left(\frac{\alpha}{n-p}-1\right)t}\IF_p(t) - \ee^{- \left(\frac{\alpha}{n-p}-1\right)\tau}\IF_p(\tau) &= \int_{\Omega^{\sml{\delta}}_t \smallsetminus \Omega^{\sml{\delta}}_\tau}  \div(J)\dif \Hff^n,
        \end{align}
    which implies that $\IF_p(t)$ is continuous in $(0,T)$. Since $\div(J) \in L^2_{\loc}(D \smallsetminus \Omega)$ and $\div(J)$ vanishes almost everywhere on $\Crit(w_p)$, by coarea formula
    \begin{equation}
        \ee^{- \left(\frac{\alpha}{n-p}-1\right)t}\IF_p(t) - \ee^{- \left(\frac{\alpha}{n-p}-1\right)\tau}\IF_p(\tau)= \int_\tau^t \int_{\partial \Omega^{\sml{p}}_s} \frac{\div(J)}{\abs{\nabla w_p}} \dif \Hff^{n-1} \dif s
    \end{equation}
    and in particular $\IF_p$ belongs to $W^{1,1}_{\loc}(0,T)$. Finally, one can obtain the formula for the distributional derivative of $\IF_p$ from the computations in \cref{lem:divX} and the definition of $\MF_p$. Since  $\Ric(\nu, \nu) \in L^\infty(\Omega_T^{\sml{p}}\smallsetminus \Omega)$ and $\alpha + p-2 >0$, $\MF_p$ is well-defined and belongs to $L^1_{\loc}(0,T)$.
\end{proof}

\begin{remark}\label{rmk:regularity_G}
A posteriori, under the assumptions and notation of \cref{thm:monotonicity-NPT}, combining the regularity of the function $\MF_p$ stated in \cref{thm:monotonicity-NPT} with \cref{eq:defivata-di-G}, we actually infer that the function $\IF_p$ is of class $W^{2,1}(0,T)$, for $T < \inf_{\partial D}\Phi$, hence it admits a $C^1([0,T])$ representative.
\end{remark}

\subsection{Proof of \texorpdfstring{\cref{thm:monotonicity-NPT}}{Theorem 3.1}}
\pushQED{\qed}
In the previous section, we already proved that $\partial \Omega^{\sml{p}}_t$ is a curvature varifold for almost every $t \in [0,T)$. Hence, we are only going to prove that $\MF_p$ belongs to $W^{1,1}(0,T)$ and that \cref{eq:derivative-monotonicity_formula-NPT} holds. The proof is divided into two parts. In the first one, we prove \cref{eq:derivative-monotonicity_formula-NPT} for every $\psi \in \Lip_c(0,T)$, showing in particular that $\MF_p(t)$ admits a locally absolutely continuous monotone nondecreasing representative on $(0,T)$. This part of the proof is obtained by adapting \cite[Theorem 3.1]{benatti_minkowskiinequalitycompleteriemannian_2022}. In the second part we will prove that $\MF_p(0^+) \geq \MF_p(0)$, in particular inferring that $\MF_p \in W^{1,1}(0,T)$ and that \cref{eq:derivative-monotonicity_formula-NPT} holds. The case where $\partial \Omega$ is smooth will be a step in this last part of the proof.

\medskip

    Let $\delta>0$ and let $\chi_\delta$ be a smooth cut-off function satisfying 
    \begin{equation}
        \chi_\delta(t) \coloneqq\begin{cases}
            &\chi_\delta(t) =0& & \text{if $t \leq \delta$,}\\
            &0\leq \chi'_\delta(t) \leq 2/\delta&& \text{if $t \in [\delta, 2 \delta]$,}\\
            &\chi_\delta(t)=1& &\text{if $t \geq2 \delta$.}\\
        \end{cases}
    \end{equation}
    Moreover, let  
    \begin{equation}
        W \coloneqq \ee^{\left(\frac{\alpha}{n-p} - 1 \right) w_p}\left[ Y+ \left(\frac{p-1}{n-p} - \frac{1}{\alpha}  \right)J\right],
    \end{equation}
    and define $W^{\sml{\delta}} \coloneqq \chi_\delta(\abs{\nabla w_p }) W$. The vector fields $J$ and $Y$ are defined in \cref{eq:Xdef} and \cref{eq:defY} respectively.
\begin{step} Let $\psi \in \Lip_c(0,T)$.
    By the divergence theorem we have
    \begin{equation}\label{eq:zzdelta0divergencetheroem}
        - \int_{D \smallsetminus \Omega} \ip{\nabla [\psi(w_p)]| W^{\sml{\delta}}}\dif \Hff^{n} = \int_{D \smallsetminus \Omega} [\chi'_\delta (\abs{\nabla w_p })\ip{W| \nabla \abs{\nabla w_p }} + \chi_\delta(\abs{\nabla w_p }) \div(W) ]\psi(w_p)\dif \Hff^n,
    \end{equation}
    for any $\psi \in \Lip_c(0,T)$. Assume now that $\psi$ is nonnegative.  We now want to send $\delta \to 0^+$. It is easy to see that $\abs{\ip{W^{\sml{\delta}}|\nabla [\psi(w_p)]} }\leq\abs{\ip{W|\nabla [\psi(w_p)]}}$ and  $\ip{W^{\sml{\delta}}|\nabla [\psi(w_p)]} \to  \ip{W|\nabla [\psi(w_p)]}$ almost everywhere on $D \smallsetminus \Crit(w_p)$. Moreover, we have
        \begin{align}\label{eq:zzboundonWnablaw}
        \begin{split}
            \abs{W} &= \abs{\nabla w_p}^{\alpha+p-2} \abs{\left(\frac{p-1}{n-p}- \frac{1}{\alpha} \right) \nabla w_p +(p-1) \frac{\nabla \abs{\nabla w_p}}{\abs{\nabla w_p}}} \\
            &  \leq \kst  \abs{\nabla w_p}^{\alpha+ p-3} \left( \abs{\nabla w_p}^2 + \abs{\nabla \abs{\nabla w_p}}\right)\leq  \kst \abs{\nabla w_p}^{\alpha+p-2} \left( \abs{\nabla w_p} + \abs{\h}\right)
        \end{split}
        \end{align}
        by \cref{eq:MeanCurvatureLevelWP}. Hence, $\ip{W| \nabla [\psi(w_p)]}$ belongs to $L^2(D \smallsetminus \Crit(w_p))$ by \cref{lem:ConvergenzaVarifoldEspToZero} and \cref{rem:ConvergenzaVarifoldEspToZeroNONsmoothOmega}. The dominated convergence theorem yields
        \begin{equation}\label{eq:zzfirstorderdelta0}
            \int_{D \smallsetminus \Omega} \ip{\nabla [\psi(w_p)]| W^{\sml{\delta}}}\dif \Hff^n\xrightarrow{\delta \to 0^+}\int_{D \smallsetminus \Crit(w_p)} \ip{\nabla [\psi(w_p)]| W} \dif\Hff^n.
        \end{equation}
        \smallskip
         On the other hand, using \cref{lem:divX,lem:divY} and the expression of the mean curvature, one can easily check that
        \begin{equation}
            \div W  =  \ee^{\left(\frac{\alpha}{n-p} - 1 \right) w_p}\abs{\nabla w_p}^{\alpha+p-2} \left[\SQ_p+ \Ric\left(\frac{\nabla  w_p}{\abs{\nabla w_p}},\frac{\nabla w_p}{\abs{\nabla w_p}}\right)\right].
        \end{equation}
        Boundedness of $\abs{\nabla w_p}$, $\alpha+p-2>0$ and dominated convergence theorem also give
        \begin{align}\label{eq:zzriccitermdelta0}
        \begin{multlined}[c][.8\textwidth]
            \int_{D \smallsetminus \Omega} \chi_\delta (\abs{\nabla w_p})\ee^{\left(\frac{\alpha}{n-p}-1\right)w_p}\abs{\nabla w_p}^{\alpha+p-2} \Ric\left(\frac{\nabla  w_p}{\abs{\nabla w_p}},\frac{\nabla w_p}{\abs{\nabla w_p}}\right) \psi(w_p)\dif \Hff^n\\
            \xrightarrow{\delta \to 0^+}\int_{D \smallsetminus \Crit(w_p)} \ee^{\left(\frac{\alpha}{n-p}-1\right)w_p}\abs{\nabla w_p}^{\alpha+p-2} \Ric\left(\frac{\nabla  w_p}{\abs{\nabla w_p}},\frac{\nabla w_p}{\abs{\nabla w_p}}\right)\psi(w_p)\dif \Hff^n.
        \end{multlined}
        \end{align}
        Since $\SQ_p\geq 0$, monotone convergence implies
          \begin{align}\label{eq:zzSQlimitdelta0}
            \begin{multlined}[c][.8\textwidth]
            \int_{D \smallsetminus \Omega} \chi_\delta (\abs{\nabla w_p})\psi(w_p)\ee^{\left(\frac{\alpha}{n-p}-1\right)w_p}\abs{\nabla w_p}^{\alpha+p-2} \SQ_p\dif \Hff^n \\ \xrightarrow{\delta \to 0^+}\int_{D \smallsetminus \Crit(w_p)}\psi(w_p) \ee^{\left(\frac{\alpha}{n-p}-1\right)w_p}\abs{\nabla w_p}^{\alpha+p-2} \SQ_p\dif \Hff^n.
        \end{multlined}
        \end{align}
        \medskip
        
        We only need to show that the remaining term in \cref{eq:zzdelta0divergencetheroem} vanishes. Observe that $\chi'_\delta (\abs{\nabla w_p}) \abs{J} \leq \kst\delta^{\alpha+p-2}$. Hence, by H\"older's inequality and \cref{eq:MeanCurvatureLevelWP} we have 
        \begin{align}\label{eq:zzXintegravanishes}
        \begin{split}
        \begin{multlined}[t][.8\textwidth]
             \abs{\int_{D \smallsetminus \Omega}\ee^{\left( \frac{\alpha}{n-p}-1\right)w_p} \chi'_\delta (\abs{\nabla w_p })\ip{J| \nabla \abs{\nabla w_p }}\psi(w_p)\dif\Hff^n}^2\\
             \begin{aligned}
             &  \leq \ee^{2\left( \frac{\alpha}{n-p}-1\right)T} \left(\int_{D\smallsetminus \Omega}  (\chi'_\delta(\abs{\nabla w_p}))^2\abs{J}^2\dif\Hff^n\right)\left(\int_{D\smallsetminus \Omega} \psi(w_p)^2 \abs{\nabla  \abs{\nabla w_p}}^2\dif\Hff^n\right)\\
                 & \leq \kst \ee^{2\left( \frac{\alpha}{n-p}-1\right)T} \abs{D} \delta^{2(\alpha+ p -2)}\left(\int_{\Omega^{\sml{p}}_T\smallsetminus \Omega}  2\abs{\nabla w_p}^2 \abs{\h}^2 +\abs{\nabla w_p}^4 \dif\Hff^n\right)\xrightarrow{\delta \to 0^+}0,
            \end{aligned}
        \end{multlined}
        \end{split}
        \end{align}
        by \cref{lem:ConvergenzaVarifoldEspToZero} and \cref{rem:ConvergenzaVarifoldEspToZeroNONsmoothOmega}, since $\alpha +p -2 >0$. 

        \smallskip
     
        The remaining part requires some further analysis. Observe that $\psi\geq 0$ and $\ip{Y | \nabla \abs{\nabla w_p}}\geq 0$ on $D \smallsetminus \Crit(w_p)$. Moreover, by Sard's theorem $\set{\abs{\nabla w_p} =s}$ is smooth for almost every $s \in (\delta, 2 \delta)$ and the function $w_p$ is smooth where $\chi'_\delta(\abs{\nabla w_p})>0$. By coarea formula (see \cite[Proposition 3.7]{benatti_asymptoticbehaviourcapacitarypotentials_2022} for more details), we have
        \begin{equation}
            \int_{D \smallsetminus \Omega} \psi(w_p)\chi_\delta'(\abs{\nabla w_p}) \ip{Y|\nabla \abs{\nabla w_p}} \dif \Hff^n \leq \frac{2}{\delta}\int_{\delta}^{2 \delta} \int_{\set{\abs{\nabla w_p}=s}}\psi(w_p) \frac{\ip{Y | \nabla \abs{\nabla w_p}}}{\abs{\nabla \abs{\nabla w_p}}} \dif \Hff^{n-1}\dif s.
        \end{equation}
        Let $0<t$ be a regular value for $\abs{\nabla w_p}$ small enough so that $\set{\abs{\nabla w_p} < t} \cap \partial \Omega = \varnothing$, whose existence is granted by Hopf lemma. Define the function
        \begin{equation}
            N(s)\coloneqq \int_{\set{\abs{\nabla w_p}=s}}\psi(w_p) \frac{\ip{Y | \nabla \abs{\nabla w_p}}}{\abs{\nabla \abs{\nabla w_p}}} \dif \Hff^{n-1}
        \end{equation}
        for almost every $s \in (0,t)$ and observe that $N(s)\geq 0$. We claim that $N(s)\to 0$ as $s \to 0^+$ which is enough to conclude the proof. Indeed,
        \begin{equation}\label{eq:zzequationvanishes}
             \frac{2}{\delta}\int_{\delta}^{2 \delta} \int_{\set{\abs{\nabla w_p}=s}}\psi(w_p) \frac{\ip{Y | \nabla \abs{\nabla w_p}}}{\abs{\nabla \abs{\nabla w_p}}} \dif \Hff^{n-1}\dif s\leq 2 \sup_{s \in[\delta,2 \delta]} N(s) \xrightarrow{\delta \to 0^+} 0.
        \end{equation}
        By the divergence theorem, given $r\in(0,t)$ regular value for $\abs{\nabla w_p}$, we have 
        \begin{align}\label{eq:zzdivergencefunctionNimprovement}
            N(t)-N(r) &=\int_{\set{r \leq \abs{\nabla w_p}\leq t}}\div(\psi(w_p) Y) \dif \Hff^n  = \int_{\set{r \leq \abs{\nabla w_p}\leq t}}\psi'(w_p)\ip{\nabla w_p | Y} + \psi(w_p) \div Y\dif \Hff^n 
        \end{align}
        Since $\alpha > (n-p)/(n-1)$, by \cref{lem:divY} we have
        \begin{equation}\label{eq:zzdivergenceYforNimprovement}
            \div Y \geq \kst \frac{\ip{Y| \nabla \abs{\nabla w_p}}}{\abs{\nabla w_p}} + \abs{\nabla w_p}^{\alpha + p-2} \mathcal{D}_\sigma
        \end{equation}
        with $\mathcal{D}_\sigma \in L^2_{\loc}(D \smallsetminus \Omega)$. Since also $\ip{\nabla w_p| Y } \in L^2_{\loc}(D \smallsetminus \Omega) $, the function 
        \begin{equation}\label{eq:zzdivergenceforNimprovmentRdef}
             R(s) \coloneqq\int_{\set{0<\abs{\nabla w_p}\leq s}} \psi'(w_p)\ip{\nabla w_p | Y} + \psi(w_p)\abs{\nabla w_p}^{\alpha + p-2} \mathcal{D}_\sigma\dif \Hff^n  
        \end{equation}
        is a bounded continuous function on $[0,t]$ and $R(s)\to 0^+$ as $s\to 0$. Moreover, the function $N(s)-R(s)$ is monotone increasing since $\ip{Y|\nabla\abs{\nabla w_p}}\geq 0$. Therefore, $N(s)= N(s)-R(s) + R(s)$ admits a (finite) limit as $s\to 0^+$, being the sum of a monotone and a bounded continuous function. Extend $N(s)$ by continuity to $0^+$. Applying coarea formula to \cref{eq:zzdivergencefunctionNimprovement}, taking into account \cref{eq:zzdivergenceYforNimprovement,eq:zzdivergenceforNimprovmentRdef} and recalling the definition of $N(s)$, we get
        \begin{align}
            [N(t)-R(t)]- [N(0)-R(0)] &\geq \kst  \int_0^t \int_{\set{\abs{\nabla w_p}=s} }\psi(w_p) \frac{\ip{Y | \nabla \abs{\nabla w_p}}}{\abs{\nabla \abs{\nabla w_p}}} \dif \Hff^{n-1} \dif s = \kst \int_0^t \frac{N(s)}{s} \dif s.
        \end{align}
        If $N(0)\neq 0$, then $N(s)/s$ would not belong to $L^1(0,t)$ contradicting the boundedness of $N(t)-R(t)$.

        \smallskip 

        Sending $\delta \to 0^+$ and plugging \cref{eq:zzSQlimitdelta0,eq:zzriccitermdelta0,eq:zzXintegravanishes,eq:zzfirstorderdelta0,eq:zzequationvanishes} into \cref{eq:zzdelta0divergencetheroem} we get \cref{eq:derivative-monotonicity_formula-NPT} for every nonnegative $\psi\in \Lip_c(0,T)$. By \cref{lem:continuity_G}, $\MF_p \in L^1_{\loc}(0,T)$. Together with the just proved \cref{eq:derivative-monotonicity_formula-NPT} for nonnegative $\psi \in \Lip_c(0,T)$, it actually implies that $\abs{\nabla w_p}^{\alpha + p-2} \SQ_p \in L^1(\Omega^{\sml{p}}_T \smallsetminus \Omega)$. We conclude that \cref{eq:derivative-monotonicity_formula-NPT} holds for every $\psi \in \Lip_c(0,T)$ by splitting it into its positive and negative parts. In particular, $\MF_p$ belongs to $W^{1,1}_{\loc}(0,T)$ and admits a locally absolutely continuous monotone nondecreasing representative. 
\end{step}
\begin{step} We want to prove that $\MF_p(0^+) \geq \MF_p(0)$. If $\partial \Omega$ is smooth, by Hopf maximum principle we get that $w_p$ is smooth with nonnvanishing gradient in a neighborhood of $\partial \Omega$. In particular, the classical divergence theorem holds there both implying $\MF_p(0)= \MF_p(0^+)$ and that \cref{eq:derivative-monotonicity_formula-NPT} holds for every $\psi \in \Lip_c([0,T))$. 

    \medskip

     Let now $\Omega$ be as in the assumptions. Writing $\partial \Omega$ in a neighborhood of some $x \in \partial\Omega$ in a local chart as the graph of a $\CS^{1,\beta}$ function $f:\R^{n-1} \to \partial \Omega$ over the tangent space, one checks that $f$ solves 
    \begin{equation}
        \div \left( \frac{\nabla f }{\sqrt{1+ \abs{\nabla f}^2}}\right) = \H
    \end{equation}
    in the weak sense. Since $\H \in L^2(\partial\Omega)$, then $f \in W^{2,2}_{\loc}$. Hence, there exists a sequence of smooth closed sets $\Omega^{\sml{k}}\subset \Omega^\circ$ such that $\abs{\Omega^{\sml{k}} \smallsetminus \Omega} \to 0$,  $\partial \Omega^{\sml{k}} \to \partial \Omega$ in $\CS^{1, \beta'}$ for some $\beta'<\beta$, and  with mean curvature $\H^{\sml{k}}$ such that $\norm{\H^{\sml{k}}}_{L^2(\partial\Omega^{\sml{k}})} \to \norm{\H}_{L^2(\partial\Omega)}$. Since $\partial \Omega^{\sml{k}} \to \partial \Omega$ in $\CS^1$, up to extracting a subsequence, we can assume that $\vec{\H}^{\sml{k}}\Hff^{n-1}\res \partial \Omega^{\sml{k}} \to \vec{\H} \Hff^{n-1} \res \partial \Omega$ in duality with continuous vector fields.

    Let $w_p^{\sml{k}}$ be the solution to \cref{eq:eps_moser_potential} associated with $\Omega^{\sml{k}}$. We will use subscript or superscript $(k)$ to denote the quantities associated with $w^{\sml{k}}_p$. Applying \cref{eq:derivative-monotonicity_formula-NPT}, we have
    \begin{equation}\label{eq:zznonsmoothdivergencetheorem}
        - \int_0^T \psi'(t) \MF_p^{\sml{k}}(t) \dif t  - \psi(0)\MF^{\sml{k}}_p(0)=  \int_{D \smallsetminus \Omega^{\sml{k}}}\psi(w_p^{\sml{k}})\ee^{\left(\frac{\alpha}{n-p} - 1 \right)  w^{\sml{k}}_p}\abs{\nabla  w^{\sml{k}}_p}^{\alpha+p-2} \SQ^{\sml{k}}_p\dif \Hff^n.
    \end{equation}
    We aim to send $k\to +\infty$ in \cref{eq:zznonsmoothdivergencetheorem}. Since $\partial \Omega^{\sml{k}} \to \partial \Omega$ in $\CS^{1,\beta'}$, the functions $w^{\sml{k}}_p$ have uniformly bounded $\CS^{1,\beta''}$ norm on $ \set{0\leq w_p \leq T}$ for some $\beta''\le \beta'$, see \cite[Theorem 1]{lieberman_boundaryregularitysolutionsdegenerate_1988}. In particular, $w^{\sml{k}}_p \to w_p$ in $\CS^1(\Omega^{\sml{p}}_T \smallsetminus \Omega^\circ)$ as $k\to +\infty$. Hence, the construction of the approximating hypersurfaces gives for free that
    \begin{equation}\label{eq:zznonsmoothsecondconvergence}
    \abs{\MF_p^{\sml{k}}(0)- \MF_p(0)} \xrightarrow{k\to+\infty}0.
    \end{equation}
     Since the estimate in \cref{lem:ConvergenzaVarifoldEspToZero} is uniform with respect to $k$, $(W^{\sml{k}})_{k \in \N}$ are uniformly bounded in $L^2$. Moreover, we have
    \begin{equation}
        \abs{\abs{\nabla w_p^{\sml{k}}}^{\alpha+p-2} \Ric\left(\frac{\nabla w_p^{\sml{k}}}{\abs{\nabla w_p^{\sml{k}}}}, \frac{\nabla w_p^{\sml{k}}}{\abs{\nabla w_p^{\sml{k}}}}\right)} \leq \norm{\Ric}_{L^\infty(D)}\abs{\nabla w_p^{\sml{k}}}^{\alpha+p-2}.
    \end{equation}
    Hence, up to extracting a subsequence $\MF_p^{\sml{k}}\rightharpoonup \MF_p$ 
    \begin{equation}
        \int_0^T\MF_p^{\sml{k}}(t) \psi'(t) \dif t \xrightarrow{k \to +\infty} \int_0^T\MF_p(t) \psi'(t) \dif t,
    \end{equation}
    for every $\psi \in \Lip_c([0,T))$. Then, the left-hand side of \cref{eq:zznonsmoothdivergencetheorem} converges to the expected quantity.

    Consider now the right-hand side. Fix $\delta >0$. Schauder's estimates imply that $\chi_{\delta}(\abs{\nabla w_p^{\sml{k}}})\SQ^{\sml{k}}_p\to \chi_{\delta}(\abs{\nabla w_p})\SQ_p$ uniformly on any compact subset of $D\smallsetminus \Omega$ as $k\to +\infty$. By Fatou's lemma and $\SQ_p^{\sml{k}}\geq \chi_{\delta}(\abs{\nabla w_p^k})\SQ^{\sml{k}}_p$, one has
    \begin{equation}
    \begin{multlined}[c][.8\textwidth]
         \liminf_{k\to +\infty} \int_{D \smallsetminus \Omega^{\sml{k}}}\psi(w^{\sml{k}}_p) \ee^{\left(\frac{\alpha}{n-p} - 1 \right)  w^{\sml{k}}_p}\abs{\nabla  w^{\sml{k}}_p}^{\alpha+p-2} \SQ^{\sml{k}}_p\dif \Hff^n \\\geq\int_{D \smallsetminus \Omega}\psi(w_p)\ee^{\left(\frac{\alpha}{n-p} - 1 \right)  w_p}\abs{\nabla  w_p}^{\alpha+p-2} \chi_\delta(\abs{\nabla w_p})\SQ_p \dif \Hff^n
         \end{multlined}
    \end{equation}
    provided $\psi\in \Lip_c([0,T))$ is nonnegative. Sending $\delta \to 0^+$ with the aid of the monotone convergence theorem, we showed that $\abs{\nabla w_p}^{\alpha + p-2}\SQ_p$ is integrable on compact subsets of $D \smallsetminus \Omega^\circ$. In particular, $\MF'_p \in L^1(0,T)$ and arguing as in the previous step we can prove that \cref{eq:derivative-monotonicity_formula-NPT} holds. Since for every nonnegative $\psi\in\Lip_c(0,T)$ we also proved that 
    \begin{equation}
        -\int_0^T \psi'(t) \MF_p (t) \dif t  - \psi(0) \MF_p(0) \geq \int_0^T \MF'_p(t) \psi(t) \dif t =  -\int_0^T \psi'(t) \MF_p (t) \dif t  - \psi(0) \MF_p(0^+),
    \end{equation}
     we also have that $\MF_p(0) \leq \MF_p(0^+)$.\qedhere
     \end{step}

    \begin{remark}\label{rmk:casolimite}
        Here we discuss the possibility of taking $\alpha = \alpha_* \coloneqq (n-p)/(n-1)$ (see also \cite{benatti_monotonicityformulasnonlinearpotential_2022}) in \cref{thm:monotonicity-NPT}, by passing \cref{eq:derivative-monotonicity_formula-NPT} to the limit as $\alpha\to  \alpha_*^+$. For any $\alpha\geq\alpha_*$, we now denote $\MF^{\sml{\alpha}}_p$ and $\SQ^{\sml{\alpha}}_p$ the functions defined in \cref{eq:monotonicty_formula} and \cref{eq:Qpdef} respectively. Observe that
        \begin{equation}
            \SQ^{\sml{\alpha}}_p \geq [\alpha-(2-p)]\frac{\abs{\nabla^\top\abs{\nabla w_p}}^2}{\abs{\nabla w_p}^2} + \abs{\mathring{\h}}^2\xrightarrow{\alpha \to \alpha^+_*} \frac{(n-2)(p-1)}{(n-1)} \frac{\abs{\nabla^\top\abs{\nabla w_p}}^2}{\abs{\nabla w_p}^2} + \abs{\mathring{\h}}^2 = \SQ^{\sml{\alpha_*}}_p
        \end{equation}
        holds almost everywhere on $\Omega^{\sml{p}}_T\smallsetminus \Omega$. Therefore, we can apply Fatou's lemma to the right-hand side of \cref{eq:derivative-monotonicity_formula-NPT}, for every nonnegative $\psi \in \Lip_c([0,T))$. On the other hand, by dominated convergence one has
        \begin{equation}
            \lim_{\alpha\to \alpha_*^+} \int_0^T \psi' (t)  \MF^{\sml{\alpha}}_{p}(t) \dif t + \psi(0) \MF^{\sml{\alpha}}_p(0) = \int_0^T \psi' (t)  \MF^{\sml{\alpha_*}}_{p}(t) \dif t + \psi(0) \MF^{\sml{\alpha_*}}_p(0).
        \end{equation}
        Hence, \cref{eq:derivative-monotonicity_formula-NPT} passes to the limit as $\alpha \to \alpha_*^+$. The same approach will be used to deduce the case $\alpha =1$ in \cref{thm:monotonicty_formula_IMCF} once we will have proved that \cref{eq:derivative-monotonicity_formula-IMCF} holds for $\alpha >1$.
    \end{remark}

\section{Monotonicity formula for the IMCF}\label{sec:IMCF}

This section is devoted to proving monotonicity formulas along the weak IMCF. Let $w_1$ be a locally Lipschitz solution to \cref{eq:IMCF} on $M$. Assume there exists $\overline{T}>0$ such that $\Omega^{\sml{1}}_t \Subset M$ for every $t \in [0,\overline{T})$. We then define the function $\MF_1$ as
\begin{equation}\label{eq:monotonicty_formula_IMCF}
    \MF_1(t)\coloneqq -\frac{1}{\alpha}\ee^{\left(\frac{\alpha}{n-1} -1\right)t} \int_{\partial \Omega^{\sml{1}}_t }\abs{\nabla w_1}^{\alpha}\dif \Hff^{n-1} - \int_0^t\ee^{\left(\frac{\alpha}{n-1} -1\right)s}\int_{\partial \Omega^{\sml{1}}_s} \abs{\nabla w_1}^{\alpha-2} \Ric(\nu,\nu) \dif \Hff^{n-1} \dif s,
\end{equation}
for almost every $t \in (0,\overline{T})$, and by
\begin{equation}
    \MF_1(0) \coloneqq - \frac{1}{\alpha} \int_{\partial \Omega} \abs{\H}^\alpha \dif \Hff^{n-1}.
\end{equation}
Using coarea formula as in \cref{sec:well-posedness}, the function $\MF_1$ is a well-defined function in $L^1_{\loc}(0,\overline{T})$ for $\alpha\geq 1 $. Our aim is to prove the following statement.

\begin{theorem}\label{thm:monotonicty_formula_IMCF}
Let  $\alpha \geq 1$. Let $\Omega \subset M$ be a compact set with $\CS^{1,1}$-boundary. Let $w_1$ be a locally Lipschitz solution to \cref{eq:IMCF} on $M$.  Assume there exists $\overline{T}>0$ such that $\Omega^{\sml{1}}_t \Subset M$ for every $t \in [0,\overline{T})$. Then, the following hold
\begin{enumerate}
    \item $\partial \Omega^{\sml{1}}_t$ is a curvature varifold with square integrable second fundamental form for almost every $t \in [0,\overline{T})$; 
    \item the function $\MF_{1}$ defined in \cref{eq:monotonicty_formula_IMCF} is essentially nondecreasing; and
    \item for every nonnegative $\psi \in \Lip_c([0,\overline{T}))$ we have 
    \begin{equation}\label{eq:derivative-monotonicity_formula-IMCF}
            - \int_0^{+\infty} \psi' (t)  \MF_{1}(t) \dif t - \psi(0) \MF_{1}(0)\geq  \int_0^{+\infty} \psi(t) \ee^{\left(\frac{\alpha}{n-1} - 1 \right)t} \int_{\partial \Omega^{\sml{1}}_t}\H^{\alpha-2} \SQ_{1}\dif \Hff^{n-1} \dif t,
        \end{equation}
  where 
    \begin{equation}\label{eq:Q1def}
    \SQ_{1}= \begin{multlined}[t]
        (\alpha-1) \frac{\abs{\nabla^\top\H}^2}{\H^2} + \abs{ \mathring{\h}}^2 
    \end{multlined}
    \end{equation}
    on $M \smallsetminus \Crit(w_1)$ and $\SQ_1 \coloneqq0$ otherwise.
\end{enumerate}
\end{theorem}

\begin{remark}\label{rmk:equivalent_formulation_monotonicity_IMCF}
    By the very definition of weak IMCF in \cite{huisken_inversemeancurvatureflow_2001}, the function $w_1$ minimizes the functional
    \begin{equation}
        J_{w_1}(v) = \int_{K} \abs{\nabla v}+ v\abs{\nabla w_1} \dif \Hff^n
    \end{equation}
    among all locally Lipschitz functions $v$ such that $\set{v \neq w_1}\subseteq K$ for a compact $K$ in $M$. Take then any vector field $X$ compactly supported in $K \subseteq D$ and let $\Theta_s:M\to M$ the family of diffeomorphisms generated by it. Then,
    \begin{align}
        0 &= \left.\frac{\dd}{\dd s}\right\vert_{s=0} J_{w_1}(w_1 \circ \Theta_s) = \left.\frac{\dd}{\dd s}\right\vert_{s=0}\int_\R \int_{\partial \Omega^{\sml{1}}_t\cap K} \abs{\det \nabla^\top\Theta_{-s}} \dif \Hff^{n-1}\dif t +\int_{K} w_1(\Theta_s) \abs{\nabla w_1}\dif \Hff^n\\
        &=\int_\R\int_{\partial \Omega^{\sml{1}}_t\cap K} -\div_{\top} X + \ip{X|\nabla w_1} \dif \Hff^{n-1} \dif t \overset{\cref{eq:MeanCurvVarifoldManifoldIntegrationbyParts}}{=} \int_\R\int_{\partial \Omega^{\sml{1}}_t\cap K}  \ip{X|\nabla w_1 + \vec{\H}} \dif \Hff^{n-1} \dif t.
    \end{align}
    In \cite{huisken_inversemeancurvatureflow_2001}, the last identity is the definition of the weak mean curvature. Here, we already have a notion of mean curvature for which \cref{eq:MeanCurvVarifoldManifoldIntegrationbyParts} holds since $\partial \Omega^{\sml{1}}_t$ is a curvature varifold. We conclude that $\vec{\H}= -\nabla w_1$ (and \cref{eq:MeanCurvatureLevelW1}) holds almost everywhere on $\partial \Omega^{\sml{1}}_t$ for almost every $t \in [0,\overline{T})$. In particular, the functional $\MF_1$ can be equivalently rewritten as
    \begin{equation}
        \MF_1(t)\coloneqq -\frac{1}{\alpha}\ee^{\left(\frac{\alpha}{n-1} -1\right)t} \int_{\partial \Omega^{\sml{1}}_t }\abs{\H}^{\alpha}\dif \Hff^{n-1} - \int_0^t\ee^{\left(\frac{\alpha}{n-1} -1\right)s}\int_{\partial \Omega^{\sml{1}}_s} \abs{\H}^{\alpha-2} \Ric(\nu,\nu) \dif \Hff^{n-1} \dif s.
    \end{equation}
\end{remark}

This family of monotonicity formulas, for different values of $\alpha$, contains both \cite[$\H^2$ Growth Formula 5.7]{huisken_inversemeancurvatureflow_2001} and the limit case of monotonicity formulas \cite[Theorem 3.1]{benatti_minkowskiinequalitycompleteriemannian_2022}, \cite[Theorem 1.1]{agostiniani_riemannianpenroseinequalitynonlinear_2022} and \cite[Theorem 1.1]{xia_newmonotonicitycapacitaryfunctions_2024}. 

Rather than using the approximation via $\varepsilon$-regularization as in \cite{huisken_inversemeancurvatureflow_2001}, we follow the spirit of these last papers. We prove \cref{thm:monotonicty_formula_IMCF} by sending $p\to 1^+$ in \cref{thm:monotonicity-NPT}.

The convergence result stated in \cref{thm:local-approximation} is not sufficient to conclude. Indeed, whether $\abs{\nabla w_p}$ converges to  $\H = \abs{\nabla w_1}$ is not clear at this level. Our first aim is to upgrade the uniform convergence of $w_p$ to $w_1$ to a convergence in $W^{1,q}$, for every $q <+ \infty$.

\subsection{Improved convergence to the weak IMCF}
In the first place, we aim to prove a regularity result for the level sets of the weak IMCF. We want to mirror \cref{lem:ConvergenzaVarifoldEspToZero} in the limit as $p\to 1^+$. There are two main ingredients in the proof. The first one is the convergence of the area of level sets. Here, we can not rely on the convergence of the gradient as in the limit $\varepsilon \to 0^+$. We will achieve the result by exploiting an energy argument which involves the convergence of the $p$-capacities.

\begin{lemma}\label{thm:area_convergence}
   Assume $\inf_{\partial D}w_1 >0$ in \cref{thm:local-approximation}. Then, for every $0<T < \inf_{\partial D} w_1$ we have
    \begin{equation}\label{eq:areaconv}
        \lim_{p\to 1^+}\int_0^T \abs{\abs{\partial \Omega^{\sml{p}}_t} -\abs{\partial \Omega^{\sml{1}}_t}} \dif t = 0.
    \end{equation}
    In particular, we also have
    \begin{equation}\label{eq:unitp-unit1-conv}
        \lim_{p\to 1^+}\int_0^T \int_{\partial \Omega^{\sml{p}}_t} \abs{\nu^{\sml{p}}- \nu^{\sml{1}}}^2 \dif \Hff^{n-1} \dif t=0,
    \end{equation}
    where $\nu^{\sml{p}}$ and $\nu^{\sml{1}}$ represent the unit normal vector field to $w_{p}$ and $w_1$ respectively.
\end{lemma}
\begin{proof}
    By coarea formula
    \begin{equation}
        \abs{\int_{\Omega^{\sml{p}}_T\smallsetminus \Omega} \abs{\nabla w_p}\dif \Hff^n - \int_{\Omega^{\sml{1}}_T\smallsetminus \Omega} \abs{\nabla w_1} \dif \Hff^n} \leq \int_0^T \abs{\abs{\partial \Omega^{\sml{p}}_t} -\abs{\partial \Omega^{\sml{1}}_t}} \dif t.
    \end{equation}
    Hence, \cref{eq:unitp-unit1-conv} follows from \cref{lem:convergence_of_unit_normal}, once we prove \cref{eq:areaconv}.
    By \cref{lem:continuity_p-capacity}, we have $\ncapa_p(\Omega, \Omega^{\sml{p}}_T)\to \abs{\partial \Omega^*}/\abs{\S^{n-1}}$. In particular, recalling that $\abs{\partial \Omega_t^{\sml{1}}} = \ee^t \abs{\partial \Omega^*}$ we have 
    \begin{equation}
        \lim_{p \to 1^+}\int_0^T\abs{ \abs{\partial \Omega_t^{\sml{1}}}- (n-p)^{p-1} \abs{\S^{n-1}}\ee^t\ncapa_p(\Omega, U^{\sml{p}}_T)\left(1- \ee^{-\frac{T}{p-1}}\right)^{p-1}}\dif t = 0.
    \end{equation}
    Hence, the conclusion follows if we prove that
    \begin{equation}
        \lim_{p \to 1^+}\int_0^T\abs{ \abs{\partial \Omega^{\sml{p}}_t}- (n-p)^{p-1} \abs{\S^{n-1}} \ee^t\ncapa_p(\Omega, U^{\sml{p}}_T)\left(1- \ee^{-\frac{T}{p-1}}\right)^{p-1}}\dif t =0.
    \end{equation}
    By coarea formula, we have
    \begin{equation}
        \int_0^{T} \int_{\partial \Omega^{\sml{p}}_t} \abs{1 - \abs{ \nabla w_p}^{p-1}}\dif\Hff^{n-1}\dif t = \int_{\Omega^{\sml{p}}_{T}\smallsetminus\Omega} \abs{ \abs{\nabla w_p }-\abs{\nabla w_p}^p} \dif \Hff^n
    \end{equation}
    $\abs{ \abs{\nabla w_p }-\abs{\nabla w_p}^p} \leq \kst$ and $t\mapsto\abs{\abs{t}-\abs{t}^p}$ converges to zero uniformly on compact sets as $p\to 1^+$. Dominated convergence theorem yields
    \begin{equation}
         \lim_{p \to 1^+}\int_0^{T} \int_{\partial \Omega^{\sml{p}}_t} \abs{1 - \abs{ \nabla w_p}^{p-1}}\dif \Hff^{n-1}\dif t  = 0.
    \end{equation}
    Since by \cref{lem:p-cap_levelset}
    \begin{equation}
        \abs{\abs{\partial \Omega_t^{\sml{p}}}- (n-p)^{p-1} \abs{\S^{n-1}}\left(1- \ee^{-\frac{T}{p-1}}\right)^{p-1} \ee^t\ncapa_p(\Omega, U^{\sml{p}}_T)} \leq \int_{\partial \Omega^{\sml{p}}_t} \abs{1 - \abs{ \nabla w_p}^{p-1}}\dif \Hff^{n-1},
    \end{equation}
    we deduce \cref{eq:areaconv}.   \qedhere
\end{proof}

The second ingredient is the upper bound for the $L^2$ norm of the second fundamental form of the level sets of $w_p$. We can not directly use \cref{eq:interior_esitmate_2ff_p}. Indeed, the constant appearing in that estimate might blow up in the limit as $p\to 1^+$. However, the monotonicity formula \cref{eq:derivative-monotonicity_formula-NPT} suggests that a sharp bound on $\MF_p(t)$ is what we actually aim for. We now prove an $L^1$ upper bound for $\MF_p(t)$ which does not depend on $p$.

\begin{lemma}\label{prop:boundedness_Fp}
Assume $\inf_{\partial D}w_1 >0$ in \cref{thm:local-approximation}.  Let $\alpha^*, \alpha$ such that $(n-p)/(n-1)< \alpha < \alpha^*$. Given $\MF_p$ in \cref{eq:monotonicty_formula}, for every $0<T < \inf_{\partial D} w_1$ we have
    \begin{equation}\label{eq:boundedness_Fp}
        \int_0^{T}\abs{\MF_p(t)}\dif t + \abs{\MF_p(0)}\leq \kst \left(\norm{\nabla w_p}_{L^\infty(\Omega^{\sml{p}}_{T} \smallsetminus \Omega^\circ)}^{\alpha^*+p-1}+1\right)\left(\norm{\H}^2_{L^2(\partial \Omega)} + \abs{\partial \Omega} +1\right)
    \end{equation}
    for a constant $\kst>0$ depending only on $n$, $\alpha^*$, $T$, $D$ and any lower bound on the Ricci tensor on $D$. 
\end{lemma}

\begin{proof}
    Consider the function $\IF_p$ defined in \cref{eq:pezzo-di-gradiente}. By \cref{cor:energy_boundedness}, we have
    \begin{equation}
        \IF_p(t)\leq \ee^{\left(\frac{\alpha}{n-p}-1\right)t}\norm{\nabla w_p}_{L^\infty(\Omega_T^{\sml{p}}\smallsetminus \Omega^\circ)}^\alpha \int_{\partial \Omega_t^{\sml{p}}} \abs{\nabla w_p}^{p-1} \dif \Hff^{n-1}\leq \kst (T) \ee ^{\frac{\alpha}{n-p}T} \norm{\nabla w_p}_{L^\infty(\Omega_T^{\sml{p}}\smallsetminus \Omega^\circ)}^\alpha ,
    \end{equation}
    which is uniformly bounded by \cref{prop:uniform_gradient}. The function $\IF_p$ is $W^{2,1}(0,T)$ by \cref{rmk:regularity_G}, and its derivative in \cref{eq:defivata-di-G} is bounded by
    \begin{equation}\label{eq:zzzderivativeestimateGp}
        \IF^{\prime}_p(t)\geq \frac{\alpha \MF_p(t)}{p-1} - \frac{\alpha \kappa}{p-1} \ee^{\left(\frac{\alpha}{n-p}+1\right) T} \abs{D} \norm{\nabla w_p}^{\alpha+p-2}_{L^\infty(\Omega^{\sml{p}}_T \smallsetminus \Omega^\circ)}
    \end{equation}
    where $\kappa$ is the lower bound on $\Ric$ on $D$. Then the function $s \mapsto \MF_p(s) - \MF_p(0)$ is nonnegative almost everywhere on $[0,T]$. We get
    \begin{align}
        \int_0^{T}\abs{\MF_p(t)}\dif t &\leq \int_0^{T}\abs{\MF_p(t)- \MF_p(0)}\dif t + T\abs{\MF_p(0)}
        = \int_0^{T}\MF_p(t)- \MF_p(0)\dif t  + T\abs{\MF_p(0)}\\
        \overset{\cref{eq:zzzderivativeestimateGp}}&{\leq} \frac{p-1}{\alpha}\left(\IF_p(T)- \IF_p(0)\right)+ 2 \abs{\MF_p(0)}T+\kappa\ee^{\left(\frac{\alpha}{n-p}+1\right) T}  \abs{D} \norm{\nabla w_p}^{\alpha+p-2}_{L^\infty(\Omega^{\sml{p}}_T \smallsetminus \Omega^\circ)}\\
        & \leq 2 T \abs{\MF_p(0)} + \kst(T) \left(\norm{\nabla w_p}_{L^\infty(\Omega^{\sml{p}}_T \smallsetminus \Omega^\circ)}^{\alpha^*+p-1}+1\right)
    \end{align}
    for a constant with the same dependencies of the proposition. It thus only remains to show that we can control $\MF_p(0)$. The control follows by H\"older's and Young's inequality since 
    \begin{align}
        &\abs{\MF_p(0)} \leq \kst \IF_p(0)  + \int_{\partial \Omega} \abs{\nabla w_p}^{\alpha+p-2}\abs{\H} \dif \Hff^{n-1}\leq \kst \left(\norm{\nabla w_p}_{L^\infty(\partial \Omega)}^{\alpha^*+p-1}+1\right)\left( \norm{\H}_{L^2(\partial \Omega)}^2 + \abs{\partial \Omega} +1 \right)
    \end{align}
    for a constant with the same dependencies of the statement, thus implying \cref{eq:boundedness_Fp}.
    \end{proof}

    Whenever $\Omega \subseteq M$ is a compact subset with $\CS^{1,1}$-boundary, the above lemma comes with two direct corollaries. The first one is the above-mentioned estimate of the second fundamental form. The second one is the more interesting one. It states that the norm of the gradients of the functions $w_p$ becomes more and more similar to the mean curvature of the level sets as $p\to 1^+$. This behavior is expected since in the limit these two notions coincide.
    
    \begin{corollary}\label{thm:equibounded_2ff_NPT}
        Under the same hypotheses of \cref{prop:boundedness_Fp}, there exists a constant $\kst>0$ depending only on $n$, $\alpha^*$, $T$, $\inf_{\partial D} w_1$, $D$ and the lower bound on the Ricci tensor on $D$ such that
        \begin{equation}\label{eq:boundedness_h}
           \left(\alpha -\frac{n-p}{n-1}\right) \int_0^T \int_{\partial \Omega^{\sml{p}}_t}\abs{\nabla w_p}^{\alpha +p-3}\abs{\h^{\sml{p}}}^2 \dif \Hff^{n-1} \dif t  \leq \kst
        \end{equation}      
    \end{corollary}

    \begin{proof}
     Denote $\tau \coloneqq \frac{1}{3}(\inf_{\partial D} w_1 - T)$. By Young's inequality and the expression of $\SQ_p$  in \cref{eq:Qpdef}, we have     
    \begin{align}
         \left(\alpha - \frac{n-p}{n-1}\right)\abs{\h^{\sml{p}}}^2 &= \left(\alpha - \frac{n-p}{n-1}\right)\left(\abs{\mathring{\h}^{\sml{p}}}^2 + \frac{(\H^{\sml{p}})^2}{n-1}\right)\\&\leq(1+ \alpha^*)\abs{\mathring{\h}^{\sml{p}}}^2 +  \frac{1+\alpha^*}{p-1} \left(\alpha- \frac{n-p}{n-1}\right)\left(\H^{\sml{p}}- \frac{n-1}{n-p} \abs{\nabla w_p}\right)^2 + \kst\abs{\nabla w_p}^2\\& \leq  (1+\alpha^*)\ee^{T+2 \tau}\ee^{\left(\frac{\alpha}{n-p}-1\right)w_p}\SQ_p + \kst \abs{\nabla w_p}^2,
    \end{align}
    for a constant $\kst >0$ depending only on $n$. Take $\psi$ such that $\psi(t)=1$ for $t\leq T$, $\psi(t)=0$ for $t\geq T+\tau$ and $\abs{\psi'(t)} \leq 8 /\tau$. Plugging $\psi$ into \cref{eq:derivative-monotonicity_formula-NPT}, we obtain
    \begin{multline}
        \left(\alpha -\frac{n-p}{n-1}\right)\int_0^T \int_{\partial \Omega_t^{\sml{p}}}\abs{\nabla w_p}^{\alpha +p-3} \abs{\h^{\sml{p}}}^2 \dif \Hff^{n-1} \dif t \\
        \leq \kst\left(\int_0^{T+\tau}\abs{\MF_p(t)}\dif t + \abs{\MF_p(0)} +\abs{D}\,\norm{\nabla w_p}^{\alpha +p }_{L^\infty(\Omega^{\sml{p}}_{T + \tau} \smallsetminus \Omega)} \right),
    \end{multline}
    from which \cref{eq:boundedness_h} readily follows by \cref{prop:boundedness_Fp} and \cref{prop:uniform_gradient}.
\end{proof}

\begin{corollary}\label{cor:equibounded_gradient_normgradient}
        Under the same hypotheses of \cref{prop:boundedness_Fp}, there exists a constant $\kst>0$ depending only on $n$, $\alpha^*$, $T$, $\inf_{\partial D} w_1$, $D$ and the lower bound on the Ricci tensor on $D$ such that
        \begin{align}\label{eq:gradient}
            \left( \alpha-\frac{n-p}{n-1}\right) \int_0^T \int_{\partial \Omega^{\sml{p}}_t}\abs{\nabla w_p}^{\alpha +p-3}\frac{\abs{\nabla\abs{\nabla w_p}}^2}{\abs{\nabla w_p}^2}\dif \Hff^{n-1} \dif t  \leq \frac{\kst}{p -1}
            \end{align}
            and
        \begin{align}
            \int_0^T \int_{\partial \Omega^{\sml{p}}_t}\abs{\nabla w_p}^{\alpha +p-3}\frac{\abs{\nabla^\top\abs{\nabla w_p}}^2}{\abs{\nabla w_p}^2}\dif \Hff^{n-1} \dif t  \leq {\kst}.
        \end{align}
    \end{corollary}

    \begin{proof} Denote $\tau \coloneqq \frac{1}{3}(\inf_{\partial D} w_1 - T)$. The expression of $\SQ_p$ in \cref{eq:Qpdef} gives
     \begin{equation}
         (\alpha -(2-p))\frac{\abs{\nabla^\top \abs{\nabla w_p}}^2}{\abs{\nabla w_p}^2} \leq\ee^{T+2 \tau}\ee^{\left(\frac{\alpha}{n-p}-1\right)w_p}\SQ_p.
     \end{equation}
    On the other hand, exploiting also the expression of $\H^{\sml{p}}$ in \cref{eq:MeanCurvatureLevelWP} and using Young's inequality, one has
     \begin{align}
        (p-1)\left( \alpha-\frac{n-p}{n-1}\right) \frac{\abs{\nabla^\perp \abs{\nabla w_p}}^2}{\abs{\nabla w_p}}&\leq \frac{1}{p-1}\left( \alpha-\frac{n-p}{n-1}\right)\left(\H^{\sml{p}}- \frac{n-1}{n-p} \abs{\nabla w_p}\right)^2 +\kst\abs{\nabla w_p}^2\\&\leq  \ee^{T+2 \tau}\ee^{\left(\frac{\alpha}{n-p}-1\right)w_p}\SQ_p+\kst\abs{\nabla w_p}^2.
     \end{align}
     Now the proof follows as in \cref{thm:equibounded_2ff_NPT}, replacing the estimate of $\abs{\h^{\sml{p}}}$ with the inequalities above.
\end{proof}
    
    \begin{corollary}\label{thm:vanishing_gradient_mean_curvature_NPT}
    Under the same hypotheses of \cref{prop:boundedness_Fp}, we have
    \begin{equation}
            \lim_{p\to 1^+} \int_0^T \int_{\partial \Omega^{\sml{p}}_t} (\H^{\sml{p}}- \abs{\nabla w_p})^2 \dif \Hff^{n-1} \dif t =0.
    \end{equation}
    \end{corollary}

    \begin{proof}
    Denote $\tau \coloneqq \frac{1}{3}(\inf_{\partial D} w_1 - T)$. By Young's inequality and the expression of $\SQ_p$ for $\alpha =3-p$ in \cref{eq:Qpdef}, we have    
         \begin{align}
             \left( \H^{\sml{p}} - \abs{\nabla w_p}\right)^2&\leq \kst\left[\left(3-p - \frac{n-p}{n-1}\right)\left( \H^{\sml{p}} - \frac{n-1}{n-p}\abs{\nabla w_p}\right)^2 + \frac{(p-1)^2}{(n-p)^2}\abs{\nabla w_p}^2\right]\\&\leq \kst\left[(p-1)\ee^{T+2 \tau} \ee^{-\frac{n-3}{n-p} w_p} \SQ_p + \frac{(p-1)^2}{(n-p)^2} \abs{\nabla w_p}^2\right],
         \end{align}
         for a constant depending only on $n$.  Take $\psi$ such that $\psi(t)=1$ for $t\leq T$, $\psi(t)=0$ for $t\geq T+\tau$ and $\abs{\psi'(t)} \leq 8 /\tau$. Plugging $\psi$ into \cref{eq:derivative-monotonicity_formula-NPT}, we obtain
         \begin{equation}
             \int_0^T \int_{\partial \Omega^{\sml{p}}_t} (\H^{\sml{p}}- \abs{\nabla w_p})^2 \dif \Hff^{n-1} \dif t \leq \kst (p-1)\left(\int_0^{T+\tau}\abs{\MF_p(t)}\dif t + \abs{\MF_p(0)} +\abs{D}\,\norm{\nabla w_p}^{2}_{L^\infty(\Omega^{\sml{p}}_{T + \tau} \smallsetminus \Omega)}\right),
         \end{equation}
         which vanishes as $p\to 1^+$ by \cref{prop:boundedness_Fp} and \cref{prop:uniform_gradient}.
    \end{proof}

    The next result addresses the first part of \cref{thm:monotonicty_formula_IMCF}, in the same way \cref{lem:ConvergenzaVarifoldEspToZero} does for \cref{thm:monotonicity-NPT}.
    \begin{theorem}\label{cor:varifold_convergence_IMCF}
        Under the same hypotheses of \cref{prop:boundedness_Fp}, for any sequence of $p$'s converging to $1^+$ there exists a subsequence $(p_k)_{k \in \N}$, such that 
        \begin{enumerate}
            \item $\partial \Omega^{\sml{p_k}}_t$ converges in the sense of varifolds to $\partial \Omega^{\sml{1}}_t$ as $k\to +\infty$, for almost every $t \in[0,T]$;

            \item $\partial \Omega^{\sml{1}}_t$ is a curvature varifold satisfying all properties in \cref{lem:ConditionsEquivalentSFF} for almost every $t \in[0,T]$;

            \item the generalized second fundamental form of $\partial \Omega^{\sml{p_k}}_t$ weakly converges to that of $\partial \Omega^{\sml{1}}_t$, for almost every $t \in[0,T]$, in the sense of \eqref{eq:ConvDeboleSFFmanifoldgenerale};

            \item the generalized second fundamental form of $\partial \Omega^{\sml{1}}_t$ satisfies \begin{equation}\label{eq:boundedL2weak2FFIMCF}
            \int_0^T \int_{\partial \Omega^{\sml{1}}_t}\abs{\h^{\sml{1}}}^2 \dif \Hff^{n-1} \dif t\leq \kst,
        \end{equation}
            for $\kst>0$ depending on $n$, $T$, $\inf_{\partial D} w_1$, $D$ and the lower bound on the Ricci tensor on $D$.
        \end{enumerate}
    \end{theorem}

    \begin{proof}
        By \cref{thm:area_convergence}, we can assume that $\abs{\partial \Omega^{\sml{p_k}}_t} \to \abs{\partial \Omega^{\sml{1}}_t}$ for almost every $t \in [0,T]$. Moreover, we have \begin{equation}\label{eq:zzbounded_mean_curvature}
        \int_{\partial\Omega^{\sml{p_k}}_t} \abs{\H^{\sml{p_k}}}^2 \dif \Hff^{n-1} \le 
        2\int_{\partial\Omega^{\sml{p_k}}_t}    
        \left( { \H}^{\sml{p_k}} - \abs{ \nabla w_{p_k}}\right)^2 \dif \Hff^{n-1} + 2\int_{\partial\Omega^{\sml{p_k}}_t} {\abs{ \nabla w_{p_k}}^2} \dif \Hff^{n-1}
    \end{equation}
    which is bounded for almost every $ t \in[0,T]$ by \cref{thm:vanishing_gradient_mean_curvature_NPT,prop:uniform_gradient,cor:energy_boundedness}. By \cref{cor:ConvergenzaLevelSetsVersusVarifolds}, $\partial\Omega^{\sml{p_k}}_t$ converges to $\partial\Omega^{\sml{1}}_t$ in the sense of varifolds for almost every $t \in [0,T]$. 

    By \cref{thm:equibounded_2ff_NPT} for $\alpha(p_k) =3-p_k$ and Fatou's lemma 
    \begin{equation}\label{eq:zzlscFatouStarting}
        \int_0^T\liminf_{k\to +\infty}\int_{\partial \Omega^{\sml{p_k}}_t} \abs{\h^{\sml{p_k}}}^2 \dif \Hff^{n-1} \dif t \leq \liminf_{k\to +\infty}\int_0^T\int_{\partial \Omega^{\sml{p_k}}_t} \abs{\h^{\sml{p_k}}}^2 \dif \Hff^{n-1} \dif t \leq \kst,
    \end{equation}
    for a constant $\kst >0$ independent of $p$. Fix $t \in [0,T]$ such that $\partial \Omega^{\sml{p_k}}_t $ converges to $\partial \Omega^{\sml{1}}$ in the sense of varifold and
     \begin{equation}\label{eq:zzzliminfbounded2ff}
         \liminf_{k \to +\infty} \int_{\partial \Omega^{\sml{p_k}}_t} \abs{\h^{\sml{p_k}}}^2 \dif \Hff^{n-1}<+\infty.
     \end{equation}
     Up to passing to a subsequence depending on $t$, we can assume that the inferior limit in \cref{eq:zzzliminfbounded2ff} is realized. Hence, $\partial \Omega^{\sml{1}}_t$ is a curvature varifold by \cref{thm:ProprietaConvergenzeVarifold}. Since by \cref{lem:ConvergenzaVarifoldEspToZero} and \cref{rem:ConvergenzaVarifoldEspToZeroNONsmoothOmega} we can assume that $\partial \Omega^{\sml{p_k}}_t$ satisfies all properties in \cref{lem:ConditionsEquivalentSFF} for every $k$, $\partial \Omega_t^{\sml{1}}$ satisfies them as well for almost every $t \in [0,T]$. Moreover,
     \begin{equation}
         \int_{\partial \Omega^{\sml{1}}_t} \abs{\h^{\sml{1}}}^2 \dif \Hff^{n-1} \leq \liminf_{k \to +\infty} \int_{\partial \Omega^{\sml{p_k}}_t} \abs{\h^{\sml{p_k}}}^2 \dif \Hff^{n-1} 
     \end{equation}
     that plugged into \cref{eq:zzlscFatouStarting} gives \cref{eq:boundedL2weak2FFIMCF}. Now $\partial \Omega_t^{\sml{1}}$ is a curvature varifold and $\sup_k \int_{\partial\Omega^{\sml{p_k}}_t} \abs{\H^{\sml{p_k}}}^2 \dif \Hff^{n-1} < +\infty$, for almost every $t \in [0,T]$, by \eqref{eq:zzbounded_mean_curvature}. Then, \cref{lem:VarifoldWeakConvergence2ff} implies the weak convergence of $\h^{\sml{p_k}}$ to $\h^{\sml{1}}$ for almost every $t \in[0,T]$ of the whole subsequence.
    \end{proof}
    
We are now ready to improve the convergence of $w_p$ to $w_1$, given in \cref{thm:local-approximation}.

\begin{theorem}\label{thm:improved_convergence}
    Let $(w_p)_{p>1}$ be the family of solutions produced in \cref{thm:local-approximation}, then $\nabla w_p$ converges strongly to $\nabla w_1$ in $L^{q}(U\smallsetminus \Omega)$ as $p\to 1^+$, for all $U\Subset D $ such that $\sup_{U} w_1 < \inf_{\partial D} w_1$. 
\end{theorem}

\begin{proof}
    Let $T \in (\sup_{U} w_1, \inf_{\partial D} w_1)$ such that $\abs{\set{w_1=T}}=0$. It is sufficient to show that for any sequence of $p$'s converging to $1$ we can extract a further subsequence $(p_k)_{k\in \N}$ such that
    \begin{equation}\label{eq:L2norm-convergence}
        \lim_{k\to +\infty} \int_{\Omega^{\sml{p_k}}_{T}\smallsetminus\Omega} \abs{\nabla w_{p_k}}^2  = \int_{\Omega^{\sml{1}}_{T}\smallsetminus\Omega} \abs{ \nabla w_1}^2 .
    \end{equation}
    Indeed, we would then get that any sequence admits a subsequence converging strongly in $L^2$ to $\nabla w_1$ on $\Omega^{\sml{1}}_T \smallsetminus\Omega$, inferring the statement of the proposition for $q=2$. The general result follows by the dominated convergence theorem, since $\abs{\nabla w_p}$ is locally equibounded.
    \smallskip

    We now turn to prove \cref{eq:L2norm-convergence}. For a sequence of $p$'s converging to $1$, we pick a subsequence $(p_k)_{k \in \N}$ along which $\partial \Omega^{\sml{p_k}}_t$ converges to $\partial \Omega^{\sml{1}}_t$ in the sense of varifolds by \cref{cor:varifold_convergence_IMCF}. Together with \cref{rmk:equivalent_formulation_monotonicity_IMCF}, it implies 
    \begin{equation}\label{eq:zzmeancurvature_weaklimit}
        \lim_{k \to +\infty}\int_{\partial \Omega^{\sml{p_k}}_t} \ip{\vec{\H}^{\sml{p_k}}| X} \dif \Hff^{n-1} = \int_{\partial \Omega^{\sml{1}}_t} \ip{\vec{\H}^{\sml{1}}|X }\dif \Hff^{n-1}=-\int_{\partial \Omega^{\sml{1}}_t} \ip{\nabla w_1|X }\dif \Hff^{n-1},
    \end{equation}
    for any bounded continuous vector field $X$ and for almost every $t \in [0,T]$. Moreover, \cref{thm:vanishing_gradient_mean_curvature_NPT} and the convergence of areas in \cref{thm:area_convergence} imply that 
    \begin{equation}\label{eq:zzgradient_weaklimit}
        \lim_{k\to +\infty} \int_{\partial \Omega^{\sml{p_k}}_t} \ip{\nabla w_{p_k}| X} \dif \Hff^{n-1}=- \lim_{k \to +\infty}\int_{\partial \Omega^{\sml{p_k}}_t}  \ip{\vec{\H}^{\sml{p_k}}| X} \dif\Hff^{n-1}\overset{\cref{eq:zzmeancurvature_weaklimit}}{=}\int_{\partial \Omega^{\sml{1}}_t} \ip{\nabla w_1|X }\dif \Hff^{n-1}.
    \end{equation}
    for almost every $t\in[0,T)$. Indeed, 
    \begin{equation}
        \abs{\int_{\partial \Omega^{\sml{p_k}}_t} \ip{\nabla w_{p_k}+\vec{\H}^{\sml{p_k}}| X} \dif \Hff^{n-1}} \leq \norm{X}_{\CS^0}\abs{\partial \Omega^{\sml{p_k}}_t}^{\frac{1}{2}}\left(\int_{\partial \Omega^{\sml{p_k}}_t}\abs{\abs{\nabla w_{p_k}}-\H^{\sml{p_k}}}^2 \dif \Hff^{n-1}\right)^{\frac{1}{2}},
    \end{equation}
    which vanishes in the limit $k\to\infty$ for almost every $t \in [0,T)$. H\"older's inequality and \cref{cor:energy_boundedness} imply
    \begin{equation} \label{eq:zzzdomination_gradient}
    \begin{split}
        \abs{\int_{\partial \Omega^{\sml{p_k}}_{t}} \ip{\nabla w_{p_k}|X}  \dif \Hff^{n-1}} &\le \int_{\partial \Omega^{\sml{p_k}}_{t}} \abs{X}\abs{\nabla w_{p_k}}^{p_k-1}\abs{\nabla w_{p_k}}^{2-p_k} \dif \Hff^{n-1}\\&\le \kst \ee^{t}  \norm{X}_{L^\infty}\norm{\nabla w_{p_k}}^{2-p_k}_{L^\infty(\Omega^{\sml{p_k}}_T \smallsetminus \Omega)},
    \end{split}
    \end{equation}
    for some positive constant. By \cref{eq:zzgradient_weaklimit} and \cref{eq:zzzdomination_gradient}, we can apply dominated convergence theorem and coarea formula to obtain
    \begin{align}
        \lim_{k \to +\infty}\int_{\Omega^{\sml{p_k}}_{T}\smallsetminus\Omega}\abs{ \nabla w_{p_k}}\ip{\nabla w_{p_k}|X} &= \lim_{k \to +\infty}\int_0^{T} \int_{\partial \Omega^{\sml{p_k}}_{t}}\ip{\nabla w_{p_k}|X}\dif \Hff^{n-1}\dif t \\
        &=\int_0^{T}\int_{\partial \Omega^{\sml{1}}_{t}} \ip {\nabla w_1 | X} \dif \Hff^{n-1} \dif t= \int_{ \Omega^{\sml{1}}_{T}\smallsetminus\Omega}  \abs{\nabla w_1}\ip{\nabla w_1 | X} .
    \end{align}
    The square norm $\abs{\nabla w_{p_k}}^2$ is equibounded in $L^2(\Omega^{\sml{p_k}}\smallsetminus \Omega)$ as $k\to +\infty$. Hence, by density we can replace $X$ with the unit normal vector field $\nu^{\sml{1}}$ to the level sets of $w_1$, showing that
    \begin{equation}
        \lim_{k \to +\infty}\int_{\Omega^{\sml{p_k}}_{T}\smallsetminus\Omega}\abs{ \nabla w_{p_k}}\ip{\nabla w_{p_k}|\nu^{\sml{1}}} = \int_{ \Omega^{\sml{1}}_{T}\smallsetminus\Omega}  \abs{\nabla w_1}\ip{\nabla w_1 | \nu^{\sml{1}}} = \int_{ \Omega^{\sml{1}}_{T}\smallsetminus\Omega}  \abs{\nabla w_1}^2 .
    \end{equation}
    Recall that we set $\nu^{\sml{1}}=0$ almost everywhere on $\Crit(w_1)$.
    To conclude, observe that
    \begin{equation}
        \abs{\int_{\Omega^{\sml{p_k}}_{T}\smallsetminus\Omega}\abs{ \nabla w_{p_k}}\ip{\nabla w_{p_k}|\nu^{\sml{1}}-\nu^{\sml{p_k}}}}\leq \left(\int_{\Omega_T^{\sml{p_k}}\smallsetminus\Omega} \abs{\nabla w_{p_k}}^3 \right)^{\frac{1}{2}} \left(\int_0^{T}\int_{\partial \Omega^{\sml{p_k}}_{t}}\abs{\nu^{\sml{1}}-\nu^{\sml{p_k}}}^2\dif \Hff^{n-1} \dif t\right)^{\frac{1}{2}}
    \end{equation}
    Hence \cref{eq:L2norm-convergence} follows by \cref{thm:area_convergence}.
\end{proof}

\begin{proof}[Proof of \cref{main-convergences}]
    Since $\sup_{\partial D} w_1 < \overline{T}$, we can take $T>0$ and $D'$ open precompact in $M$ such that $D \subseteq \Omega_T^{(1)}\Subset D'$ and $\sup_{\partial D'}w_1 < \overline{T}$. By \cref{thm:local-approximation}, we can construct a family $(\Phi_p)_{p>1}$ such that the corresponding solutions $w_p$ to \cref{eq:moser_potential} converge uniformly to $w_1$ on $\overline{D'\smallsetminus \Omega}$ as $p\to 1^+$. Then we get (1) by restricting $w_p$ to $D$.  By \cref{thm:improved_convergence} we get (2). To conclude, (3) follows from \cref{cor:varifold_convergence_IMCF}.
\end{proof}

\subsection{Lower semicontinuity of \texorpdfstring{$\MF_p'$}{Fp'}} We would like to prove \cref{thm:monotonicty_formula_IMCF} by passing \cref{eq:derivative-monotonicity_formula-NPT} in \cref{thm:monotonicity-NPT} to the limit as $p \to 1^+$. We work with the family $(w_p)_{p>1}$ provided by \cref{main-convergences}. By construction, this is the approximating family obtained from \cref{thm:local-approximation}, and hence the estimates established in the previous subsection apply. The work done up to this point only covers the convergence of $\MF_p(t)$. Indeed, we can now pass to the limit both the gradient and the mean curvature. This would be enough to prove the monotonicity, but it does not say anything about the derivative of the limit monotone quantity $\MF_1(t)$ expressed in \cref{eq:derivative-monotonicity_formula-IMCF}. Hence we need a more thorough analysis. More precisely, we need to control the three terms appearing in the expression of $\SQ_p$ in \cref{eq:Qpdef}. Observe that we only provide a lower bound for the derivative of $\MF_1(t)$. Indeed, we can only show the lower semicontinuity of the terms appearing in $\SQ_p$. Moreover, only two terms survive in this process. As far as the third one is concerned, we drop it since it is nonnegative. 

\smallskip

In the next lemma, we study the lower semicontinuity of the term containing the traceless second fundamental form. As one may expect, the proof strongly relies on the varifold convergence of $\partial \Omega_t^{\sml{p}}$ to $\partial \Omega^{\sml{1}}_t$ established above. The extra steps are required to deal with different values of $\alpha$.

\begin{lemma}\label{lem:tracelessLSC}
    Assume $\inf_{\partial D} w_1 >T>0$ in \cref{main-convergences}. Then, for every nonnegative function $\psi\in \Lip_c([0,T))$ and $\alpha >1$, we have
    \begin{multline}\label{eq:tracelessLSC}
        \int_0^T \psi(t)  \ee^{\left(\frac{\alpha}{n-1}-1\right)t}\int_{\partial \Omega^{\sml{1}}_t} \H^{\alpha -2} \abs{\mathring{\h}^{\sml{1}}}^2 \dif \Hff^{n-1} \dif t \\ \leq \liminf_{p\to 1}\int_0^T \psi(t)  \ee^{\left(\frac{\alpha}{n-p}-1\right)t}\int_{\partial \Omega^{\sml{p}}_t} \abs{\nabla w_p}^{\alpha + p-3} \abs{\mathring{\h}^{\sml{p}}}^2 \dif \Hff^{n-1} \dif t 
    \end{multline}
\end{lemma}

\begin{proof}
Pick a sequence $(p_k)_{k\in\N}$ converging to $1$, which realizes the limit inferior in \cref{eq:tracelessLSC}. Let $S$ be a continuous $(1,2)$-tensor on $M$, we can apply \cref{cor:varifold_convergence_IMCF} to get that
\begin{equation}\label{eq:zzPointwiseConvergencetraceless}
    \lim_{k \to +\infty} \int_{\partial \Omega^{\sml{p_k}}_t} \ip{\mathring{\h}^{\sml{p_k}}| S }\dif \Hff^{n-1} = \int_{\partial \Omega^{\sml{1}}_t} \ip{\mathring{\h}^{\sml{1}}| S }\dif \Hff^{n-1},
\end{equation}
for almost every $t \in[0,T]$. Moreover, the sequence of maps
\begin{equation}
    t \mapsto \int_{\partial \Omega^{\sml{p_k}}_t} \ip{\mathring{\h}^{\sml{p_k}}| S }\dif \Hff^{n-1}.
\end{equation}
is equibounded in $L^2(0,T)$. Indeed, by H\"older's inequality and \cref{cor:energy_boundedness} we estimate
\begin{equation}
    \int_0^T\left(  \int_{\partial \Omega^{\sml{p_k}}_t} \ip{\mathring{\h}^{\sml{p_k}}| S }\dif \Hff^{n-1}\right)^2 \dif t \leq \kst(T) \, \norm{S}^2_{L^\infty}  \int_0^T \int_{\partial \Omega^{\sml{p_k}}_t} \abs{\mathring{\h}^{\sml{p_k}}}^2 \abs{\nabla w_{p_k}}^{1-p_k} \dif \Hff^{n-1} \dif t.
\end{equation}
Since $4-2p_k \to 2$ as $k \to +\infty$, the right-hand side is equibounded by \cref{thm:equibounded_2ff_NPT} applied with $\alpha(p_k) = 4-2p_k$. Therefore, by weak convergence in $L^2(0,T)$ and \cref{eq:zzPointwiseConvergencetraceless}, it follows that
\begin{equation}\label{eq:zzweakL2convergence}
    \int_0^T\int_{\partial \Omega^{\sml{p_k}}_t} \ip{\mathring{\h}^{\sml{p_k}}| S }\dif \Hff^{n-1} \dif t\xrightarrow{k\to +\infty}  \int_0^T \int_{\partial \Omega^{\sml{1}}_t} \ip{\mathring{\h}^{\sml{1}}| S }\dif \Hff^{n-1} \dif t 
\end{equation}
for any continuous $(2,1)$-tensor $S$ on $M$. Applying coarea formula to both sides of \cref{eq:zzweakL2convergence}, we obtain
\begin{equation}\label{eq:zzweakL2convergenceVolume}
    \int_{\Omega_{T}^{\sml{p_k}} \smallsetminus \Omega} \abs{\nabla w_{p_k}} \ip{\mathring{\h}^{\sml{p_k}}| S} \dif \Hff^n\xrightarrow{k\to +\infty} \int_{\Omega_{T}^{\sml{1}} \smallsetminus \Omega} \abs{\nabla w_{1}} \ip{\mathring{\h}^{\sml{1}}| S} \dif \Hff^n.
\end{equation}
Observe that
\begin{equation}\label{eq:zzzequiboundednessandLimit}
    \sup_k\int_{\Omega_{T}^{\sml{p_k}} \smallsetminus \Omega} \abs{\nabla w_{p_k}}^2 \abs{\mathring{\h}^{\sml{p_k}}}^2 \dif \Hff^n  + \int_{\Omega_{T}^{\sml{1}} \smallsetminus \Omega}\abs{\nabla w_1}^2 \abs{\mathring{\h}^{\sml{1}}}^2 \dif \Hff^n < +\infty
\end{equation}
by \cref{thm:equibounded_2ff_NPT} for $\alpha(p_k)=4-p_k$, \cref{cor:varifold_convergence_IMCF} and since $w_1$ is a locally Lipschitz function. In particular, by density \cref{eq:zzweakL2convergenceVolume} holds for every $S \in L^2(D\smallsetminus \Omega)$. Let now $\delta>0$ and let $\chi_\delta$ be a smooth cut-off function satisfying 
    \begin{equation}
        \chi_\delta(t) \coloneqq\begin{cases}
            &\chi_\delta(t) =0& & \text{if $t \leq \delta$,}\\
            &0\leq \chi'_\delta(t) \leq 2/\delta&& \text{if $t \in [\delta, 2 \delta]$,}\\
            &\chi_\delta(t)=1& &\text{if $t \geq2 \delta$.}\\
        \end{cases}
    \end{equation}
Applying \cref{thm:improved_convergence}, one can easily show that
\begin{multline}\label{eq:zzweaklimitpigiami2ff}
    \psi(w_{p_k})^{\frac{1}{2}}  \ee^{\left(\frac{\alpha}{n-p_k}-1\right)\frac{w_{p_k}}{2}} \chi_{\delta}(\abs{\nabla w_{p_k}})\abs{\nabla w_{p_k}}^{\frac{\alpha+{p_k}-4}{2}} \\\xrightarrow{k\to +\infty} \psi(w_{1})^{\frac{1}{2}}  \ee^{\left(\frac{\alpha}{n-1}-1\right)\frac{w_{1}}{2}} \chi_{\delta}(\abs{\nabla w_{1}})\abs{\nabla w_{1}}^{\frac{\alpha-3}{2}} \in L^\infty(\Omega^{\sml{1}}_T \smallsetminus \Omega)
\end{multline}
strongly in $L^2(D \smallsetminus \Omega)$. Coupling it with the weak convergence in \cref{eq:zzweakL2convergenceVolume}, we have 
\begin{multline}
    \int_{\Omega^{\sml{p_k}}_T\smallsetminus \Omega} \psi(w_{p_k})^{\frac{1}{2}}  \ee^{\left(\frac{\alpha}{n-p_k}-1\right)\frac{w_{p_k}}{2}} \chi_{\delta}(\abs{\nabla w_{p_k}})\abs{\nabla w_{p_k}}^{\frac{\alpha+{p_k}-2}{2}}\ip{\mathring{\h}^{\sml{p_k}}|S} \dif \Hff^{n}\\
    \xrightarrow{k\to +\infty} \int_{\Omega^{\sml{1}}_T\smallsetminus \Omega}  \psi(w_{1})^{\frac{1}{2}}  \ee^{\left(\frac{\alpha}{n-1}-1\right)\frac{w_{1}}{2}} \chi_{\delta}(\abs{\nabla w_{1}})\abs{\nabla w_{1}}^{\frac{\alpha-1}{2}}\ip{\mathring{\h}^{\sml{1}}|S}\dif \Hff^n.
\end{multline}
for every continuous $(1,2)$-tensor on $M$. By \cref{eq:zzzequiboundednessandLimit,eq:zzweaklimitpigiami2ff} the norm of the limit is characterized by the supremum over continuous tensors $S$ with $\norm{S}_{L^2(\Omega^{\sml{1}}_T \smallsetminus \Omega)} \le 1$. Hence, by coarea formula we have
\begin{multline}
    \liminf_{p\to 1^+}\int_0^T \psi(t)  \ee^{\left(\frac{\alpha}{n-p}-1\right)t}\int_{\partial \Omega^{\sml{p}}_t} \abs{\nabla w_p}^{\alpha + p-3} \abs{\mathring{\h}^{\sml{p}}}^2 \dif \Hff^{n-1} \dif t \\\geq \int_0^T \psi(t)  \ee^{\left(\frac{\alpha}{n-1}-1\right)t}\int_{\partial \Omega^{\sml{1}}_t} \chi_{\delta}(\abs{\nabla w_1})^2\H^{\alpha -2} \abs{\mathring{\h}^{\sml{1}}}^2 \dif \Hff^{n-1},
\end{multline}
from which the statement follows by monotone convergence theorem sending $\delta \to 0$.
\end{proof}

We now study the lower semicontinuity of the term involving $\nabla^\top \log(\abs{\nabla w_p})$. The proof is similar to the one just proposed. The only difference is the characterization of the limit. Concerning the second fundamental form, the convergence in the sense of varifolds characterizes the limit. Here, we need a further step. 

\begin{lemma}\label{lem:LSCgradTangenzialeH}
     Assume $\inf_{\partial D} w_1 >T>0$ in \cref{main-convergences}. Then, $\log(\H^{\sml{1}}) \in W^{1,2}(\partial\Omega^{\sml{1}}_t \smallsetminus S_t)$ for almost every $t \in [0,T]$, where $S_t$ is the singular set of $\partial\Omega^{\sml{1}}_t$. Moreover, for every $\psi\in \Lip_c([0,T))$ and $\alpha >1$, we have
    \begin{multline}\label{eq:lsc_tangential_logaritmic}    
        \int_0^T \psi(t)  \ee^{\left(\frac{\alpha}{n-1}-1\right)t}\int_{\partial \Omega^{\sml{1}}_t} \H^{\alpha -2} \frac{\abs{ \nabla^\top\H}^2}{\H^2} \dif \Hff^{n-1}\\\leq \liminf_{p\to 1^+}\int_0^T \psi(t)  \ee^{\left(\frac{\alpha}{n-p}-1\right)t}\int_{\partial \Omega^{\sml{p}}_t} \abs{\nabla w_p}^{\alpha + p-3} \frac{\abs{ \nabla^\top \abs{ \nabla w_p}}^2}{\abs{\nabla w_p}^2
        }\dif \Hff^{n-1} \dif t 
    \end{multline}
\end{lemma}

\begin{proof}
    Fix $\delta>0$. We claim that $\log(\abs{\nabla w_1} ) \in W^{1,2}(\partial \Omega^{\sml{1}}_t \smallsetminus S_t)$ for almost every $t \in [0,T)$ and
    \begin{equation}\label{eq:LSCtangentialclaim}
        \int_{\Omega^{\sml{p}}_T \smallsetminus \Omega}\abs{\nabla w_p}\ip{\frac{\nabla^\top \abs{\nabla w_p}}{\abs{\nabla w_p}+\delta}| X} \dif \Hff^n \xrightarrow{p \to 1^+} \int_{\Omega^{\sml{1}}_T \smallsetminus \Omega} \abs{\nabla w_1}\ip{\frac{\nabla^\top \abs{\nabla w_1}}{\abs{\nabla w_1}+\delta}| X} \dif \Hff^n
    \end{equation}
    weakly for every $X\in L^2(D)$. The statement then follows as in \cref{lem:tracelessLSC}. Indeed, the lower semicontinuity implies
    \begin{multline}
        \liminf_{p\to 1^+}\int_0^T \psi(t)  \ee^{\left(\frac{\alpha}{n-p}-1\right)t}\int_{\partial \Omega^{\sml{p}}_t} \abs{\nabla w_p}^{\alpha + p-3} \left(\frac{\abs{ \nabla^\top \abs{ \nabla w_p}}}{\abs{\nabla w_p}+\delta}\right)^2\dif \Hff^{n-1} \dif t \\ \geq \int_0^T \psi(t)  \ee^{\left(\frac{\alpha}{n-1}-1\right)t}\int_{\partial \Omega^{\sml{1}}_t} \chi_{\delta}(\abs{\nabla w_1})^2\H^{\alpha -2} \left(\frac{\abs{ \nabla^\top\H}}{\H+\delta} \right)^2\dif \Hff^{n-1} \dif t,
    \end{multline}
    where $\chi_\delta$ is the usual smooth cut-off function. The monotone convergence theorem grants that we can send $\delta \to 0^+$.

    \smallskip

    The claim will be proved as follows: first, we provide a $L^2$ weak limit; second, we characterize that limit. By \cref{cor:equibounded_gradient_normgradient,prop:RegolaritaWp}, $\abs{\nabla w_p}\in \CS^0 \cap W^{1,2}(\Omega^{\sml{p}}_T \smallsetminus \Omega^\circ)$, then $\log(\abs{\nabla w_p}+\delta)\in\CS^0 \cap W^{1,2}(\Omega^{\sml{p}}_T \smallsetminus \Omega^\circ)$. In particular, \cref{cor:equibounded_gradient_normgradient} with $\alpha =3-p$ and coarea formula give
    \begin{equation}
        \int_{\Omega_T^{\sml{p}}\smallsetminus \Omega} \abs{\nabla w_{p}}\left(\frac{\abs{\nabla^\top \abs{\nabla w_p}}}{\abs{\nabla w_p}+\delta}\right)^2 \dif \Hff^n \leq \int_{\Omega_T^{\sml{p}}\smallsetminus \Omega}\abs{\nabla w_{p}}\frac{\abs{\nabla^\top\abs{\nabla w_p}}^2}{\abs{\nabla w_p}^2}\dif \Hff^n \leq \kst,
    \end{equation}
    for a positive constant $\kst$ independent of $p$. By \cref{thm:improved_convergence}, $\abs{\nabla w_p}\Hff^{n}\res (\Omega^{\sml{p}}\smallsetminus \Omega) $ weakly*-converges to $\abs{\nabla w_1} \Hff^{n}\res (\Omega^{\sml{1}}\smallsetminus \Omega)$. Since $\abs{\nabla w_p}$ is equibounded by \cref{prop:uniform_gradient}, \cref{lem:JointConvergenceLSC} implies the existence of a subsequence $(p_k)_{k\in \N}$ converging to $1$ as $k\to +\infty$ and a vector field $Y_{\delta}$, which satisfies
    \begin{equation}
    \int_{\Omega^{\sml{1}}_T\smallsetminus\Omega}\abs{\nabla w_1} \abs{Y_{\delta}}^2 \dif \Hff^n \leq \liminf_{k\to +\infty} \int_{\Omega^{\sml{p_k}}_T\smallsetminus\Omega} \abs{\nabla w_{p_k}}\frac{\abs{\nabla^\top\abs{\nabla w_{p_k}}}^2}{\abs{\nabla w_{p_k}}^2} \dif \Hff^n
    \end{equation}
    and
    \begin{equation}\label{eq:zzjointconvergencefullsetintegral}
        \int_{\Omega^{\sml{1}}_T\smallsetminus\Omega}\abs{\nabla w_1} \ip{Y_{\delta}| X}\dif \Hff^n = \lim_{k\to +\infty} \int_{\Omega^{\sml{p_k}}_T\smallsetminus\Omega} \abs{\nabla w_{p_k}}\ip{\frac{\nabla^\top \abs{\nabla w_{p_k}}}{\abs{\nabla w_{p_k}}+\delta}| X}\dif \Hff^n
    \end{equation}
    for every $X \in L^2(D)$. In particular, coarea formula yields
    \begin{align}
        \begin{multlined}[.7\textwidth] \lim_{k \to +\infty}\int_0^T \eta(t) \int_{\partial \Omega^{\sml{p_k}}_t}\ip{\frac{\nabla^\top \abs{\nabla w_{p_k}}}{\abs{\nabla w_{p_k}}+\delta}| X} \dif \Hff^{n-1} \dif t \\ \begin{aligned}[t]&= \lim_{k\to +\infty}\int_{\Omega^{\sml{p_k}}_T\smallsetminus\Omega} \eta(w_{p_k})\abs{\nabla w_{p_k}}\ip{\frac{\nabla^\top \abs{\nabla w_{p_k}}}{\abs{\nabla w_{p_k}}+\delta}| X}\dif \Hff^n\\& =\int_{\Omega^{\sml{1}}_T\smallsetminus\Omega} \eta(w_{1})\abs{\nabla w_{1}}\ip{Y_\delta| X}\dif \Hff^n = \int_0^T \eta(t) \int_{\partial \Omega^{\sml{1}}_t} \ip{Y_\delta| X} \dif \Hff^{n-1} \dif t
        \end{aligned}
        \end{multlined}
    \end{align}
    for every $X \in \CS^0_c(D)$ and every $\eta \in \CS^\infty_c(0,T)$. By H\"older's inequality, \cref{lem:exp_areagrowth} and the coarea formula, we have
    \begin{equation}
        \int_0^T \left(\int_{\partial \Omega^{\sml{1}}_t} \ip{Y_\delta | X}\dif \Hff^{n-1} \right)^2\dif t \leq \abs{\partial \Omega^*} \ee^T \norm{X}^2_{L^\infty}   \int_{\Omega^{\sml{1}}_T\smallsetminus\Omega}\abs{\nabla w_1} \abs{Y_{\delta}}^2 \dif \Hff^n,
    \end{equation}
    which is bounded for every $X \in \CS^0_c(D)$. On the other hand, H\"older's inequality and \cref{cor:energy_boundedness} yield
    \begin{equation}
        \begin{split}
        &\int_0^T \left(\int_{\partial \Omega^{\sml{p_k}}_t} \ip{\frac{\nabla^\top \abs{\nabla w_{p_k}}}{\abs{\nabla w_{p_k}}+\delta}| X} \dif \Hff^{n-1} \right)^2 \dif t \leq \kst \ee^T \norm{X}_{L^\infty}^2  \int_0^T \int_{\partial \Omega_t^{\sml{p_k}}} \frac{\abs{\nabla^\top \abs{\nabla w_{p_k}}}^2}{\abs{\nabla w_{p_k}}^2} \abs{\nabla w_{p_k}}^{1-{p_k}} \dif \Hff^{n-1},
        \end{split}
    \end{equation}
    which is equibounded by \cref{cor:equibounded_gradient_normgradient} applied with $\alpha(p_k) =4-2p_k $, since $4-2p_k\to 2$ as $k\to +\infty$. Hence, by the density of $\CS^\infty_c(0,T)$ in $L^2(0,T)$ we have that
    \begin{equation}
         \int_{\partial \Omega^{\sml{p_k}}_t} \ip{\frac{\nabla^\top \abs{\nabla w_{p_k}}}{\abs{\nabla w_{p_k}}+\delta}| X} \dif \Hff^{n-1}\xrightharpoonup{k\to +\infty} \int_{\partial \Omega^{\sml{1}}_t} \ip{Y_\delta | X }\dif \Hff^{n-1}
    \end{equation}
    weakly in $L^2(0,T)$.

\medskip

    We are left to characterize the limit $Y_\delta$. First of all, we prove that $Y_\delta$ is $\Hff^{n-1}$-a.e. tangential on $\partial (\Omega^{\sml{1}} \smallsetminus S_t)$ for almost every $t \in (0,T)$. Let $\eta \in \CS^0_c(\Omega^{\sml{1}}_T \smallsetminus \Omega)$ and $\xi \in \CS^0_c(0,T)$. Since $[\xi(w_p) \eta]\nabla w_p \to[\xi(w_1) \eta] \nabla w_1$ strongly in $L^2$ by \cref{thm:improved_convergence} and \cref{thm:local-approximation}, we get by \cref{eq:zzjointconvergencefullsetintegral} that
\begin{align}
    \int_{\Omega^{\sml{1}}_T \smallsetminus \Omega}\xi(w_1) \eta \abs{\nabla w_1} \ip{Y_\delta | \nabla w_1}  &=  \lim_{k \to +\infty} \int_{\Omega^{\sml{p_k}}_T\smallsetminus\Omega} \xi(w_{p_k}) \eta\abs{\nabla w_{p_k}}\ip{\frac{\nabla^\top \abs{\nabla w_{p_k}}}{\abs{\nabla w_{p_k}}+\delta}| \nabla w_{p_k}} =0.
\end{align}
H\"older's inequality, \cref{lem:exp_areagrowth} and the coarea formula yield
\begin{equation}
    \int_0^T \left(\int_{\partial \Omega^{\sml{1}}_t}\eta \ip{Y_\delta | \nabla w_1}\dif \Hff^{n-1} \right)^2\dif t \leq \abs{\partial \Omega^*}\ee^T \norm{\eta \nabla w_1}^2_{L^\infty}\int_{\Omega^{\sml{1}}_T \smallsetminus \Omega} \abs{\nabla w_1} \abs{Y_\delta}^2 \dif \Hff^n.
\end{equation}
Hence, by arbitrariness of $\xi$, we deduce that
\begin{equation}\label{eq:Ydeltaistangential}
    \int_{\partial \Omega^{\sml{1}}_t} \eta \ip{Y_\delta | \nabla w_1} \dif \Hff^{n-1} = 0
\end{equation}
for almost every $t \in (0,T)$ and for every $\eta \in \CS^0_c(\Omega^{\sml{1}}_T \smallsetminus \Omega)$. Recall that $\partial\Omega^{\sml{1}}_{t}$ is of class $\CS^1$ out of a closed set $S_t$ having Hausdorff dimension less than $n-8$ (see \cref{prop:regularityLevelsetIMCF}). By coarea formula, $\ip{Y_\delta |\nabla w_1} \in L^2(\partial\Omega_t^{\sml{1}})$ for almost every $t\in[0,T]$. Hence, arbitrariness of $\eta$ in \cref{eq:Ydeltaistangential} implies that $Y_\delta$ is tangent $\Hff^{n-1}$-a.e. along $\partial\Omega_t^{\sml{1}} \smallsetminus S_t$. 

\smallskip

If we now prove that
\begin{equation}\label{eq:zzintegrationbypartfinal}
    -\int_{\partial \Omega^{\sml{1}}_t} \ip{Y_\delta | X} \dif \Hff^{n-1} = \int_{\partial \Omega^{\sml{1}}_t} \log (\abs{\nabla w_1}+\delta) \left( \ip{\vec{\H}^{\sml{1}}| X} + \div_{\top }X\right)\dif \Hff^{n-1}
\end{equation}
for every $X \in \CS^1_c(D)$, we both conclude that $\log(\abs{\nabla w_1} + \delta ) \in W^{1,2}(\partial \Omega^{\sml{1}}_t \smallsetminus S_t)$ and that $Y_\delta = \nabla^\top \log(\abs{\nabla w_1} +\delta)$. Moreover, chain rule implies that $\abs{\nabla w_1} \in W^{1,2}(\partial \Omega^{\sml{1}}_t \smallsetminus S_t)$.

By \cref{lem:ConvergenzaVarifoldEspToZero} and \cref{rem:ConvergenzaVarifoldEspToZeroNONsmoothOmega} and \eqref{eq:IntegrationByPartsSobolevOnVarifold} we can integrate by parts, thus
    \begin{equation}\label{eq:zzIdentity0}
    \begin{multlined}[c][.7\textwidth]
        -\int_0^{T} \int_{\partial\Omega^{\sml{p_k}}_{t}}  \ip{ \frac{\nabla^\top \abs{\nabla w_{p_k}}}{\abs{\nabla w_{p_k}}+\delta}| X} \dif \Hff^{n-1} \dif t \\\begin{aligned}[t] &=\int_0^{T}\int_{\partial\Omega^{\sml{{p_k}}}_{t}} \log(\abs{\nabla w_{p_k}}+\delta)\left(\ip{\vec{\H}^{\sml{p_k}}| X}+\div_{\top} X\right) \dif \Hff^{n-1} \dif t \\
        &=\int_{\Omega^{\sml{{p_k}}}_{T}\smallsetminus \Omega}\abs{\nabla w_{p_k}} \log(\abs{\nabla w_{p_k}}+\delta)\left(\ip{\vec{\H}^{\sml{p_k}}| X}+\div_{\top} X\right) \dif \Hff^n ,
        \end{aligned}
    \end{multlined}
    \end{equation}
    for any $X \in \CS^1_c(D)$. We now consider the two terms in the last integral separately. By H\"older's inequality, \cref{prop:uniform_gradient} and the coarea formula, we have 
    \begin{multline}\label{eq:lsctangential-firstconvergence}
        \abs{\int_{\Omega^{\sml{p_k}}_{T}\smallsetminus \Omega} \log(\abs{\nabla w_{p_k}}+\delta)\abs{\nabla w_{p_k}}\ip{\vec{\H}^{\sml{p_k}} +\nabla w_{p_k}| X} }\\\leq \kst\norm{X}_{L^\infty} \int_0^T \int_{\partial \Omega^{\sml{p_k}}_t }(\H^{\sml{p_k}}- \abs{\nabla w_{p_k}})^2 \dif \Hff^{n-1} \dif t 
    \end{multline}
    for some constant $\kst$ uniformly bounded in $p$, for any $X \in \CS^1_c(D)$. By \cref{thm:vanishing_gradient_mean_curvature_NPT}, the right-hand side vanishes as $k\to +\infty$, hence we can replace the mean curvature with the gradient in \cref{eq:zzIdentity0}.
    Moreover, by \cref{thm:improved_convergence} we have that
    \begin{equation}\label{eq:lsctangential-secondconvergence}
        \lim_{k \to +\infty}\int_{\Omega^{\sml{{p_k}}}_{T}\smallsetminus \Omega} \log(\abs{\nabla w_{p_k}}+\delta)\abs{\nabla w_{p_k}}\ip{\nabla w_{p_k}| X} = \int_{\Omega^{\sml{1}}_{T}\smallsetminus \Omega} \log(\abs{\nabla w_{1}}+\delta)\abs{\nabla w_1}\ip{\nabla w_{1}| X} .
    \end{equation}
    On the other side, the following relation
    \begin{equation}
    \div^{\sml{p_k}}_\top X= \div X +  \ip{ \nabla_{\nu^{\sml{p_k}}} X| \nu^{\sml{1}} -\nu^{\sml{p_k}} }+ \ip{\nabla_{(\nu^{\sml{1}} - \nu^{\sml{p_k}})} X| \nu^{\sml{1}} }-\ip{ \nabla_{\nu^{\sml{1}}} X| \nu^{\sml{1}} }
    \end{equation}
    holds for any $X \in \CS^1_c(D)$. Hence, \cref{thm:area_convergence} and \cref{thm:improved_convergence} imply
    \begin{equation}\label{eq:zzConvConv}
        \lim_{k\to +\infty} \int_{\Omega^{\sml{p_k}}_{T}\smallsetminus \Omega} \abs{\nabla w_{p_k}} \log(\abs{\nabla w_{p_k}}+\delta)\div_{\top}^{\sml{p_k}}X  = \int_{\Omega^{\sml{1}}_{T}\smallsetminus \Omega} \abs{\nabla w_1}   \log(\abs{\nabla w_{1}}+\delta) \div^{\sml{1}}_{\top}X ,
    \end{equation}
    for any $X \in \CS^1_c(D)$. Putting together \cref{eq:zzIdentity0,,eq:lsctangential-firstconvergence,,eq:lsctangential-secondconvergence,,eq:zzConvConv}, we get
    \begin{equation}\label{eq:convergence_weak_tangential_derivative}
    \begin{split}
    \int_{\Omega^{\sml{1}}_T \smallsetminus \Omega} \abs{\nabla w_1}\ip{Y_\delta | X} &=\lim_{k\to +\infty}\int_0^{T}\int_{\partial\Omega^{\sml{p_k}}_{t}}  \ip{\frac{\nabla^\top \abs{\nabla w_{p_k}}}{\abs{\nabla w_{p_k}}+\delta}| X} \dif \Hff^{n-1} \dif t \\&= -\int_0^{T} \int_{\partial\Omega^{\sml{1}}_{t}} \log(\abs{\nabla w_1}+\delta)\left(\ip{\vec{\H}| X}+\div_{\top} X\right)\dif \Hff^{n-1} \dif t 
    \end{split}
    \end{equation}
    for any $X \in \CS^1_c(D)$. \cref{eq:convergence_weak_tangential_derivative} also implies
    \begin{equation}\label{eq:zzIntegrationYdelta2}
        \int_0^{T} \eta(t)\int_{\partial\Omega^{\sml{1}}_{t}} \ip{Y_\delta | X} \dif \Hff^{n-1} \dif t= -\int_0^{T} \eta(t)\int_{\partial\Omega^{\sml{1}}_{t}} \log(\abs{\nabla w_1}+\delta)\left(\ip{\vec{\H}| X}+\div_{\top}  X\right) \dif \Hff^{n-1} \dif t .
    \end{equation}
    for every $\eta \in \CS^0_c(0,T)$. Since
    \begin{equation}
        \int_0^T\left(\int_{\partial\Omega_{t}^{\sml{1}}} \log(\abs{\nabla w_1}+\delta) \left(\ip{\vec{\H}| X}+ \div_{\top}X\right)\dif \Hff^{n-1}\right)^2 \dif t
        \le \kst\norm{X}_{\CS^{1}}^2\ee^T \abs{\partial \Omega^*}
    \end{equation}
    for every $X \in \CS^1_c(D)$, we conclude that \cref{eq:zzIntegrationYdelta2} holds for $\eta \in L^2(0,T)$. Hence, \cref{eq:zzintegrationbypartfinal} holds. The characterization of $Y_\delta$ implies that \cref{eq:LSCtangentialclaim} holds for every $X\in \CS^1_c(D)$, therefore for $X \in L^2(D)$ by density.
    \end{proof}

\subsection{Proof of \texorpdfstring{\cref{thm:monotonicty_formula_IMCF}}{Theorem 4.1}}
\pushQED{\qed}
We prove the statement for $\alpha>1$. The case $\alpha = 1$ follows as in \cref{rmk:casolimite}, since $\SQ^{\sml{\alpha}}_1 \geq \abs{\mathring{\h}}^2 = \SQ^{\sml{1}}_1$, where $\SQ^{\sml{\alpha}}_1$ denotes \cref{eq:Q1def} for every $\alpha \geq 1$.  Let $0<T<\overline{T}$ and let $D$ be an open precompact subset with smooth boundary containing $\Omega$ such that $T < \inf_{\partial D} w_1$ and $\sup_{\partial D} w_1 <\overline{T}$. 

\smallskip

Assume that $\alpha > 1 > (n-p)/(n-1)$. We aim to pass \cref{thm:monotonicity-NPT} to the limit as $p \to 1^+$ for nonnegative $\psi \in \Lip_c([0,T))$. We consider each term in \cref{eq:derivative-monotonicity_formula-NPT} separately, starting from the right-hand side. Since $\psi$ is nonnegative, recalling the expression for $\SQ_{p}$ in \cref{thm:monotonicity-NPT}, it readily follows from \cref{lem:tracelessLSC} and \cref{lem:LSCgradTangenzialeH} that
\begin{multline}\label{eq:zzRHS}
    \liminf_{p\to1^+}\int_0^T \psi(t) \ee^{\left(\frac{\alpha}{n-p} - 1 \right)t} \int_{\partial \Omega^{\sml{p}}_t} \abs{\nabla w_p}^{\alpha+p-3} \SQ_{p}\dif \Hff^{n-1} \dif t
    \\\begin{aligned}[t]&\ge
    \liminf_{p\to1^+}\int_0^T \psi(t) \ee^{\left(\frac{\alpha}{n-p} - 1 \right)t} \int_{\partial \Omega^{\sml{p}}_t} \abs{\nabla w_p}^{\alpha+p-3} \left( \left[\alpha-(2-p)\right] \frac{\abs{\nabla^\top \abs{\nabla w_p}}^2}{\abs{\nabla w_p}^2} + \abs{ \mathring{\h}}^2 \right)\dif \Hff^{n-1} \dif t
    \\& \ge
    \int_0^T \psi(t) \ee^{\left(\frac{\alpha}{n-1} - 1 \right)t} \int_{\partial \Omega^{\sml{1}}_t} \H^{\alpha-2} \left( \left[\alpha-1\right] \frac{\abs{\nabla^\top \H}^2}{\abs{\H}^2} + \abs{ \mathring{\h}}^2 \right)\dif \Hff^{n-1} \dif t.
    \end{aligned}
\end{multline}

We now want to prove that
\begin{equation}\label{eq:USCFpZero}
    \limsup_{p\to1^+} - \MF_p(0) \le - \MF_1(0) = \frac{1}{\alpha} \int_{\partial \Omega} \abs{\H}^{\alpha}.
\end{equation}
Indeed, applying Young's inequality with conjugate exponents $a=(\alpha+p-1)/(\alpha+p-2)$ and $b=\alpha+p-1$ and using \cref{cor:energy_boundedness}, we estimate
\begin{equation}
    \begin{split}
        -\MF_p(0) 
        &= \int_{\partial \Omega}  \H \abs{\nabla w_p}^{\alpha+p-2} - \abs{\nabla w_p}^{\alpha+p-1} \left( \frac{n-1}{n-p} - \frac1\alpha \right) \dif \Hff^{n-1}
        \\& \le
        \int_{\partial\Omega}  \abs{\nabla w_p}^{\alpha+p-1} \left( \frac{\alpha+p-2}{\alpha+p-1} -\frac{n-1}{n-p} + \frac1\alpha \right) + \frac{\abs{\H}^{\alpha+p-1}}{\alpha+p-1}\dif \Hff^{n-1} \\
        &\le  \kst  \left( \frac{\alpha+p-2}{\alpha+p-1} -\frac{n-1}{n-p} + \frac1\alpha \right) \norm{\nabla w_p}_{L^\infty(\partial\Omega)}^{\alpha} + \int_{\partial \Omega} \frac{\abs{\H}^{\alpha+p-1}}{\alpha+p-1} \dif \Hff^{n-1}
    \end{split}
\end{equation}
Therefore, \cref{eq:USCFpZero} follows, being $\norm{\nabla w_p}_{L^\infty(\partial\Omega)}$  uniformly bounded as $p\to1^+$ by \cref{prop:uniform_gradient}.

It only remains to show that
\begin{equation}\label{eq:zzPrimoPezzo}
    \lim_{p\to1^+} \int_0^T \psi'(t) \MF_p(t) \dif t  = \int_0^T \psi'(t) \MF_1(t) \dif t. 
\end{equation}
Recalling the definition of $\MF_p$ in \cref{eq:monotonicty_formula}, by coarea formula we have
\begin{equation}
\label{eq:main_piece_F_p_monotonicity}
\begin{split}
   \int_0^T \psi'(t)\MF_{p}(t)\dif t  =\begin{multlined}[t][.7\textwidth]
   \left(\frac{p-1}{n-p}-\frac{1}{\alpha}\right)\int_{\Omega_T^{\sml{p}}\smallsetminus \Omega}\psi'(w_p) \ee^{\left(\frac{\alpha}{n-p} -1 \right)w_p} \abs{\nabla w_p}^{\alpha+p}\dif \Hff^n\\
   + \int_0^T\psi'(t)\ee^{\left(\frac{\alpha}{n-p} - 1\right)t} \int_{\partial \Omega^{\sml{p}}_t}\abs{\nabla w_p }^{\alpha+p-2} \left(\abs{\nabla w_p} -\H \right) \dif \Hff^{n-1} \dif t\\-\int_0^T\psi'(t)\int_{\Omega^{\sml{p}}_t\smallsetminus \Omega} \ee^{\left(\frac{\alpha}{n-p} - 1\right)w_p}\abs{\nabla w_p}^{\alpha+p-2} \Ric(\nu^{\sml{p}},\nu^{\sml{p}}) \dif \Hff^n\dif t.
    \end{multlined}
\end{split}
\end{equation}

Since $w_p\to w_1$ locally uniformly by \cref{thm:local-approximation} and in $W^{1,2}_{\loc}$ by \cref{thm:improved_convergence}, we have
\begin{equation}\label{eq:LimitI1}
\begin{split}
    \lim_{p\to1^+} \int_{\Omega_T^{\sml{p}}\smallsetminus \Omega}\psi'(w_p) \ee^{\left(\frac{\alpha}{n-p} -1 \right)w_p} \abs{\nabla w_p}^{\alpha+p}\dif \Hff^n &=  \int_{\Omega^{\sml{1}}_T \smallsetminus \Omega} \psi'(w_1) \ee^{\left(\frac{\alpha}{n-1} -1 \right)w_1} \abs{\nabla w_1}^{\alpha+1} \dif \Hff^n
\end{split}
\end{equation}
The second integral vanishes. Indeed, by H\"older's inequality and \cref{cor:energy_boundedness} we have
\begin{equation}\label{eq:LimitI2}
\begin{multlined}[c][.8\textwidth]
    \abs{\int_{\partial \Omega^{\sml{p}}_t}\abs{\nabla w_p }^{\alpha+p-2} \left(\abs{\nabla w_p} -\H \right) \dif \Hff^{n-1}}\\
    \begin{aligned}[t]&\leq \left(\int_{\partial \Omega^{\sml{p}}_t}\abs{\nabla w_p }^{2(\alpha+p-2)}\dif \Hff^{n-1}\right)^{\frac{1}{2}} \left(\int_{\partial \Omega^{\sml{p}}_t}  \left(  \abs{\nabla w_p} - \H\right)^2\dif \Hff^{n-1}\right)^{\frac{1}{2}}\\
    &\leq\kst \ee^t\norm{\nabla w_p}^{\frac{2\alpha +p-3}{2}}_{L^\infty(\Omega^{\sml{p}}_T \smallsetminus \Omega)}\left(\int_{\partial \Omega^{\sml{p}}_t}  \left(  \abs{\nabla w_p} - \H\right)^2\dif \Hff^{n-1}\right)^{\frac{1}{2}}.
    \end{aligned}
\end{multlined}
\end{equation}
By \cref{prop:uniform_gradient,thm:vanishing_gradient_mean_curvature_NPT}, the right-hand side integrated on $[0,T]$ vanishes as $p\to1^+$. It only remains to treat the term involving the Ricci tensor. Observe that
\begin{equation}\label{eq:zzRicci_bound}
    \ee^{\left(\frac{\alpha}{n-p} -1 \right)w_p} \abs{\nabla w_p}^{\alpha+p-2} \Ric \left( \nu^{\sml{p}}, \nu^{\sml{p}} \right) \leq \ee^{\left(\frac{\alpha}{n-p}-1 \right)T} \abs{\nabla w_p}^{\alpha+p-2} \norm{\Ric}_{L^\infty(D)}
\end{equation}
holds on $\Omega^{\sml{p}}_T \smallsetminus \Omega$. Hence, there exists a positive constant $\kst$ depending only on $T$, such that
\begin{equation}
    \int_0^T\left(\int_{\Omega^{\sml{p}}_t\smallsetminus \Omega} \ee^{\left(\frac{\alpha}{n-p} - 1\right)w_p}\abs{\nabla w_p}^{\alpha+p-2} \Ric(\nu^{\sml{p}},\nu^{\sml{p}}) \dif \Hff^n\right)^2 \dif t \leq \kst\abs{D}^2\norm{\Ric}^2_{L^\infty(D)} \norm{\nabla w_p}^{2(\alpha+p-2)}_{L^\infty(\Omega^{\sml{p}}_T \smallsetminus \Omega^\circ)}
\end{equation}
which is bounded by \cref{prop:uniform_gradient}, since $\alpha>1$. In particular, the family of maps
\begin{equation}\label{eq:zzfamily_of_p_ricci_maps}
    t \mapsto \int_{\Omega^{\sml{p}}_t\smallsetminus \Omega} \ee^{\left(\frac{\alpha}{n-p} - 1\right)w_p}\abs{\nabla w_p}^{\alpha+p-2} \Ric(\nu^{\sml{p}},\nu^{\sml{p}}) \dif \Hff^n
\end{equation}
is weakly compact in $L^2(0,T)$. Pick any weak convergent subsequence in $L^2(0,T)$. By \cref{thm:improved_convergence}, there exists a further subsequence $(p_k)_{k\in \N}$ converging to $1$ as $k\to +\infty$, such that $\nabla w_{p_k} \to \nabla w_1$ almost everywhere on $\Omega^{\sml{1}}_T \smallsetminus \Omega$. Hence, 
\begin{equation}\label{eq:zzaeconvergence}
    \ee^{\left(\frac{\alpha}{n-p_k} -1 \right)w_{p_k}} \abs{\nabla w_{p_k}}^{\alpha+{p_k}-2} \Ric \left( \nu^{\sml{{p_k}}}, \nu^{\sml{{p_k}}} \right)\xrightarrow{k\to +\infty} \ee^{\left(\frac{\alpha}{n-1} -1 \right)w_1} \abs{\nabla w_1}^{\alpha-1} \Ric \left( \nu^{\sml{1}}, \nu^{\sml{1}} \right)  
\end{equation}
almost everywhere on $\Omega_T^{\sml{1}}\smallsetminus\Crit(w_1)$. Since $\alpha>1$, \cref{eq:zzaeconvergence} can be extended to the whole $\Omega_T^{\sml{1}} \smallsetminus \Omega$ by \cref{eq:zzRicci_bound}. Moreover, \cref{eq:zzRicci_bound} also guarantees that the dominated convergence theorem can be applied. Hence,
\begin{equation}\label{eq:zzz_ae_convergence_integrals}
\begin{multlined}[c][.8\textwidth]
    \lim_{k\to +\infty}\int_{\Omega^{\sml{{p_k}}}_t\smallsetminus \Omega} \ee^{\left(\frac{\alpha}{n-{p_k}} - 1\right)w_{p_k}}\abs{\nabla w_{p_k}}^{\alpha+{p_k}-2} \Ric(\nu^{\sml{{p_k}}},\nu^{\sml{{p_k}}}) \dif \Hff^n\\
    = \int_{\Omega^{\sml{1}}_t\smallsetminus \Omega} \ee^{\left(\frac{\alpha}{n-1} - 1\right)w_1}\abs{\nabla w_1}^{\alpha-1} \Ric(\nu^{\sml{1}},\nu^{\sml{1}}) \dif \Hff^n
\end{multlined}
\end{equation}
for every $t\in[0,T]$. Therefore, for any sequence of $p$'s the family of maps \cref{eq:zzfamily_of_p_ricci_maps} admits a further subsequence weakly converging in $L^2(0,T)$ to the right-hand side of \cref{eq:zzz_ae_convergence_integrals}. We get that 
\begin{equation}\label{eq:LimitI3}
    \begin{multlined}[c][.8\textwidth]
    \int_{\Omega^{\sml{{p}}}_t\smallsetminus \Omega} \ee^{\left(\frac{\alpha}{n-{p}} - 1\right)w_{p}}\abs{\nabla w_{p}}^{\alpha+{p}-2} \Ric(\nu^{\sml{{p}}},\nu^{\sml{{p}}}) \dif \Hff^n\xrightharpoonup{p \to 1^+}\\
    \int_{\Omega^{\sml{1}}_t\smallsetminus \Omega} \ee^{\left(\frac{\alpha}{n-1} - 1\right)w_1}\abs{\nabla w_1}^{\alpha-1} \Ric(\nu^{\sml{1}},\nu^{\sml{1}}) \dif \Hff^n
    \end{multlined}
\end{equation}
weakly in $L^2(0,T)$. Plugging \cref{eq:LimitI1,eq:LimitI2,eq:LimitI3} into \cref{eq:main_piece_F_p_monotonicity}, we conclude \cref{eq:zzPrimoPezzo} and thus the proof of the theorem. \endproof

\section{Geometric consequences}\label{sec:geometric_consequences}

This section outlines a series of geometric consequences derived from the previously established monotonicity formulas. Given the extensive background of these formulas, some of these consequences are already present in the literature. For example, Geroch’s monotonicity formula, which we will revisit in \cref{thm:Geroch-montonicity}, was proved along the weak IMCF flow by Huisken and Ilmanen in \cite{huisken_inversemeancurvatureflow_2001}. However, \cref{thm:monotonicty_formula_IMCF} is broad enough to be applied even on manifolds with nonnegative Ricci curvature and Euclidean volume growth. In this setting, we can provide a proof of the Willmore-type and Minkowski-type inequalities, respectively proved in \cite{agostiniani_sharpgeometricinequalitiesclosed_2020} and \cite{benatti_minkowskiinequalitycompleteriemannian_2022}. We also improve the rigidity statement with respect to the original \cite[Theorem 1.2]{benatti_minkowskiinequalitycompleteriemannian_2022}. By monotonicity along the IMCF flow (rather than approximating the inequality with the potential theoretic counterpart), we will show that certain assumptions in the original theorem can also be derived as consequences of rigidity.

The first result discussed below is, to the best of the authors' knowledge, completely original. It is a Gauss--Bonnet-type theorem for level sets $\partial \Omega^{\sml{p}}_t$ of solutions $w_p$ to \cref{eq:moser_potential}. In dimension $n=3$, the level sets of the weak IMCF have enough regularity to be approximated in $W^{2,2}$ by smooth hypersurfaces. This property enabled Huisken and Ilmanen to establish the classic connection between the integral of the induced scalar curvature and the topology of the level set. In the framework of the nonlinear potential theory, such regularity is not available. Although we can define the integral of the induced scalar curvature, we are not able to relate it to the Euler characteristic of the level set, which is, in fact, not well-defined in general. However, we can prove that the integral of the induced scalar curvature of almost every level can only assume discrete values in $8 \pi \Z$.

\subsection{A Gauss--Bonnet-type theorem}\label{sec:GB}
Let $w_p$ be the solution to \cref{eq:moser_potential} for some $\Phi \in \Lip(\overline{D})$ such that $\inf_{\partial D} \Phi >0$, and fix $T< \inf_{\partial D} \Phi$. Since $\partial \Omega^{\sml{p}}_t$ is smooth out of a $\Hff^{n-1}$-negligible set for almost every $t \in [0,T]$, the Gauss equation  
\begin{equation}\label{eq:Gauss-Codazzi}
    \sca^\top = \sca - 2 \Ric(\nu,\nu) + \H^2 - \abs{\h}^2
\end{equation}
defines the induced scalar curvature at each of these regular points. Thanks to \cref{lem:ConvergenzaVarifoldEspToZero}, all the quantities appearing on the right-hand side are defined on the whole level set in a weak sense, hence the quantity
\begin{equation}\label{eq:integral_scalar_curvature}
    \int_{\partial \Omega^{\sml{p}}_t}\sca^\top \dif\Hff^{n-1}= \int_{\partial \Omega^{\sml{p}}_t}\sca - 2 \Ric(\nu,\nu) + \H^2 - \abs{\h}^2\dif \Hff^{n-1}
\end{equation}
is well-defined for almost every $t \in [0,T]$.

In dimension $n=3$, one would like to relate the quantity \cref{eq:integral_scalar_curvature} to the topological properties of the level sets as the Gauss--Bonnet theorem would dictate. \textit{A priori}, it is not even clear whether the integral of the induced scalar curvature assumes discrete values. Indeed, due to the lack of regularity of solutions $w_p$ to \cref{eq:moser_potential}, the integral of the induced scalar curvature as in \cref{eq:integral_scalar_curvature} is not related to the topology of the level sets $\partial \Omega^{\sml{p}}_t$. The topological interpretation is instead clear for level sets of the solutions $w^\varepsilon_p$ of \cref{eq:eps_moser_potential}. In the spirit of $\varepsilon$-approximation, we aim to obtain a weak version of the Gauss--Bonnet Theorem in the limit $\varepsilon\to 0$. This will be achieved in \cref{cor:GB} below (for another weak version of Gauss--Bonnet theorem obtained by approximation we refer to \cite{kuwert_conformalimmersions_2012} that considers immersed surface in $\mathbb{R}^n$). The main difficulty here is that the right-hand side of \cref{eq:integral_scalar_curvature} is neither lower nor upper semicontinuous under varifold convergence. However, we will show the full convergence of $\int_{\partial \Omega^{\sml{\varepsilon}}_t} \sca^\top \dif \Hff^{n-1}$ to \cref{eq:integral_scalar_curvature} as $\varepsilon \to 0^+$.

\begin{proposition}\label{prop:convergenza-seconda-forma}
    Let $p\in(1,2]$ and let $\Omega \subset M$ be a closed bounded set with $\CS^{1,\beta}$-boundary, for some $\beta>0$. Assume that $\partial \Omega$ has generalized mean curvature $\H\in L^2(\partial \Omega)$. Let $w_p$ (resp. $w^\varepsilon_p$) be the solution to the problem \cref{eq:moser_potential} (resp. \cref{eq:eps_moser_potential}) for some $\Phi \in \Lip(\overline{D})$, where $D$ is an open precompact subset with smooth boundary containing $\Omega$, such that $\inf_{\partial D} \Phi >0$. Fix $T<\inf_{\partial D} \Phi $. Let $(\varepsilon_k)_{k \in \N}$ be a vanishing sequence such that $\partial \Omega_t^{\sml{\varepsilon_k}}$ converges to $\partial \Omega^{\sml{p}}_t$ in the sense of varifold for almost every $t \in [0,T]$. Then,
    \begin{align}\label{eq:ConvergenceMeanCurvature}
    \lim_{k\to +\infty}\int_{\partial \Omega^{\sml{\varepsilon_k}}_t} \abs{\H^{\sml{\varepsilon_k}}}^2\dif\Hff^{n-1}&=\int_{\partial \Omega_t^{\sml{p}}}\abs{\H^{\sml{p}}}^2\dif\Hff^{n-1},\\\label{eq:ConvergenceSecondFundamentalForm}
    \lim_{k\to +\infty}\int_{\partial \Omega_t^{\sml{\varepsilon_k}}}\abs{{\h}^{\sml{\varepsilon_k}}}^2\dif\Hff^{n-1}&=\int_{\partial \Omega^{\sml{p}}_t} \abs{{\h}^{\sml{p}}}^2\dif\Hff^{n-1},
\end{align}
    for almost every $t \in [0,T]$.
\end{proposition}

\begin{proof}
We can assume, up to removing a negligible subset of $[0,T]$, that $\partial \Omega^{\sml{\varepsilon_k}}_t$ is smooth for every $t$  in a fixed set of full measure for every $k\in \N$. Moreover, \cref{lem:ConvergenzaVarifoldEspToZero} implies
\begin{equation}\label{eq:zzzsemicontinuitytraceless}
     \liminf_{k\to +\infty}\int_{\partial \Omega_t^{\sml{\varepsilon_k}}}\abs{\mathring{\h}^{\sml{\varepsilon_k}}}^2\dif\Hff^{n-1}\geq\int_{\partial \Omega^{\sml{p}}_t} \abs{\mathring{\h}^{\sml{p}}}^2\dif\Hff^{n-1},
\end{equation}
for almost every $t \in(0,T)$.

\medskip

Consider now a generic function $F\in\CS^0([0,1])$. Since $F(\Err_\varepsilon) \nabla w_p^\varepsilon$ uniformly converges to $F(0) \nabla w_p$ and $F(\Err_\varepsilon) \nabla \abs{\nabla w_p^\varepsilon}$ converges to $F(0) \nabla \abs{\nabla w_p}$ weakly in $L^2_{\loc}(D \smallsetminus \Omega)$, we have
\begin{align}\label{eq:zzlscDoppioProdottoMeanCurcature}
    \lim_{k\to +\infty} \int_{\partial \Omega^{\sml{\varepsilon_k}}_t} F(\Err_{\varepsilon_k}) \abs{ \nabla w_p^{\varepsilon_k}}^2\dif \Hff^{n-1} &= F(0) \int_{\partial \Omega_{t}^{\sml{p}}} \abs{ \nabla w_p}^2\dif \Hff^{n-1},\\ \intertext{and}\label{eq:zzlscGradienteQuadrato}
    \lim_{k\to +\infty} \int_{\partial \Omega^{\sml{\varepsilon_k}}_t} F(\Err_{\varepsilon_k}) \ip{\vec{\H}^{\sml{\varepsilon_k}}| \nabla w_p^{\varepsilon_k}}\dif \Hff^{n-1} &= F(0) \int_{\partial \Omega_{t}^{\sml{p}}} \ip{\vec{\H}^{\sml{p}}|\nabla w_p}\dif\Hff^{n-1},
\end{align}
for almost every $t \in(0,T)$.

\smallskip 
Assume now $F$ is nonnegative. Let $\delta>0$ and let $\chi_\delta$ be a smooth cut-off function satisfying 
    \begin{equation}
        \chi_\delta(t) \coloneqq\begin{cases}
            &\chi_\delta(t) =0& & \text{if $t \leq \delta$,}\\
            &0\leq \chi'_\delta(t) \leq 2/\delta&& \text{if $t \in [\delta, 2 \delta]$,}\\
            &\chi_\delta(t)=1& &\text{if $t \geq2 \delta$.}
        \end{cases}
    \end{equation}
The function $\chi_\delta(\abs{\nabla w^\varepsilon_p})F(\Err_\varepsilon)$ uniformly converges to $\chi_\delta(\abs{\nabla w_p})F(0)$ as $\varepsilon \to 0^+$. By smooth convergence far from $\Crit(w_p)$, we get
{\allowdisplaybreaks
\begin{align}
    \liminf_{k\to +\infty} \int_{\partial \Omega^{\sml{\varepsilon_k}}_t}F (\Err_{\varepsilon_k}) \abs{\H^{\sml{\varepsilon_k}}}^2\dif \Hff^{n-1} &\geq \liminf_{k\to +\infty} \int_{\partial \Omega^{\sml{\varepsilon_k}}_t} F(\Err_{\varepsilon_k})\chi_\delta (\abs{\nabla w_p^{\varepsilon_k}}) \abs{\H^{\sml{\varepsilon_k}}}^2\dif \Hff^{n-1} \\
    &\geq F(0) \int_{\partial \Omega^{\sml{p}}_t} \chi_\delta(\abs{\nabla w_p})\abs{\H^{\sml{p}}}^2 \dif \Hff^{n-1}.
\end{align}}
Passing to the limit as $\delta \to 0$, by the monotone convergence theorem we get
\begin{equation}\label{eq:zzlscMeanCurvatureWithWeight}
    \liminf_{k\to +\infty} \int_{\partial \Omega^{\sml{\varepsilon_k}}_t} F(\Err_{\varepsilon_k}) \abs{\H^{\sml{\varepsilon_k}}}^2\dif \Hff^{n-1} \geq F(0) \int_{\partial \Omega^{\sml{p}}_t}\abs{\H^{\sml{p}}}^2 \dif \Hff^{n-1}
\end{equation}
Since 
\begin{equation}
    \abs{\H^{\sml{\varepsilon_k}}}^2 = \left(1 + \frac{2-p}{p-1} \Err_{\varepsilon_k}\right)^2 \left[\frac{\abs{\nabla^\perp \abs{\nabla w_p^{\varepsilon_k}}}^2}{\abs{\nabla w_p^{\varepsilon_k}}^2}- \abs{\nabla w_p^{\varepsilon_k}}^2\right] + 2 \left(1 + \frac{2-p}{p-1} \Err_{\varepsilon_k}\right) \ip { \vec{\H}^{\sml{\varepsilon_k}}| \nabla w^{\varepsilon_k}_p},
\end{equation}
by \cref{eq:MeanCurvatureLevelWepsP}, plugging in \cref{eq:zzlscMeanCurvatureWithWeight,eq:zzlscGradienteQuadrato,eq:zzlscDoppioProdottoMeanCurcature} we conclude
\begin{equation}\label{eq:zzzsemicontoniutynormal}
    \liminf_{k\to +\infty} \int_{\partial \Omega^{\sml{\varepsilon_k}}_t} F(\Err_{\varepsilon_k}) \frac{\abs{\nabla^\perp \abs{\nabla w_p^{\varepsilon_k}}}^2}{\abs{\nabla w_p^{\varepsilon_k}}^2}\dif \Hff^{n-1} \geq F(0) \int_{\partial \Omega^{\sml{p}}_t} \frac{\abs{\nabla^\perp \abs{\nabla w_p}}^2}{\abs{\nabla w_p}^2} \dif \Hff^{n-1}.
\end{equation}

\medskip

We can prove a similar result for $\nabla^\top\abs{\nabla w^\varepsilon_p}$. Indeed, for $\delta>0$ consider $\log( \abs{\nabla w_p^{\varepsilon_k}} + \delta)$, which belongs to $\CS^0_{\loc} \cap W^{1,2}_{\loc}(D\smallsetminus \Omega^\circ)$ by \cref{prop:StimeBaseWepsilonP}. Let also $X$ be a vector field in $\CS^1_c(D\smallsetminus\Omega)$. Applying \cref{eq:IntegrationByPartsSobolevOnVarifold}, we can write
    \begin{equation}
        \int_{\partial \Omega^{\sml{\varepsilon_k}}_t} \ip{\frac{\nabla^\top \abs{\nabla w_p^{\varepsilon_k}}}{\abs{\nabla w_p^{\varepsilon_k}}+ \delta}| X}\dif \Hff^{n-1} = -\int_{\partial \Omega^{\sml{\varepsilon_k}}_t}  \log\left(\abs{\nabla w_p^{\varepsilon_k}}+ \delta\right)\left(\ip{\vec{\H}^{\sml{\varepsilon_k}}|X}+\div_{\top}X \right) \dif \Hff^{n-1}.
    \end{equation}
    Since $\abs{\nabla w^{\varepsilon_k}_p} \to \abs{ \nabla w_p}$ in $\CS^0_{\loc}(D\smallsetminus\Omega^\circ)$, the varifold convergence of $\partial \Omega^{\sml{\varepsilon_k}}_t$ to $\partial \Omega^{\sml{p}}_t$ as $k\to +\infty$ implies
    \begin{equation}
    \begin{split}
        &\lim_{k\to +\infty}\int_{\partial \Omega^{\sml{\varepsilon_k}}_t}  \log\left( \abs{ \nabla w_p^{\varepsilon_k}}+\delta\right)\left(\ip{\vec{\H}^{\sml{\varepsilon_k}}|X}+ \div_{\top}X \right) \dif \Hff^{n-1} \\
        &\qquad=\int_{\partial \Omega^{\sml{p}}_t}  \log\left( \abs{ \nabla w_p}+\delta\right)\left(\ip{\vec{\H}^{\sml{p}}|X}+\div_{\top}X \right) \dif \Hff^{n-1} =-\int_{\partial \Omega^{\sml{p}}_t} \ip{\frac{\nabla^\top \abs{\nabla w_p}}{\abs{\nabla w_p}+ \delta}| X}\dif \Hff^{n-1},     
    \end{split}
    \end{equation}
    which is well-defined since $\log( \abs{\nabla w_p} + \delta) \in \CS^0_{\loc} \cap W^{1,2}_{\loc}(D\smallsetminus\Omega^\circ)$. Since $\partial \Omega^{\sml{\varepsilon_k}}_t$ converges to $\partial \Omega^{\sml{p}}_t$ as varifolds, the lower semicontinuity part of \cref{lem:JointConvergenceLSC} implies
    \begin{equation}
        \int_{\partial \Omega^{\sml{p}}_t} 
       \frac{  \abs{\nabla^\top\abs{\nabla w_p}}^2}{(\abs{\nabla w_p} + \delta)^2} \dif \Hff^{n-1}
        \leq \liminf_{k\to+\infty} \int_{\partial \Omega^{\sml{\varepsilon_k}}_t} \frac{\abs{\nabla^\top\abs{\nabla w_p^{\varepsilon_k}}}^2}{(\abs{\nabla w_p^{\varepsilon_k}} + \delta)^2} \dif \Hff^{n-1} 
        \le
         \liminf_{k\to+\infty} \int_{\partial \Omega^{\sml{\varepsilon_k}}_t} 
        \frac{ \abs{\nabla^\top\abs{\nabla w_p^{\varepsilon_k}}}^2 }{\abs{\nabla w_p^{\varepsilon_k}}^2}
         \dif \Hff^{n-1} .
    \end{equation}
    Letting now $\delta\to 0$, monotone convergence theorem yields
    \begin{equation}\label{eq:zzzsemicontinuitytangential}
         \liminf_{k\to+\infty} \int_{\partial \Omega^{\sml{\varepsilon_k}}_t} 
        \frac{ \abs{\nabla^\top\abs{\nabla w_p^{\varepsilon_k}}}^2 }{\abs{\nabla w_p^{\varepsilon_k}}^2}
         \dif \Hff^{n-1}\geq \int_{\partial \Omega^{\sml{p}}_t} 
       \frac{  \abs{\nabla^\top\abs{\nabla w_p}}^2}{\abs{\nabla w_p}^2} \dif \Hff^{n-1}
    \end{equation}
\medskip

Consider now $Y_\varepsilon$ defined in \cref{eq:defY} for $\alpha =3-p$ and let $\psi \in \Lip_{c}(0,T)$ be a nonnegative function. By the definition $\MF_p$ and $\IF_p$, in \cref{eq:monotonicty_formula} and \cref{eq:pezzo-di-gradiente} respectively, one sees that  
\begin{equation}
    \int_{\partial \Omega^{\sml{p}}_t} \ip{Y| \nu^{\sml{p}}} \dif \Hff^{n-1} =\begin{multlined}[t][.5\textwidth] \ee^{-\frac{3-n}{n-p}t}\left[\MF_p(t) - \left(\frac{p-1}{n-p}- \frac{1}{3-p} \right) \IF_p(t)\right]\\+ \int_0^t \ee^{\frac{3-n}{n-p}(s-t)}\int_{\partial \Omega^{\sml{p}}_s} \Ric(\nu^{\sml{p}}, \nu^{\sml{p}})\dif \Hff^{n-1} \dif s.
    \end{multlined}
\end{equation}
By \cref{thm:monotonicity-NPT} and \cref{lem:continuity_G}, the function
\begin{equation}
    t \mapsto \int_{\partial \Omega^{\sml{p}}_t} \ip{Y| \nu^{\sml{p}}} \dif \Hff^{n-1}
\end{equation}
is $W^{1,1}_{\loc}(0,T)$ and we have an explicit expression for its derivative. Testing the distributional derivative against a nonnegative $\psi \in \Lip_c(0,T)$ and using coarea formula we get
\begin{equation}\label{eq:zzintegrationbypartpconvergenceeps}
\int_{D \smallsetminus \Crit(w_p)}\div(Y) \psi(w_p) \dif \Hff^n =    -\int_{D \smallsetminus \Omega} \ip{Y|\nabla [\psi(w_p)]}\dif \Hff^n,
\end{equation}
where the expression of $\div (Y)$ away from the critical set of $w_p$ is given in \cref{lem:divY}.

Since $w^\varepsilon_p$ converges to $w_p$ in $\CS_{\loc}^1(D\smallsetminus \Omega)$ and $\nabla \abs{\nabla w_p^\varepsilon}$ converges to $\nabla \abs{\nabla w_p}$ weakly in $L_{\loc}^2(D\smallsetminus \Omega)$ as $\varepsilon \to 0$ (see \cref{prop:RegolaritaWp}), we can easily show that
\begin{equation}\label{eq:zzzconvergenceYandintegrationbypart}
\begin{split}
    -\int_{D\smallsetminus \Omega} \ip{Y| \nabla [\psi(w_p)]} \dif \Hff^{n} &= -\lim_{\varepsilon\to 0^+} \int_{D\smallsetminus \Omega} \ip{Y_\varepsilon| \nabla \psi(w^\varepsilon_p)} \dif \Hff^{n} \\&= \lim_{\varepsilon\to 0^+} \int_{D\smallsetminus \Omega}\psi(w^\varepsilon_p) \abs{\nabla w^\varepsilon_p} (\mathcal{D}^{\sml{\varepsilon}}_+ + \mathcal{D}^{\sml{\varepsilon}}_\sigma)\dif\Hff^n .
    \end{split}
\end{equation}
where $\mathcal{D}^{\sml{\varepsilon}}_+$ and $\mathcal{D}^{\sml{\varepsilon}}_\sigma$ are defined in \cref{lem:divY}.
On the other hand, by coarea formula 
\begin{align} \label{eq:zzdivergeneceinpieces}
   \int_{D\smallsetminus \Omega}\psi(w^\varepsilon_p) \abs{\nabla w^\varepsilon_p} \mathcal{D}^{\sml{\varepsilon}}_+\dif \Hff^n
    &=   \int_0^T \psi(t)\int_{\partial \Omega^{\sml{\varepsilon}}_t} \abs{\mathring{\h}^{\sml{\varepsilon}}}^2 +  \Psi(\Err_\varepsilon)\frac{\abs{\nabla^\perp \abs{\nabla w^{\varepsilon}_p}}^2}{\abs{\nabla w^{\varepsilon}_p}^2}   +\frac{\abs{\nabla^\top \abs{\nabla w^{\varepsilon}_p}}^2}{\abs{\nabla w^{\varepsilon}_p}^2} \dif \Hff^{n-1}\dif t
\end{align}
where
\begin{equation}
    \Psi(s) \coloneqq (p-1)^2 \left(1+ \frac{2-p}{p-1}s\right) \left[\frac{1}{p-1} - \frac{n-2}{n-1} \left(1+ \frac{2-p}{p-1} s\right)\right]>0.
\end{equation}
Using again the convergence properties of $w_p^\varepsilon$ to $w_p$ as $\varepsilon \to 0^+$ (see \cref{prop:RegolaritaWp}), we have
\begin{equation}\label{eq:zzconvergencesigmaD}
    \lim_{\varepsilon\to 0^+}\int_{D \smallsetminus \Omega} \psi(w^{\varepsilon}_p)\abs{\nabla w_p^{\varepsilon}} \mathcal{D}^{\sml{\varepsilon}}_\sigma  \dif \Hff^n = \int_{D \smallsetminus \Crit(w_p)} \psi(w_p)\abs{\nabla w_p}\mathcal{D}_\sigma \dif \Hff^{n}.
\end{equation} 

By applying \cref{eq:zzzsemicontinuitytraceless,eq:zzzsemicontinuitytangential,eq:zzzsemicontoniutynormal} and Fatou's lemma, one can see that the left-hand side of \cref{eq:zzdivergeneceinpieces} is lower semicontinuous as $k\to +\infty$. Together with \cref{eq:zzconvergencesigmaD}, it yields
\begin{equation}\label{eq:zzzfatoulimitcontradiction}
    \liminf_{k\to +\infty}\int_{D \smallsetminus \Omega} \psi(w^{\varepsilon_k}_p) \abs{\nabla w_p^{\varepsilon_k}} ( \mathcal{D}^{\sml{\varepsilon}}_+ + \mathcal{D}^{\sml{\varepsilon}}_\sigma) \dif \Hff^n\geq \int_{D\smallsetminus \Crit(w_p)} \psi(w_p) \div(Y) \dif \Hff^n.
\end{equation}
Suppose that the inequality in any of \cref{eq:zzzsemicontinuitytraceless,eq:zzzsemicontinuitytangential,eq:zzzsemicontoniutynormal} is strict on a subset of positive measure of $[0,T]$. Then, the inequality would be strict in \cref{eq:zzzfatoulimitcontradiction}, which leads to a contradiction by \cref{eq:zzzconvergenceYandintegrationbypart,eq:zzintegrationbypartpconvergenceeps}. The same argument can be repeated for any subsequence of $(\varepsilon_k)_{k\in \N}$. Therefore, we get
\begin{equation}\label{eq:zzzconvergenceSecondTraceless}
    \lim_{k \to +\infty} \int_{\partial \Omega_t^{\sml{\varepsilon_k}}} \abs{\mathring{\h}^{\sml{\varepsilon_k}}}^2 \dif \Hff^{n-1} = \int_{\partial \Omega_t^{\sml{p}}} \abs{\mathring{\h}^{\sml{p}}}^2 \dif \Hff^{n-1}
\end{equation}
and
\begin{equation}\label{eq:zzconvergencenormalderivative}
    \lim_{k \to +\infty} \int_{\partial \Omega_t^{\sml{\varepsilon_k}}} \Psi(\Err_{\varepsilon_k})\frac{\abs{\nabla^\perp \abs{\nabla w_p^{\varepsilon_k}}}^2}{\abs{\nabla w_p^{\varepsilon_k}}^2} \dif \Hff^{n-1} = \Psi(0)\int_{\partial \Omega_t^{\sml{p}}} \frac{\abs{\nabla^\perp \abs{\nabla w_p}}^2}{\abs{\nabla w_p}^2} \dif \Hff^{n-1}
\end{equation}
for almost every $t$. Since $\Psi(s)>0$ for $s \in [0,1]$, there exists 
$\kst>0$ such that $\Psi(s)-\kst s>0$. Hence, rewriting $\Psi(s)=\Psi(s)-\kst s + \kst s $ in \cref{eq:zzdivergeneceinpieces}, with the same argument leading to \cref{eq:zzconvergencenormalderivative}, one also obtains
\begin{equation}\label{eq:zzvanishingoftherest}
    \lim_{k \to +\infty} \int_{\partial \Omega_t^{\sml{\varepsilon_k}}} \Err_{\varepsilon_k}\frac{\abs{\nabla^\perp \abs{\nabla w_p^{\varepsilon_k}}}^2}{\abs{\nabla w_p^{\varepsilon_k}}^2} \dif \Hff^{n-1} = 0.
\end{equation}
By \cref{eq:zzlscGradienteQuadrato,eq:zzlscDoppioProdottoMeanCurcature,eq:zzconvergencenormalderivative,eq:zzvanishingoftherest} we get \cref{eq:ConvergenceMeanCurvature}
by \cref{eq:MeanCurvatureLevelWepsP}. Finally \cref{eq:ConvergenceMeanCurvature} and \cref{eq:zzzconvergenceSecondTraceless} imply \cref{eq:ConvergenceSecondFundamentalForm} and this concludes the proof of the proposition.
\end{proof}

\begin{corollary}\label{cor:scalar_curvature_convergence}
Under the same hypotheses of \cref{prop:convergenza-seconda-forma}, we have
    \begin{equation}
        \lim_{k\to +\infty}\int_{\partial \Omega^{\sml{\varepsilon_k}}_t} \sca_{\sml{\varepsilon_k}}^\top\dif\Hff^{n-1} = \int_{\partial \Omega^{\sml{p}}_t} \sca_{\sml{p}}^\top\dif\Hff^{n-1},
    \end{equation}
for almost every $t \in [0,T]$, where $\sca^\top$ denotes the induced scalar curvature on the level sets.
\end{corollary}

\begin{proof}
    By \cref{eq:integral_scalar_curvature} and \cref{prop:convergenza-seconda-forma}, it only remains to show that 
    \begin{align}
         \lim_{k \to +\infty} \int_{\partial \Omega^{\sml{\varepsilon_k}}_t} \sca\dif\Hff^{n-1} &= \int_{\partial \Omega^{\sml{p}}_t} \sca\dif\Hff^{n-1},\label{eq:scalare_globale_convergenza}\\\intertext{and}\label{eq:Ricci_globale_convergenza}
           \lim_{k\to +\infty} \int_{\partial \Omega^{\sml{\varepsilon_k}}_t} \Ric(\nu^{\sml{\varepsilon_k}},\nu^{\sml{\varepsilon_k}})\dif\Hff^{n-1} &= \int_{\partial \Omega^{\sml{p}}_t} \Ric(\nu^{\sml{p}},\nu^{\sml{p}})\dif\Hff^{n-1}.
    \end{align}
    Both \cref{eq:scalare_globale_convergenza,eq:Ricci_globale_convergenza} follow by the very definition of varifold convergence in \cref{def:RectifiableVarifold}.
\end{proof}

\begin{corollary}[Gauss--Bonnet-type theorem]\label{cor:GB}
    Under the same hypotheses of \cref{prop:convergenza-seconda-forma}, if $(M,g)$ is of dimension $n=3$, there exists finite the limit
        \begin{equation}
        \lim_{k \to +\infty} \chi(\partial \Omega^{\sml{\varepsilon_k}}_t) = \frac{1}{4\pi}\int_{\partial \Omega^{\sml{p}}_t} \sca^\top_{\sml{p}} \dif \Hff^{n-1} ,
    \end{equation}
        for almost every $t\in[0,T]$, where $\chi(\partial \Omega^{\sml{\varepsilon_k}}_t)$ is the Euler characteristic of $\partial \Omega^{\sml{\varepsilon_k}}_t$. 
\end{corollary}

\begin{proof}[Proof of \cref{cor:GB}]
    Up to removing a negligible set, we can assume that $\partial \Omega^{\sml{\varepsilon_k}}_t$ is smooth for every $k \in \N$. The result follows by Gauss--Bonnet theorem and \cref{cor:scalar_curvature_convergence}.
\end{proof}

\begin{proof}[Proof of \cref{GB}]
    The theorem follows by \cref{cor:GB} and $\chi(\partial \Omega^{\sml{\varepsilon_k}}_t)\in 2\Z$ since $\partial \Omega^{\sml{\varepsilon_k}}_t$ is a boundary.
\end{proof}

\subsection{Willmore-type and Minkowski-type inequalities} 
Let $(M,g)$ be a complete noncompact Riemannian manifold of dimension $n$ with nonnegative Ricci curvature and Euclidean volume growth, namely
\begin{equation}
    \AVR(g)\coloneqq \lim_{r \to +\infty} \frac{\abs{B_r(o)}}{\abs{\B^n}r^n} >0.
\end{equation}
Let $\Omega\subset M$ be a bounded set with $\CS^1$  boundary. Then by \cite[Corollary 1.3]{brendle_sobolevinequalitiesmanifoldsnonnegative_2021} we know that the following sharp isoperimetric inequality holds
\begin{equation}\label{eq:brendle_isoperimetric}
    \frac{\abs{\S^{n-1}}^{n}}{\abs{\B^n}^{n-1}}\AVR(g)  \leq \frac{\abs{\partial \Omega}^{n}}{\abs{\Omega}^{n-1}}.
\end{equation}
In particular, by \cite{xu_isoperimetrypropernessweakinverse_2023} the proper IMCF $w_1$ starting from $\Omega$ exists globally on $M\smallsetminus\Omega$ and thus it can be approximated by solutions $w_p$ to \cref{eq:moser_potential} locally uniformly in the sense of \cref{thm:local-approximation}. Hence, we can prove the following result. 

\begin{theorem}\label{thm:minkowski}
    Let $(M,g)$ be a complete noncompact Riemannian manifold of dimension $n\ge 3$ with nonnegative Ricci curvature and Euclidean volume growth. Fix $\alpha\geq 1$. Then, for any bounded set $\Omega\subset M$ with smooth boundary, there holds
    \begin{equation}\label{eq:minkowski}
        \abs{\S^{n-1}}^\frac{\alpha}{n-1} \AVR(g)^{\frac{\alpha}{n-1}}\leq \abs{\partial \Omega^*}^{\frac{\alpha}{n-1}-1}\int_{\partial \Omega} \abs{\frac{\H}{n-1}}^\alpha\dif \Hff^{n-1}.
    \end{equation}
    Moreover, equality holds in \cref{eq:minkowski} if and only if $\Omega$ is strictly outward minimizing with strictly mean-convex boundary and $(M\smallsetminus \Omega,g)$ is isometric to the conical end
    \begin{align}
        \left(\partial \Omega\times[r_0,+\infty),\; \dd r^2 + \left(\frac{r}{r_0}\right)^2 g_{\partial \Omega}\right), && \text{where }r_0 \coloneqq \left(\frac{\abs{\partial \Omega}}{\AVR(g) \abs{\S^{n-1}}}\right)^{\frac{1}{n-1}}.
    \end{align}
\end{theorem}
In \cref{thm:minkowski}, the case $\alpha=n-1$ yields the Willmore-type inequality in \cite{agostiniani_sharpgeometricinequalitiesclosed_2020}, while the case $\alpha =1 $ is the Minkowski-type inequality in \cite{benatti_minkowskiinequalitycompleteriemannian_2022}. For $\alpha=1$, the rigidity statement in \cref{thm:minkowski} is an improved version of \cite[Theorem 1.2]{benatti_minkowskiinequalitycompleteriemannian_2022}. Indeed, being $\Omega$ outward minimizing and with strictly mean-convex boundary is here a consequence of the rigidity rather than an assumption (cf. \cite[Theorem 1.2]{benatti_minkowskiinequalitycompleteriemannian_2022}). 

\begin{proof}
    Consider the function
    \begin{equation}
    \mathcal{M}_\alpha(\partial \Omega)=\abs{\partial \Omega^*}^{\frac{\alpha}{n-1} -1}\int_{\partial \Omega}\abs{\frac{\H}{n-1}}^{\alpha} \dif \Hff^{n-1}.
\end{equation}
By \cref{thm:monotonicty_formula_IMCF}, for any $\Omega \subseteq M$ with $\CS^{1,1}$-boundary the map $t \mapsto \mathcal{M}_{\alpha}(\partial \Omega^{\sml{1}}_t)$ is monotone nonincreasing and
\begin{equation}\label{eq:minkowski_monotonicity}
    \frac{\dd}{\dd t} \mathcal{M}_\alpha(\partial \Omega_t^{\sml{1}})\leq -\frac{\alpha}{(n-1)^\alpha}\abs{\partial \Omega_t^{\sml{1}}}^{\frac{\alpha}{n-1}-1}\int_{\partial \Omega^{\sml{1}}_t} \H^{\alpha-2} \left( (\alpha-1) \frac{\abs{\nabla^\top\H}^2}{\H^2} + \abs{\mathring{\h}}^2 + \Ric(\nu,\nu)\right) \dif \Hff^{n-1} 
\end{equation}
in the distributional sense on $[0,+\infty)$. By contradiction, suppose that there exists $\Omega \subseteq M$ bounded such that $\mathcal{M}_\alpha(\partial \Omega)< ( \AVR(g)\abs{\S^{n-1}})^{\alpha/(n-1)}$. Evolve $\Omega$ by IMCF. Applying de l'H\^opital rule (see \cite[Appendix A]{benatti_isoperimetricriemannianpenroseinequality_2022}) we have
\begin{equation}\label{eq:zz_hopitalminkowski}
    \liminf_{t \to +\infty}\frac{\abs{\Omega_t^{\sml{1}}}}{\abs{\partial \Omega^{\sml{1}}_t }^{\frac{n}{n-1}}} \geq \liminf_{t \to +\infty}\frac{\abs{\Omega_t^{\sml{1}}\smallsetminus \Crit(w_1)}}{\abs{\partial \Omega^{\sml{1}}_t }^{\frac{n}{n-1}}} \geq\liminf_{t \to +\infty} \frac{n-1}{n}\abs{\partial \Omega^{\sml{1}}_t }^{-\frac{n}{n-1}}  \int_{\partial \Omega^{\sml{1}}_t} \frac{1}{\H}\dif \Hff^{n-1}.
\end{equation}
By H\"older inequality, we have
\begin{equation}\label{eq:zz_holder_minkowski}
    \int_{\partial \Omega^{\sml{1}}_t} \frac{1}{\H}\dif \Hff^{n-1}\geq \abs{\partial \Omega^{\sml{1}}_t}^{\frac{1}{\alpha}+1}\left(\int_{\partial \Omega^{\sml{1}}_t} \H^\alpha\dif \Hff^{n-1}\right)^{-\frac{1}{\alpha}}.
\end{equation}
Then, coupling \cref{eq:zz_holder_minkowski,eq:zz_hopitalminkowski} we have
\begin{align}
    \liminf_{t \to +\infty}\frac{\abs{\Omega_t^{\sml{1}}}}{\abs{\partial \Omega^{\sml{1}}_t }^{\frac{n}{n-1}}}  \geq \liminf_{t \to+\infty}\frac{1}{n} \mathcal{M}_\alpha(\partial \Omega^{\sml{1}}_t)^{-\frac{1}{\alpha}} > \frac{\abs{\B^n}}{\abs{\S^{n-1}}^{\frac{n}{n-1}}}\AVR(g)^{-\frac{1}{n-1}},
\end{align}
which contradicts the isoperimetric inequality \cref{eq:brendle_isoperimetric} on $(M,g)$.
\medskip

Suppose now that $\mathcal{M}_\alpha(\partial \Omega) =(\AVR(g)\abs{\S^{n-1}})^{\alpha/(n-1)}$ for some bounded set $\Omega \subset M$ with smooth boundary. We get that the right-hand side in \cref{eq:minkowski_monotonicity} vanishes for almost every $t>0$. Hence $\mathring{\h}$ and $\Ric(\nu,\nu)$ both vanish on $\partial \Omega_t^{\sml{1}}\smallsetminus S_t$, for almost every $t$. Since $\Ric\ge 0$, \cite[Lemma 4.8]{benatti_minkowskiinequalitycompleteriemannian_2022} implies that the mean curvature of $\partial \Omega_t^{\sml{1}}\smallsetminus S_t$ is constant around each point in $\partial \Omega_t^{\sml{1}}\smallsetminus S_t$ (note that \cite[Lemma 4.8]{benatti_minkowskiinequalitycompleteriemannian_2022} is a pointwise result that does not rely on a global smoothness assumption). Moreover, the mean curvature of $\partial \Omega_t^{\sml{1}}\smallsetminus S_t$ is uniformly bounded by the local Lipschitz constant for $w_1$.

All in all, we can fix a sequence $t_k\to0^+$ such that $\abs{\Omega_{t_k}^{\sml{1}} \smallsetminus U_{t_k}^{\sml{1}}}=0$, so that $\partial \Omega^{\sml{1}}_{t_k}$ converges in the $\CS^{1,\beta}$ sense to $\partial \Omega^*$ around each point out of $S_0$ by \cref{prop:regularityLevelsetIMCF}\cref{item:convergenceC1bIMCF}. Since the mean curvature of $\partial \Omega_{t_k}^{\sml{1}} \smallsetminus S_{t_k}$ is uniformly bounded, they also converge in the $W^{2,2}$ sense around such points. Furthermore, we can assume that the mean curvature of $\partial \Omega_{t_k}^{\sml{1}}\smallsetminus S_{t_k}$ is locally constant, i.e., it is constant around any point in $\partial \Omega_{t_k}^{\sml{1}}\smallsetminus S_{t_k}$. 

We now claim that $\Omega=\Omega^*$. Suppose by contradiction that $\Hff^{n-1}(\partial\Omega^* \smallsetminus\partial \Omega)>0$. Hence, there exists a connected component $\Sigma$ of $\partial \Omega^*$ such that $\Hff^{n-1}(\Sigma \smallsetminus\partial \Omega)>0$. We consider two cases. If $\Hff^{n-1}(\Sigma\cap\partial\Omega) = 0$, taking into account \cite[Remark 5.11]{mondino_weaklaplacianboundsminimal_2023}, $\Sigma$ is minimal in the sense of \cite[Definition 5.10]{mondino_weaklaplacianboundsminimal_2023}. On the other side, suppose that $\Hff^{n-1}(\Sigma\cap\partial\Omega) > 0$. By \cref{prop:outwardminimizinghull}\cref{item:hull0}, $S_0 \cap \partial\Omega=\varnothing$. Hence, $\partial \Omega^{\sml{1}}_{t_k}$ locally converges in $W^{2,2}$ to $\Sigma$ in an open neighborhood $V$ of $\Sigma \cap \partial \Omega$ which does not contain $S_0$. Since the mean curvature of $\partial \Omega^{\sml{1}}_{t_k}$ is locally constant, $\Sigma$ has locally constant mean curvature on $V\cap \Sigma $. Each connected component of $V \cap \Sigma $ contains at least a point of $\partial \Omega^*\smallsetminus (\partial \Omega \cup S_0)$, hence the mean curvature of $\Sigma$ vanishes in $V$. Hence also in this case we get that $\Sigma$ is minimal in the sense of \cite[Definition 5.10]{mondino_weaklaplacianboundsminimal_2023}. In both cases, \cite[Theorem 5.12]{mondino_weaklaplacianboundsminimal_2023} (see also \cite[Lemma 2.2]{choefraser_18}) implies that $\Delta {\rm d}_\Sigma \le 0$ in the sense of distributions on $M\smallsetminus \Sigma$, where ${\rm d}_\Sigma$ is the distance from $\Sigma$.  Integrating $\Delta {\rm d}_\Sigma$ on $\set{r<{\rm d}_\Sigma < R}$, for $R$ large enough, the divergence theorem gives that $\Hff^{n-1}(\set{{\rm d}_\Sigma=R}) \le \Hff^{n-1}(\set{{\rm d}_\Sigma=r})$. Hence, $(M,g)$ would have linear volume growth, contradicting the assumption ${\rm AVR}(g)>0$. 

Therefore $\Omega^*=\Omega$ and $\Omega$ is strictly outward minimizing with strictly positive constant mean curvature $\H_j>0$ on each connected component $\Sigma_j$ of $\partial \Omega$, where $j \in \{1,\ldots,m\}$. It follows from \cite[Smooth Start Lemma 2.4, Remark p. 379]{huisken_inversemeancurvatureflow_2001} that $\partial \Omega^{\sml{1}}_t$ is a classical solution to the inverse mean curvature flow starting from $\partial \Omega$ for any $t \in [0,T)$, for some $T>0$. Arguing as in the proof of \cite[Theorem 1.2]{benatti_minkowskiinequalitycompleteriemannian_2022}, recalling that $\mathcal{M}_\alpha(\partial \Omega) =(\AVR(g)\abs{\S^{n-1}})^{\alpha/(n-1)}$, one gets that $(\set{0\leq w_1<T},g)$ is isometric to a disjoint union of truncated cones of the form
\begin{align}\label{eq:zzzsplittingconicolocale}
    \left( \Sigma_i \times [r_i,R_i), \dd r^2 + \left( \frac{r}{r_i}\right)^2 g_{\Sigma_i}\right),&& \text{where } r_i= \frac{n-1}{\H_i}\text{ and }R_i= r_i\ee^{\frac{T}{n-1}},
\end{align}
for $i=1,\ldots, m$.

We now prove that each connected component of $\partial \Omega$ is the boundary of a connected component of $\Omega$.  Take a connected component $\Omega_j$ of $\Omega$. Suppose that $\partial \Omega_j$ is the union of more than two $\Sigma_i$'s. Glueing along the $\Sigma_i$'s touching $\Omega_j$ the cone \cref{eq:zzzsplittingconicolocale} for $R_i = +\infty$, we would obtain a smooth Riemannian manifold with nonnegative Ricci curvature and $\AVR(g)>0$ with more than one end. This would contradict the Cheeger--Gromoll's Splitting Theorem. Therefore, $\partial \Omega_j$ is connected and $\partial \Omega_j = \Sigma_j$, up to relabeling the indices.

We now want to show that $\partial \Omega$ is connected. Assume $m\geq 2$. We can now apply \cite[Lemma 4.9]{benatti_minkowskiinequalitycompleteriemannian_2022} to each $\Sigma_j$ to get
\begin{equation}\label{eq:zzzapertureconiche}
    \abs{\Sigma_i} \geq r_i^{n-1} \abs{\S^{n-1}} \AVR(g)=  \left( \frac{n-1}{ \H_i}\right)^{n-1} \abs{\S^{n-1}} \AVR(g).
\end{equation}
Since $\mathcal{M}_\alpha(\partial \Omega) =(\AVR(g) \abs{\S^{n-1}})^{\alpha/(n-1)}$, we have
\begin{equation}
    \begin{split}
       ( \abs{\mathbb{S}^{n-1}} \mathrm{AVR}(g))^{\frac{\alpha}{n-1}} \abs{\partial\Omega}^{1-\frac{\alpha}{n-1}}
        &= \int_{\partial\Omega} \abs{\frac{\H}{n-1}}^\alpha 
        = \sum_{i=1}^{m} \int_{\Sigma_i} \abs{\frac{\H_i}{n-1}}^\alpha \overset{\cref{eq:zzzapertureconiche}}{\ge} (\abs{\mathbb{S}^{n-1}} \mathrm{AVR}(g))^{\frac{\alpha}{n-1}} \sum_{i=1}^{m} \abs{\Sigma_j}^{1-\frac{\alpha}{n-1}}\\
        &>(\abs{\mathbb{S}^{n-1}} \mathrm{AVR}(g))^{\frac{\alpha}{n-1}} \left( \sum_{i=1}^{m} \abs{\Sigma_j}\right)^{1-\frac{\alpha}{n-1}}= ( \abs{\mathbb{S}^{n-1}} \mathrm{AVR}(g))^{\frac{\alpha}{n-1}} \abs{\partial\Omega}^{1-\frac{\alpha}{n-1}} ,
    \end{split}
\end{equation}
where the strict subadditivity holds since $\alpha \geq 1$ and $m\ge2$. This gives the desired contradiction, thus $m=1$. 

Hence, $\partial \Omega$ is connected and $\Omega^{\sml{1}}_T \smallsetminus \Omega$ splits as
\begin{align}
    \left(\partial \Omega\times[r_0,R_0),\; \dd r^2 + \left(\frac{r}{r_0}\right)^2 g_{\partial \Omega}\right), && \text{where } r_0= \frac{n-1}{\H}\text{ and }R_0= r_0\ee^{\frac{T}{n-1}}.
\end{align}
Since $\partial \Omega$ saturates \cref{eq:minkowski} and has constant mean curvature, we obtain
\begin{equation}
    \abs{\partial \Omega}= \left( \frac{n-1}{\H}\right)^{n-1}\abs{\S^{n-1}}\AVR(g) = r_0^{n-1} \abs{\S^{n-1}} \AVR(g).
\end{equation}
The rigidity part in \cite[Lemma 4.7]{benatti_minkowskiinequalitycompleteriemannian_2022} ensures we can take $R_0=+\infty$, which concludes the proof.
\end{proof}

\subsection{Geroch monotonicity formula} 
Huisken and Ilmanen first introduced weak IMCF with the precise purpose of proving the Riemannian Penrose inequality in \cite{huisken_inversemeancurvatureflow_2001}. This inequality states that the total physical mass of a time-symmetric isolated gravitational system satisfying the dominant energy condition is bounded from below by the mass of a black hole sitting in it. The total mass they considered is the $\ADM$ mass \cite{arnowitt_coordinateinvarianceenergyexpressions_1961}, which is defined as
\begin{equation}
    \ma_{\ADM} \coloneqq \lim_{r \to +\infty} \int_{\set{\abs{x}=r}} (\partial_i g_{ji} - \partial_{j} g_{ii}) \frac{x^j}{\abs{x}} \dif \Hff^{2}.
\end{equation}
An isolated gravitational system is modeled as an asymptotically flat Riemannian manifold $(M,g)$, i.e. the manifold is diffeomorphic outside a compact set to $\R^3 \smallsetminus \set{\abs{x} \leq R}$ for some $R\geq 0$ and the metric $g$ approaches the flat metric $\delta$ at infinity. In our case we will assume that space is $\CS^1_{\tau>1/2}$-asymptotically flat, which is $\abs{g-\delta}= O_1(\abs{x}^{-\tau})$. In the time-symmetric case, the dominant energy condition reads as having nonnegative scalar curvature. Thanks to \cite{bartnik_massasymptoticallyflatmanifold_1986,chrusciel_boundaryconditionsspatialinfinity_1986}, these are the mildest conditions under which the $\ADM$ mass is a well-defined geometric invariant. The horizon of the black hole is represented by a minimal outermost surface, i.e. it is not shielded from infinity by any other minimal surfaces. The radius of the horizon is twice the mass of the black hole. Assuming that the horizon coincides with the boundary of the manifold, the Riemannian Penrose inequality is
\begin{equation}\label{eq:RPI}
    \sqrt{\frac{\abs{\partial M}}{16 \pi }} \leq \ma_{\ADM}.
\end{equation}

The proof of Huisken and Ilmanen goes as follows. Consider the weak IMCF $w_1$ starting from $\partial M$. One can easily see that the Hawking mass of the boundary coincides with the left-hand side of \cref{eq:RPI}. Indeed, the Hawking mass $\ma_H$ of a surface $\Sigma \subseteq M$ is 
\begin{equation}
    \ma_H(\Sigma)\coloneqq \sqrt{\frac{\abs{\Sigma}}{16 \pi} }\left(1-\frac{1}{16 \pi}\int_{\Sigma} \H^2 \dif \Hff^{2}\right)
\end{equation}
and the boundary is minimal. Moreover, a thorough asymptotic analysis (see also \cite{agostiniani_riemannianpenroseinequalitynonlinear_2022,benatti_isoperimetricriemannianpenroseinequality_2022}) shows that
\begin{equation}
    \limsup_{t \to +\infty } \ma_H(\partial \Omega_t^{\sml{1}}) \leq \ma_{\ADM}.
\end{equation}

The inequality in \cref{eq:RPI} is justified by the monotonicity of the Hawking mass along the weak IMCF. Such monotonicity was first shown along the smooth IMCF by Geroch in \cite{geroch_energyextraction_1973}. One of the main contributions by Huisken and Ilmanen was to prove the validity of Geroch's monotonicity formula along the weak flow, overcoming the issues given by the lack of regularity. 

A consequence of \cref{thm:monotonicty_formula_IMCF} is a different proof of \cite[Theorem 5.8]{huisken_inversemeancurvatureflow_2001}, based on $p$-harmonic approximation rather than $\varepsilon$-regularization.

\begin{theorem}\label{thm:Geroch-montonicity}
    Let $(M,g)$ be an asymptotically flat $3$-dimensional Riemannian manifold with nonnegative scalar curvature. Let $\Omega\subseteq M$ be a compact subset with $\CS^{1,1}$-boundary. Then $t\mapsto \ma_H(\partial \Omega^{\sml{1}}_t)$ is essentially monotone nondecreasing and 
    \begin{equation}\label{eq:Georch-type_MF}
        \frac{\dd}{\dd t} \ma_{H}(\partial \Omega^{\sml{1}}_t) \geq \sqrt{\frac{{\abs{\partial \Omega^{\sml{1}}_t}}}{(16 \pi)^3}}\left(4 \pi(2 -  \chi(\partial \Omega^{\sml{1}}_t))+\int_{\partial \Omega^{\sml{1}}_t} 2 \frac{\abs{\nabla^\top \H}^2}{\H^2}+ \abs{\mathring{\h}}^2+\sca \dif \Hff^{2}\right)
    \end{equation}
    in the distributional sense on $[0,+\infty)$.
\end{theorem}

\begin{proof}
    Taking $\alpha=2$ in \cref{thm:monotonicty_formula_IMCF}, we have the following growth formula for the Willmore energy
    \begin{equation}
        \frac{\dd}{\dd t}\int_{\partial \Omega^{\sml{1}}_t} \H^2 \dif \Hff^{2} \leq -\int_{\partial \Omega^{\sml{1}}_t} 2 \frac{\abs{\nabla^\top \H}^2}{\H^2} + 2 \abs{\mathring\h}^2 +2 \Ric(\nu,\nu)\dif \Hff^{2}.
    \end{equation}
    Since $\partial \Omega^{\sml{1}}_t$ is $\CS^1$ and $W^{2,2}$, one can define both the induced scalar curvature as $\sca^\top = \sca - 2 \Ric(\nu, \nu) +\H^2 - \abs{\h}^2$ and the Euler characteristic. Moreover, $\partial \Omega^{\sml{1}}_t$ can be approximated by smooth hypersurfaces in $W^{2,2}$ and $\CS^1$. In particular, Gauss--Bonnet theorem holds for $\partial \Omega^{\sml{1}}_t$ and we have
    \begin{equation}
        \frac{\dd}{\dd t}\int_{\partial \Omega^{\sml{1}}_t} \H^2 \dif \Hff^{n-1} \leq  4 \pi \chi(\partial \Omega^{\sml{1}}_t) -\int_{\partial \Omega^{\sml{1}}_t} 2 \frac{\abs{\nabla^\top \H}^2}{\H^2} + \abs{\mathring\h}^2+ \sca + \frac{\H^2}{2} \dif \Hff^{2}.
    \end{equation}
    Recalling that $\abs{\partial \Omega^{\sml{1}}_t} = \ee^t \abs{\partial \Omega^*}$, we conclude \cref{eq:Georch-type_MF}.
\end{proof}

\begin{remark}\label{rem:ApplicazioniAfterGB}
    Exploiting \cref{thm:monotonicity-NPT} in place of \cref{thm:monotonicty_formula_IMCF} in the proof of \cref{thm:Geroch-montonicity}, and using \cref{cor:GB}, one can explicitly compute the derivative of the quantities studied in \cite{agostiniani_riemannianpenroseinequalitynonlinear_2022,benatti_nonlinearisocapacitaryconceptsmass_2023a,xia_newmonotonicitycapacitaryfunctions_2024}.
\end{remark}

\begin{appendices}
\crefalias{section}{appendix}
\section{Varifold theory and weak convergence on Riemannian manifolds}\label{sec:AppendiceVarifolds}

In this section, we recall basic definitions and facts about varifold theory on Riemannian manifolds. We refer the reader to \cite{simon_lecturesgeometricmeasuretheory_1983, hutchinson_secondfundamentalformvarifolds_1986, ambrosio_functionsboundedvariationfree_2000} for a complete treatment of the theory recalled here (see also \cite{mantegazza_curvaturevarifoldsboundary_1996} and \cite[Appendix 6.6]{mondino_willmorefunctionalothercurvature_2011}).

Throughout this section, let $(M,g)$ be an $n$-dimensional Riemannian manifold and let $U\subset M$ be a fixed compact subset of $M$ that is the closure of some open set $\Int U$. We will always understand without loss of generality that $M$ is isometrically embedded in the Euclidean space $\R^N$ for some $N>n$. Let also $m \in \set{1,\ldots, n-1}$ be fixed.

We denote by $G(N,m)$ the Grassmannian of unoriented $m$-subspaces of $\R^N$. We identify $G(N,m)$ with the set of projection matrices onto an $m$-dimensional subspace of $\R^N$, represented with respect to the canonical basis. We define $G_m(\R^N)\coloneqq \R^N\times G(N,m)$ and
\begin{equation}
    G_m(U) \coloneqq \set{ (x,P) \in U \times G(N,m)  \st P \text{ is a $m$-dimensional subspace of } T_xM}.
\end{equation}
In particular, $G_m(U)$ can be seen as a subset of $\R^N \times \R^{N^2}$.

\begin{definition}\label{def:RectifiableVarifold}
    An $m$-dimensional (integer rectifiable) varifold $V$ on $U$ is a nonnegative Radon measure on $G_m(U)$ identified as follows. There exists an $m$-dimensional countably rectifiable set $\Sigma \subset U$ and there exists an $\Hff^m$-measurable function $\theta:\Sigma\to \N$, with $\theta\ge 1$, such that $\theta \in L^1_{\loc}(\Sigma,\Hff^m)$, where $\Hff^m$ is the $m$-dimensional Hausdorff measure on $\R^N$, such that $V =\theta(x)\, \Hff^m(x)\otimes \delta_{T_x\Sigma}$. More explicitly
    \begin{equation}
        \int_{G_m(U)} \varphi(x,P) \dif V(x,P) = \int_\Sigma \varphi(x,T_x\Sigma) \, \theta(x) \dif \Hff^m(x),
    \end{equation}
    for any $\varphi \in \CS^0(G_m(U))$.
    
    If $\pi:G_m(U)\to U$ is the natural projection, the push-forward $\mu_V \coloneqq \pi_\sharp V$ is called the weight of $V$. The mass of $V$ is the number $V(G_m(U))=\mu_V(U)$. In the previous notation, we have that $\mu_V= \theta \Hff^m\res \Sigma$.
    
    We say that a sequence of $m$-varifolds $V_k$ converges to an $m$-varifold $V$ in the sense of varifolds if convergence holds in duality with continuous functions on $G_m(U)$.
\end{definition}

We observe that, given $\Sigma$ and $\theta$ as above, the weight $\mu_V= \theta \Hff^m\res \Sigma$ uniquely defines a varifold $V$. Hence, we shall usually refer to the measure $\mu_V$ on $U$ as the varifold, with a slight abuse of terminology. If $\Sigma$ is an $m$-dimensional $\CS^1$-submanifold of $M$, it readily defines a varifold by taking $\theta\equiv1$ on $\Sigma$. In this case, for the sake of simplicity, we will refer to $\Sigma$ as a varifold on $G_m(U)$.

A varifold on $U$ is in particular a varifold on $\R^N$ in the classical sense \cite{simon_lecturesgeometricmeasuretheory_1983}. Moreover, observe that a sequence of $m$-varifolds $V_k$ converges to an $m$-varifold $V$ as in the previous definition if and only if convergence holds in duality with continuous functions on $G_m(\R^N)$, i.e., if and only if the sequence converges in the usual sense of varifolds in $\R^N$.

\medskip

In the following, we denote by $Q(x)= [Q_{ij}(x)]$ the matrix identifying the orthogonal projection onto $T_xM$ in $\R^N$, for any $x \in M$. For any $f \in \CS^1(M)$ and any subspace $P$ of $T_xM$, if $\overline{f}$ is a $\CS^1$ local extension of $f$ around $x$, the quantity
\begin{equation}
    P_{ij}\partial_j f(x) \coloneqq P_{ij}\frac{\partial \overline{f}}{\partial x_j}(x),
\end{equation}
is well-defined and independent of the extension (here and below, the sum over repeated indices is understood). If $\varphi=\varphi(x,P)= \varphi((x_i), (P_{jk}))$ is a $\CS^1$-function on $\R^N\times \R^{N^2}$, we denote by
\begin{align}
    \partial_i \varphi \coloneqq \frac{\partial \varphi}{\partial x_i}, &&
    D_{jk} \varphi \coloneqq \frac{\partial \varphi}{\partial P_{jk}},
\end{align}
the partial derivatives of $\varphi$. Following \cite{hutchinson_secondfundamentalformvarifolds_1986}, it is possible to give a notion of extrinsic curvature relative to $\R^N$ for varifolds on $U$ as follows.

\begin{definition}
Let $V$ be an $m$-varifold on $U$. We say that $V$ has generalized second fundamental form relative to $\R^N$ if there exist $B_{ijk} \in L^1_{\loc}(V)$, for $i,j,k \in \set{1,\ldots, N}$, such that
\begin{equation}\label{eq:DefSFFvarifolds}
    \int_{G_m(\R^N)} P_{ij}\partial_j \varphi + B_{ijk} D_{jk} \varphi + B_{jij} \varphi \dif V =0,
\end{equation}
for any $i=1,\ldots,N$ and any $\varphi \in \CS^1_c(\R^N\times \R^{N^2})$. In such a case, $V$ is said to be a curvature varifold, and we refer to the map $B:G_m(U)\to \R^{N^3}$ as the generalized second fundamental form relative to $\R^N$.
\end{definition}

If $V$ is a curvature varifold on $U$, we define its generalized second fundamental form relative to $M$ by
\begin{align}\label{eq:DefSFFvarifoldManifold}
    \h: G_m(U) \to \R^{N^3}
    &&
    \h_{ij}^k(x,P) \coloneqq \frac12 \left( B_{ijk} + B_{jik} - B_{kij} - P_{jr} P_{iq} \partial_q Q_{rk} (x) - P_{ir} P_{jq} \partial_q Q_{rk} (x) \right)
\end{align}
for $V$-a.e. $(x,P)\in G_m(U)$. We further define $\abs{\h}^2 \coloneqq \sum_{ijk} \abs{\h_{ij}^k}^2$.

We remark that the definition in \cref{eq:DefSFFvarifoldManifold} may a priori differ from the one given in \cite[Definition 5.2.5]{hutchinson_secondfundamentalformvarifolds_1986}. In fact, the definition in \cref{eq:DefSFFvarifoldManifold} is inspired by the one given in \cite[Definition 3.5]{buetLeonardiMasnou_secondfundamformvarifolds}, and it represents its generalization to the setting of varifolds into (non-flat) Riemannian manifolds. The advantage of definition \cref{eq:DefSFFvarifoldManifold} (or of \cite[Definition 3.5]{buetLeonardiMasnou_secondfundamformvarifolds}) is that $\h_{ij}^k$ is here expressed in terms of a linear combination of the functions $B_{ijk}$'s plus a continuous function on $G_m(U)$; moreover, $\h_{ij}^k$ is manifestly symmetric with respect to indices $i,j$ by \cite[Proposition 5.2.4(i)]{hutchinson_secondfundamentalformvarifolds_1986}.

It can be readily checked by direct computations that if $V$ is induced by a smooth $m$-dimensional submanifold $\Sigma$ of $M$, and if $A:T\Sigma\times T\Sigma\to T\Sigma^\perp \subset TM$ is the classically defined second fundamental form of the embedding $\Sigma \hookrightarrow M$, then $\ip{A(Pe_i, Pe_j) | e_k} = \h_{ij}^k$, where $\{e_i\}$ is the canonical basis of $\R^N$. Moreover $\abs{\h} = \abs{A}_g$, where $\abs{A}_g$ is the usual norm of $A$ as a tensor on $\Sigma$.

Finally, we remark that definition \cref{eq:DefSFFvarifoldManifold} is equivalent to \cite[Definition 5.2.5]{hutchinson_secondfundamentalformvarifolds_1986} if and only if several (equivalent) conditions hold, which are listed in the next lemma. In fact, notice that \cref{lem:ConditionsEquivalentSFF}\cref{item:Hutchinson} is exactly \cite[Definition 5.2.5]{hutchinson_secondfundamentalformvarifolds_1986}.

\begin{lemma}\label{lem:ConditionsEquivalentSFF}
    Let $V$ be an $m$-dimensional curvature varifold on $U$. Then, the following are equivalent.
    \begin{enumerate}
        \item $(\delta_{ks}-P_{ks})(B_{ijs}-B_{jis})=0$ holds $V$-a.e.;
        \item \label{item:Hutchinson} $\h_{ij}^k(x,P)=  P_{lj}B_{ikl} - P_{lj}P_{iq} \partial_q Q_{kl}(x)$ holds $V$-a.e.;
        \item 
        $P_{is}\h^{s}_{ij}=0$ holds $V$-a.e.;
        \item $\h^k_{ij} = \frac{1}{2} (Q_{ks}-P_{ks})(B_{ijs}+B_{jis})$ holds $V$-a.e.;
        \item\label{item:2FFProjected} $\h^{k}_{ij}= (Q_{ks}-P_{ks})P_{jl}P_{ir} \h^{s}_{lr}$ holds $V$-a.e..
    \end{enumerate}
\end{lemma}

The proof of \cref{lem:ConditionsEquivalentSFF} is omitted and it follows from the algebraic properties in \cite[Proposition 5.2.4]{hutchinson_secondfundamentalformvarifolds_1986}.

\begin{definition}
Let $\mu_V=\theta \Hff^m\res \Sigma$ be an $m$-varifold on $U$. We say that $V$ has generalized mean curvature relative to $\R^N$ if there exist $\vec{\mathsf{H}}\in L^1_{\loc}(\mu_V, \R^N)$ such that
\begin{equation}\label{def:MeanCurvVarifold}
    \int_M \div_\top X \dif \mu_V = - \int_M \ip{\vec{\mathsf{H}}|X } \dif \mu_V,
\end{equation}
for any vector field $X \in \CS^1_c(\R^N)$, where $\div_\top X(x) \coloneqq \div_{T_x\Sigma} X(x)$ for $\mu_V$-a.e. $x\in M$. In such a case, $\vec{\mathsf{H}}$ is called generalized mean curvature relative to $\R^N$.
\end{definition}

If $\mu_V=\theta \Hff^m\res \Sigma$ is a curvature varifold on $U$, then it has generalized mean curvature given by 
\begin{equation}\label{eq:MeanCurvTracciaVarifold}
    \mathsf{H}^k (x) = B_{iki}(x,T_x\Sigma)
\end{equation}
for $\mu_V$-a.e. $x \in M$, where $\mathsf{H}^k$ denotes the $k$-th component of $\vec{\mathsf{H}}$ \cite[Proposition 3.10]{mantegazza_curvaturevarifoldsboundary_1996}.

If $\mu_V=\theta \Hff^m\res \Sigma$ is a varifold on $U$ with generalized mean curvature $\vec{\mathsf{H}}$ relative to $\R^N$, then its generalized mean curvature $\vec{\H}$ relative to $M$ is defined by
\begin{align}\label{eq:DefMeanCurvVarifoldManifold}
    \vec{\H}:M\to\R^N, && \H^k (x)\coloneqq \mathsf{H}^k(x) - P_{ji}(x) \partial_i Q_{kj}(x),
\end{align}
for $\mu_V$-a.e. $x \in M$, where $P(x)$ is the projection onto $T_x\Sigma$ and $\H^k$, $\mathsf{H}^k$ denotes the $k$-th component of $\vec{\H}$, $\vec{\mathsf{H}}$ respectively. Moreover, for any $m$-dimensional subspace $P$ of $T_xM$ the vector whose $k$-th component is given by $P_{ji} \partial_i Q_{kj}(x)$ is orthogonal to $M$. Hence
\begin{equation}\label{eq:MeanCurvVarifoldManifoldIntegrationbyParts}
    \int_M \div_\top X \dif \mu_V = - \int_M \ip{\vec{\H}|X} \dif \mu_V,
\end{equation}
for any vector field $X \in \CS^1_c(M)$. 

Conversely, it is readily checked that if $\mu_V$ is an $m$-varifold on $U$ and there exists $\vec{\H} \in L^1_{\loc}(\mu_V,\R^N)$ such that $\vec{\H}(x) \in T_xM$ for $\mu_V$-a.e. $x \in M$ and such that 
\begin{equation}\label{eq:DefMeanCurvvarifoldManifold2}
    \int_M \div_\top X \dif \mu_V = - \int_M \ip{\vec{\H}|X} \dif \mu_V,
\end{equation}
for any vector field $X \in \CS^1_c(M)$, then $V$ has generalized mean curvature relative to $\R^N$, and \cref{eq:DefMeanCurvVarifoldManifold} holds. Indeed, given a field $X \in \CS^1_c(\R^N)$, splitting $X(x)=Q(x)X(x) + (1-Q(x))X(x)$, then $\div_\top ((1-Q)X) = P_{ij}\partial_i\left( (\delta_{kj}-Q_{kj})X^k\right) = -P_{ij}\partial_i Q_{kj} X^k$, and the claim follows.

In view of the previous observation, we shall say that a varifold on $U$ \emph{has generalized mean curvature} if either \cref{def:MeanCurvVarifold} or \cref{eq:DefMeanCurvvarifoldManifold2} holds.

\medskip

We also define the traceless second fundamental form as follows.

\begin{definition}
    Let $V$ be an $m$-dimensional curvature varifold on $U$. Then its traceless second fundamental form relative to $M$ is defined by
    \begin{align}\label{eq:DefTracelessSFF}
        \mathring{\h}:G_m(U) \to \R^{N^3}
        &&
        \mathring{\h}_{ij}^k(x,P) \coloneqq \h_{ij}^k(x,P) - \frac1m \H^k(x) P_{ij},
    \end{align}
    for $V$-a.e. $(x,P) \in G_m(U)$.
\end{definition}

We observe that, since $P_{\alpha\beta}P_{\beta\gamma}= P_{\alpha\gamma}$ and $P_{\alpha\beta}= P_{\beta\alpha}$, there holds
\begin{equation}\label{eq:MeanCurvTracciaVarifoldManifold}
\begin{split}
 P_{ij}\h_{ij}^k &= 
 \frac12 P_{ij}\left(  B_{ijk} + B_{jik} - B_{kij} - P_{jr} P_{iq} \partial_q Q_{rk} - P_{ir} P_{jq} \partial_q Q_{rk} \right) \\
 &= \frac12 \left(
  B_{jkj} + B_{iki}  - 2 P_{rq} \partial_q Q_{rk} \right)  \\
 &\overset{\cref{eq:MeanCurvTracciaVarifold}}{=} \mathsf{H}^k - P_{rq} \partial_q Q_{rk} \overset{\cref{eq:DefMeanCurvVarifoldManifold}}{=} \H^k, 
\end{split}
\end{equation}
where in the second equality we used that $P_{ij}B_{ijk} = P_{ij}B_{ikj} = B_{jkj}$ by \cite[Proposition 5.2.4(i)]{hutchinson_secondfundamentalformvarifolds_1986} and \cite[Proposition 3.7]{mantegazza_curvaturevarifoldsboundary_1996}, similarly $P_{ij}B_{jik} = B_{iki}$, and $2P_{ij}B_{kij} = B_{kjj}=0$ by \cite[Propostion 3.6]{mantegazza_curvaturevarifoldsboundary_1996}. Moreover, since $P_{\alpha\alpha}=m$, we also have
\begin{equation}\label{eq:TracelessVarifoldQuadrato}
\begin{split}
    \abs{\mathring{\h}}^2 & \coloneqq
    \sum_{ijk}\abs{\mathring{\h}_{ij}^k}^2 =
    \abs{\h}^2 + \frac{1}{m^2}\sum_{ijk} (\H^k)^2 P_{ij}P_{ij} - \frac{2}{m} \sum_{ijk} \H^k P_{ij} \h_{ij}^k \\
    &= \abs{\h}^2 +\frac1m \abs{\vec{\H}}^2 - \frac2m \sum_k \H^k \H^k = \abs{\h}^2 -\frac1m \abs{\vec{\H}}^2.
\end{split}
\end{equation}

The next theorem collects sufficient conditions implying varifold convergence, as well as some consequences of the convergence. It is a corollary of the fundamental results proved in \cite{allard_firstvariationvarifold_1972}, \cite[Chapter 8]{simon_lecturesgeometricmeasuretheory_1983}, together with those on the theory of curvature varifolds proved in \cite{hutchinson_secondfundamentalformvarifolds_1986} and \cite{mantegazza_curvaturevarifoldsboundary_1996}.

\begin{theorem}\label{thm:ProprietaConvergenzeVarifold}
Let $(M,g)$, $U$ be as above.
Let $V_k$ be a sequence of $m$-dimensional varifolds on $U$, with $\mu_{V_k}= \theta_k \Hff^m\res \Sigma_k$.
\begin{enumerate}
    \item \label{item:varifold1}If $V_k$ has generalized mean curvature for any $k$, and $\sup_k \mu_{V_k}(U) + \int_M \abs{\vec{\H}_{V_k}}^2 \dif \mu_{V_k}<+\infty$, then there exists a subsequence that converges in the sense of varifolds to an $m$-dimensional varifold $V$ on $U$. Moreover, $V$ has generalized mean curvature and
    \begin{equation}
        \lim_{k\to+\infty} \int_M \ip{\vec{\H}_{V_k}| X}\dif \mu_{V_k} = \int_M \ip{\vec{\H}_{V}|X}\dif \mu_{V},
    \end{equation}
    for any $X \in \CS^0(M)$.
    Moreover
    \begin{equation}
        \liminf_{k\to+\infty} \int_M \abs{\H_{V_k}}^2 \dif \mu_{V_k} \ge  \int_M \abs{\H_{V}}^2\dif \mu_{V}.
    \end{equation}

    \item \label{item:varifold2}If $V_k$ converges in the sense of varifolds to an $m$-dimensional varifold $\mu_V= \theta \Hff^m\res \Sigma$, if $V_k$ is a curvature varifold such that $\sup_k \int_{G_m(U)} \abs{\h_{V_k}}^2 \dif V_k < +\infty$, then $V$ is a curvature varifold and
    \begin{equation}\label{eq:WeakConvergenceComponentsSFF}
         \lim_{k\to+\infty} \int_{G_m(U)}  \varphi(x,P)\,(\h_{V_k})_{ij}^l(x,P) \dif V_k = \int_{G_m(U)}\varphi(x,P)\,(\h_{V})_{ij}^l(x,P) \dif V,
    \end{equation}
    for any $\varphi \in \CS^0(G_m(\R^N))$, and
    \begin{align}\label{eq:LSCtracelessVarifolds}
        \liminf_{k\to+\infty} \int_{G_m(U)} \abs{\mathring\h_{V_k}}^2 \dif V_k \ge \int_{G_m(U)} \abs{\mathring\h_{V}}^2 \dif V.   \end{align}
    If in addition, $V_k$ satisfies any of the equivalent conditions of \cref{lem:ConditionsEquivalentSFF} for every $k$, then $V$ does.
\end{enumerate}
\end{theorem}

\begin{proof}
Let $V_k$ be as in \cref{item:varifold1}. Since $U$ is compact, by the above discussion and by \cref{eq:DefMeanCurvVarifoldManifold}, $V_k$ has generalized mean curvature relative to $\R^N$ and the latter is uniformly bounded in $L^2_{\mu_{V_k}}$. Hence the classical compactness theorem for sequences of varifolds applies \cite[Chapter 8]{simon_lecturesgeometricmeasuretheory_1983} and the first item follows.

Let $V_k$ be as in \cref{item:varifold2}. Since $U$ is compact, by \cite[Proposition 5.2.6 (ii)]{hutchinson_secondfundamentalformvarifolds_1986} it follows that 
\begin{equation}
\sup_k \int_{G_m(U)} \sum_{ijk} \abs{B_{ijk}}^2 \dif V_k <+\infty.
\end{equation}
Hence, \cite[Theorem 5.3.2]{hutchinson_secondfundamentalformvarifolds_1986} (see also \cite[Theorem 6.1]{mantegazza_curvaturevarifoldsboundary_1996}) implies that $V$ is a curvature varifold and that \cref{eq:WeakConvergenceComponentsSFF} holds. We now claim that
\begin{equation}\label{eq:WeakConvergenceComponentsSFFtraceless}
         \lim_{k\to+\infty} \int_{G_m(U)}  \varphi(x,P)\,(\mathring\h_{V_k})_{ij}^l(x,P) \dif V_k = \int_{G_m(U)}\varphi(x,P)\,(\mathring\h_{V})_{ij}^l(x,P) \dif V,
    \end{equation}
for any $\varphi \in \CS^0(G_m(\R^N))$. Let $\varphi \in \CS^0(G_m(\R^N))$ be fixed. For any $i,j, l$ we have
\begin{align}
\int_{G_m(U)} \varphi \H_{V_k}^l P_{ij} \dif V_k & \overset{\cref{eq:MeanCurvTracciaVarifoldManifold}}{=} \int_{G_m(U)} \varphi  P_{ij} P_{ab} (\h_{V_k})^l_{ab} \dif V_k = \int_{G_m(U)} \psi_{ab}(\h_{V_k})^l_{ab} \dif V_k \\
&\xrightarrow[k\to+\infty]{\cref{eq:WeakConvergenceComponentsSFF}}
\int_{G_m(U)} \psi_{ab}(\h_{V})^l_{ab} \dif V
\overset{\cref{eq:MeanCurvTracciaVarifoldManifold}}{=}
\int_{G_m(U)} \varphi \H_{V}^l P_{ij} \dif V
\end{align}
where $\psi_{ab} \in \CS^0(G_m(\R^N))$ is defined by $\psi_{ab}(x,P)\coloneqq \varphi(x,P) P_{ij}P_{ab}$. Hence by \cref{eq:WeakConvergenceComponentsSFF} and by definition of traceless second fundamental form, \cref{eq:WeakConvergenceComponentsSFFtraceless} follows.  Hence \cite[Theorem 4.4.2, Theorem 5.3.2]{hutchinson_secondfundamentalformvarifolds_1986} (see also \cite[Theorem 6.1]{mantegazza_curvaturevarifoldsboundary_1996}) implies \cref{eq:LSCtracelessVarifolds}. If $V_k$ satisfies any of the assumptions in \cref{lem:ConditionsEquivalentSFF}, then one can take $\phi(x,P)= f(x,P) P_{sl} $ in \cref{eq:WeakConvergenceComponentsSFF} for an arbitrary continuous function $f$ on $G_m(U)$. Summing over $l$, by arbitrariness of $f$ we get that $P_{sl} (\h_V)^{l}_{ij}=0$.  
\end{proof}

\begin{lemma}\label{lem:VarifoldWeakConvergence2ff}
Let $(M,g)$, $U$ be as above.
Let $V_k$ be a sequence of $m$-dimensional varifolds on $U$, with weight $\mu_{V_k}= \theta_k \Hff^m \res\Sigma_k$ and with mean curvature $\vec{\H}_{V_k}$ relative to $M$. Suppose that $V_k$ converges in the sense of varifolds to a varifold $V$ with weight $\mu_V= \theta \Hff^m \res\Sigma$ and mean curvature $\vec{\H}_V$ relative to $M$. Then the following holds.

\begin{enumerate}
    \item For any vector field $X \in \CS^1_c(M)$, there holds
    \begin{equation}\label{eq:ConvDeboleMeanCurvOnSpace}
        \lim_{k\to +\infty} \int_{M} \ip{ X |\vec{\H}_{V_k}}\dif \mu_{V_k} =  \int_{M} \ip{X|\vec{\H}_{V}}\dif \mu_V.
    \end{equation}

    \item If also $V_k, V$ are curvature varifolds with generalized second fundamental forms $(B_{V_k})_{ijl}, B_{ijl}$, respectively, for any $k$, if $\sup_k \int |\vec{\H}_{V_k}|^2 \dif \mu_{V_k} <+\infty$, then
    \begin{equation}\label{eq:ConvDeboleSFFgenerale}
    \lim_{k\to +\infty} \int_{M} \phi(x) (B_{V_k})_{ijl} (x,T_x\Sigma_k) \dif \mu_{V_k} =  \int_{M} \phi(x) B_{ijl}(x,T_x\Sigma) \dif \mu_V,
    \end{equation}
    \begin{equation}\label{eq:ConvDeboleSFFmanifoldgenerale}
    \lim_{k\to +\infty} \int_{M} \phi(x) (\h_{V_k})_{ij}^l (x,T_x\Sigma_k) \dif \mu_{V_k} =  \int_{M} \phi(x) (\h_V)_{ij}^l(x,T_x\Sigma) \dif \mu_V,
    \end{equation}
    for any $\phi \in \CS^0_c(M)$ and any $i,j,l$. Moreover, \cref{eq:ConvDeboleSFFmanifoldgenerale} also holds with $({\h}_{V_k})_{ij}^l, ({\h}_V)_{ij}^l$ replaced by $(\mathring{\h}_{V_k})_{ij}^l, (\mathring{\h}_V)_{ij}^l$.
\end{enumerate}
\end{lemma}

\begin{proof}
The limit \cref{eq:ConvDeboleMeanCurvOnSpace} readily follows by varifold convergence as
\begin{equation}
    \int_{M} \ip{ X |\vec{\H}_{V_k}}\dif \mu_{V_k} = - \int_{G_m(U)} \div_P(X) \dif V_k \xrightarrow[k]{} \int_{G_m(U)} \div_P(X) \dif V = \int_{M} \ip{X|\vec{\H}_{V}}\dif \mu_V,
\end{equation}
where $\div_P(X)$ denotes tangential divergence of $X$ along a subspace $P$, hence $\varphi(x,P) = (\div_P(X) )(x)$ is continuous on $G_m(U)$.\\
Assume now the hypotheses in item (2). We first claim that
\begin{equation}\label{eq:ConvDeboleMeanCurvOnGrassmann}
        \lim_{k\to +\infty} \int_{G_m(U)} \varphi (x,P)(\H_{V_k})^l(x)\dif V_k =  \int_{G_m(U)} \varphi(x,P) (\H_{V})^l(x) \dif V,
\end{equation}
for any $\varphi \in \CS^0_c(\R^N \times \R^{N^2})$. Indeed, recalling \cref{eq:DefMeanCurvVarifoldManifold} and since $V_k(G_m(U))$ is uniformly bounded by varifold convergence, regarding $\vec{\H}_{V_k}, \vec{\mathsf{H}}_{V_k}$ as functions on $G_m(U)$, there holds $\sup_k \int_{G_m(U)} \abs{\vec{\mathsf{H}}_{V_k}}^2 \dif V_k <+\infty$. Hence $\int (\mathsf{H}_{V_{k_r}})^l \varphi(x,P) \dif V_{k_r} \to \int T^l(x,P) \varphi(x,P) \dif V = \int T^l(x,T_x\Sigma) \varphi(x,T_x\Sigma) \dif \mu_V$ for any $\varphi \in \CS^0_c(\R^N \times \R^{N^2})$, for some $T^l \in L^2(V)$ and along some subsequence $V_{k_r}$. Let $X$ be a compactly supported field of class $\CS^1$ in $\R^N$. Testing the previous convergence against the components $\varphi=X^l(x)$, as in item (1) we get
\begin{equation}
    \int T^l(x,T_x\Sigma) X^l \dif \mu_V = \lim_r -\int \div_\top  X \dif \mu_{V_{k_r}} = -\int \div_\top  X \dif \mu_V = \int \ip{X | \vec{\mathsf{H}}_V} \dif \mu_V,
\end{equation}
thus $T^l = (\mathsf{H}_V)^l = B_{jlj}$ $V$-a.e. for any $l$, where we used \cref{eq:MeanCurvTracciaVarifold}. By the same reasoning, any subsequence of $(B_{V_k})_{jlj} V_k$ has a subsequence that converges to $ B_{jlj} V$ as measures on $G_m(U)$. Hence $(B_{V_k})_{jlj} V_k \to B_{jlj} V$ in duality with $\CS^0_c(\R^N \times \R^{N^2})$ functions. Recalling \cref{eq:DefMeanCurvVarifoldManifold}, \cref{eq:ConvDeboleMeanCurvOnGrassmann} follows.

Now we shall take $\varphi(x,P) = \phi(x) P_{rs}$, for fixed indices $r,s$, as a test function in the definition of generalized second fundamental form \cref{eq:DefSFFvarifolds}. Since $D_{jk} \varphi = \phi D_{jk}(P_{rs}) = \phi \delta_{jr}\delta_{ks}$, by \cref{eq:ConvDeboleMeanCurvOnGrassmann} we obtain
\begin{equation}
    \begin{split}
        \int \phi(x) (B_{V_k})_{irs}(x,T_x\Sigma_k) \dif \mu_{V_k} 
        &=
        \int (B_{V_k})_{ijk} D_{jk}\varphi \dif V_k
        =
        -\int P_{ij}\partial_j\varphi - (B_{V_k})_{jij}\varphi \dif V_k\\
        &\xrightarrow[k\to+\infty]{}
        -\int P_{ij}\partial_j\varphi - B_{jij}\varphi \dif V =  \int \phi(x) B_{irs} (x,T_x\Sigma)\dif \mu_V,
    \end{split}
\end{equation}
for any $i,r,s$, proving \cref{eq:ConvDeboleSFFgenerale}. Recalling \cref{eq:DefSFFvarifoldManifold}, \cref{eq:ConvDeboleSFFmanifoldgenerale} then follows from \cref{eq:ConvDeboleSFFgenerale} and from varifold convergence. Finally, recalling \cref{eq:DefTracelessSFF}, the last assertion in item (2) follows from \cref{eq:ConvDeboleSFFmanifoldgenerale} and \cref{eq:ConvDeboleMeanCurvOnGrassmann}. Indeed
\begin{equation}
    \int_M \varphi(x) (\H_{V_k})^l (T_x\Sigma_k)_{ij} \dif \mu_{V_k} = \int_{G_m(U)} \varphi(x) P_{ij}  (\H_{V_k})^l(x) \dif V_k \xrightarrow[k]{} \int_{G_m(U)} \varphi(x) P_{ij}  (\H_{V})^l(x) \dif V,
\end{equation}
for any $i,j,l$ by \cref{eq:ConvDeboleMeanCurvOnGrassmann}.
\end{proof}

\begin{remark}\label{rmk:weakSard}
    Let $w$ be a locally Lipschitz function defined on a $n$-dimensional manifold $(M,g)$ and suppose that $\set{w \leq t}$ is compact in $M$ for every $t \in (a,b)$, $a,b \in \R$. By coarea formula (see \cite{maggi_setsfiniteperimetergeometric_2012}), $\Hff^{n-1}(\partial \set{w\leq t} \cap \Crit(w)) =0$ for almost every $t\in(a,b)$. Indeed,
    \begin{equation}
        0 = \int_{\set{a <w <b}\cap \Crit(w)}\abs{\nabla w } \dif \Hff^n = \int_a^b\Hff^{n-1}(\partial \set{w\leq t} \cap \Crit(w)) \dif t.
    \end{equation}
\end{remark}

\begin{lemma}\label{thm:MeanCruvatureDivergence}
Let $D\subset M$ be an open precompact Lipschitz subset of an $n$-dimensional Riemannian manifold $(M,g)$.
Let $a<b$, and let $w \in \CS^1(D)$ be a function such that $\{a\le w \le b\}\Subset D$ and $w \in \CS^\infty(D\smallsetminus\set{\nabla w=0})$. Suppose that $\Sigma_t\coloneqq \partial\set{ w \le t} \cap D$ is an $(n-1)$-dimensional varifold with generalized mean curvature $\vec\H_t$, and that $\nabla w\neq0$ at $\Hff^{n-1}$-a.e. point on $\Sigma_t$, for almost every $t\in (a,b)$.

Then
\begin{equation}\label{eq:zzIdentity}
    \vec{\H}_t =
        - \div \left( \frac{\nabla
     w}{\abs{\nabla w}}\right) \frac{\nabla
     w}{\abs{\nabla w}}
     \qquad
     \text{$\Hff^{n-1}$-a.e. on $\Sigma_t$, for a.e. $t\in(a,b)$.}
\end{equation}
If also $w$ is Lipschitz on bounded open subsets of $D$, and if $f \in \CS^0(D) \cap  W^{1,p}(A)$ for some $p>1$ for any bounded open set $A\subset D$, then
\begin{equation}\label{eq:IntegrationByPartsSobolevOnVarifold}
    - \int_{\Sigma_t} f ( \ip{\vec{\H}_t | X} + \div_\top X) \dif \Hff^{n-1} = \int_{\Sigma_t} \ip{ \nabla^\top f | X} \dif \Hff^{n-1},
\end{equation}
for a.e. $t\in(a,b)$, for any vector field $X \in \CS^1_c(M)$.
\end{lemma}

\begin{proof}
The identity \cref{eq:zzIdentity} holds since $\Sigma_t$ is a smooth hypersurface around $\Hff^{n-1}$-a.e. point, hence the mean curvature vector can be classically computed $\Hff^{n-1}$-almost everywhere on $\Sigma_t$.

To derive \cref{eq:IntegrationByPartsSobolevOnVarifold}, let $f_k \in \CS^\infty(D)$ be a sequence of functions converging to $f$ in $W^{1,p}$ on bounded open subsets of $D$ and locally uniformly on $D$. Fix $X \in \CS^1_c(M)$. Hence
\begin{equation}
    \begin{split}
        - \int_{\Sigma_t} f ( \ip{\vec{\H}_t  | X} + \div_\top X)
        = \lim_k - \int_{\Sigma_t} f_k ( \ip{\vec{\H}_t  | X} + \div_\top X) 
        = \lim_k \int_{\Sigma_t}  \ip{ \nabla^\top f_k | X},
    \end{split}
\end{equation}
for a.e. $t\in(a,b)$. Moreover
\begin{equation}
    \int_a^b \int_{\Sigma_t} \abs{\nabla^\top f_k}^p \le \norm{\nabla w}_{L^\infty(\set{a< w< b})}  \int_{\set{a< w< b}} \abs{\nabla f_k}^p.
\end{equation}
Hence $\liminf_k \int_{\Sigma_t} \abs{\nabla^\top f_k}^p <+\infty$ for a.e. $t \in (a,b)$. Finally, since we can assume, up to a subsequence, that $\nabla f_k$ converges to $\nabla f$ pointwise almost everywhere on $D$, it follows that $\lim_k \nabla^\top f_k(x) = \nabla^\top f(x)$ for $\Hff^{n-1}$-a.e. $x \in \Sigma_t$ for a.e. $t\in (a,b)$. Hence for a.e. $t \in (a,b)$ there is a (nonrelabeled) subsequence $\nabla^\top f_k$ converging to $\nabla^\top f$ in $L^q(\Hff^{n-1}(\Sigma_t))$ for any $q \in [1,p)$. Hence $\lim_k \int_{\Sigma_t}  \ip{ \nabla^\top f_k | X}= \int_{\Sigma_t}  \ip{ \nabla^\top f | X}$ for a.e. $t \in (a,b)$, and \cref{eq:IntegrationByPartsSobolevOnVarifold} follows.
\end{proof}

We conclude this part with some convergence results about sequences of sets of finite perimeter in relation to varifold convergence in codimension one. For basic definitions and theory about sets of finite perimeter and functions of bounded variation, we refer the reader to \cite{ambrosio_functionsboundedvariationfree_2000, maggi_setsfiniteperimetergeometric_2012, miranda_heatsemigroupfunctionsbounded_2007}.

\begin{lemma}\label{lem:area_convergence_under_BV_convergence}
Let $(M,g)$ be an $n$-dimensional Riemannian manifold, and let $D\subset M$ be an open subset.
Let $\phi_k \in {\rm BV}_{\loc}(D)$ be a sequence of functions on $M$ that converges to a function $\phi \in {\rm BV}_{\loc}(D)$ in $L^1$ on any bounded open subset of $D$. Suppose that there exists $c\in\R$ and a compact set $K\subset \overline{D}$ such that $\mu(\set{\phi_k<c}\smallsetminus K)=0$ for any $k \in \N$, and that
\begin{equation}\label{eq:AxStrictBVconvergence}
\lim_{k\to+\infty}  \abs{D \phi_k}(E) =  \abs{D \phi}(E),     
\end{equation}
for any bounded open set $E \subset D$. Then, for any $a<b<c$ the functions $(a,b)\ni t\mapsto A_k(t) \coloneqq \Hff^{n-1}(\partial^* \set{\phi_k < t } \cap D)$ 
converge in measure to $(a,b)\ni t\mapsto A(t) \coloneqq \Hff^{n-1}(\partial^* \set{\phi < t }\cap D)$, where $\partial^*$ denotes essential boundary.\\
In particular, there exists a subsequence $(k_l)_l \subset \N$ such that
\begin{align}
     \lim_{l\to+\infty} \Hff^{n-1}(\partial^* \set{\phi_{k_l} < t }\cap D) =\Hff^{n-1}(\partial^* \set{\phi < t }\cap D),
\end{align}
for almost every $t \in(a,b)$.
\end{lemma}

\begin{proof}
By assumptions, we can fix $K, \Omega$ a compact and a bounded open set, respectively, such that $K\subset \overline{\Omega}$ and $K \cap D \subset \Omega \subset D$, such that $\mu(\set{\phi_k<c}\smallsetminus K)=0$ for any $k$. Since $\phi_k\to \phi$ in $L^1(\Omega)$, Fubini's Theorem implies that $\set{\phi_k<t}\to \set{\phi<t}$ in $L^1(\Omega)$ for a.e. $t \in \R$  (see, e.g., the argument in \cite[Remark 13.11]{maggi_setsfiniteperimetergeometric_2012}). Hence $\mu(\set{\phi<b'}\smallsetminus K)=0$ for some $b'\in(b,c)$.

Letting $A_k, A$ be as in the statement, since $\mu(\set{\phi_k < b} \smallsetminus K) = \mu(\set{\phi < b}\smallsetminus K)=0$, then $A_k(t) = \Hff^{n-1}(\Omega \cap \partial^* \set{\phi_k < t} )$ and $A(t) = \Hff^{n-1}(\Omega \cap \partial^* \set{\phi < t} )$ for a.e. $t \in (a,b)$. Hence by coarea formula the functions $A_k, A$ are well-defined in $L^1(a,b)$, indeed
\begin{equation}
    \int_a^b A_k(t) \dif t =  \abs{D \phi_k} (\set{a< \phi_k<  b}) \leq  \abs{D \phi_k} (\Omega)\xrightarrow[k\to+\infty]{\cref{eq:AxStrictBVconvergence}}  \abs{D \phi} (\Omega)<+\infty,
\end{equation}
and analogously $\int_a^b A(t) \dif t  \leq \abs{D \phi} (\Omega) <+\infty$.\\
Fix $\delta>0$. We aim at proving that
    \begin{equation}
         \limsup_{k\to +\infty}  \abs{\set{ \abs{ A_k(t)- A(t)}>\delta}}=0.
    \end{equation}
    First of all, observe that
    \begin{equation}\label{eq:limsup_integral_aree_max}
    \begin{split}
        \limsup_{k\to +\infty}\int_{\set{A\geq A_k}} A(t)-A_k(t)\dif t &\leq \limsup_{k\to +\infty} \int_{\set{A\geq A_k}} \max\set{ A(t)-A_k(t),0} \dif t \\
        & \leq \limsup_{k\to +\infty} \int_a^b \max\set{ A(t)-A_k(t),0} \dif t\\
        &\leq  \int_a^b \limsup_{k\to +\infty}\max\set{ A(t)-A_k(t),0} \dif t,
        \end{split}
    \end{equation}
    where the last inequality is ensured by applying a reverse Fatou Lemma, since $\max\set{A-A_k,0}\leq \max\set{A,0}=A\in L^1(a,b)$. Since $\set{\phi_k< t} \to \set{ \phi < t} $ in $L^1(\Omega)$ for almost every $t \in [a,b]$, by lower semicontinuity of the perimeter we find
    \begin{equation}\label{eq:limsup_aree_leq_0}
        \limsup_{k\to +\infty} A(t) - A_k(t) =A(t)- \liminf_{k \to +\infty} A_k(t)\leq 0
    \end{equation}
    for almost every $t \in [a,b]$. Hence, by \cref{eq:limsup_aree_leq_0,eq:limsup_integral_aree_max} we get
    \begin{equation}\label{eq:limit area_on_A>An}
         \limsup_{k\to +\infty}\int_{\set{A\geq A_k}} A(t)-A_k(t)\dif t=0.
    \end{equation}
    By Markov's inequality we also have    \begin{equation}\label{eq:almost_measure_convergence}
    \begin{split}
        \limsup_{k\to +\infty} \abs{ \set{A(t)-A_k(t)> \delta}} &\le 
        \frac{1}{\delta}\limsup_{k\to +\infty}\int_{\set{A> \delta+ A_k}} A(t)-A_k(t)\dif t \\
        &\le\frac{1}{\delta}\limsup_{k\to +\infty}\int_{\set{A\geq A_k}} A(t)-A_k(t)\dif t \overset{\textbox{\cref{eq:limit area_on_A>An}}}{=}0.
        \end{split}
    \end{equation}
    Finally, \cref{eq:AxStrictBVconvergence} with $E=\Omega$ is equivalent to
    \begin{equation}
        \lim_{k\to+\infty} \int_{-\infty}^{+\infty} \Hff^{n-1}(\Omega \cap \partial^* \set{\phi_k < t} ) \dif t =  \int_{-\infty}^{+\infty} \Hff^{n-1}(\Omega \cap \partial^* \set{\phi < t} ) \dif t,
    \end{equation}
    hence the lower semicontinuity property $\liminf_k \Hff^{n-1}(\Omega \cap \partial^* \set{\phi_k < t} ) \ge \Hff^{n-1}(\Omega \cap \partial^* \set{\phi < t} )$ holding for a.e. $t\in\R$, implies that $\lim_{k} \int_\alpha^\beta \Hff^{n-1}(\Omega \cap \partial^* \set{\phi_k < t} ) \dif t =  \int_\alpha^\beta \Hff^{n-1}(\Omega \cap \partial^* \set{\phi < t} ) \dif t$ for any $\alpha\le \beta$. Hence we obtain
    \begin{equation}\label{eq:convergence_of_integrals}
    \begin{split}
        \lim_{k\to+\infty} \int_a^b A_k(t) \dif t &= \lim_{k\to +\infty} \int_a^b \Hff^{n-1}(\Omega \cap \partial^* \set{\phi_k < t} ) \dif t =  \int_a^b \Hff^{n-1}(\Omega \cap \partial^* \set{\phi < t} ) \dif t \\& = \int_a^b A(t) \dif t.
    \end{split}
    \end{equation}
    
    Summing up we get
    \begin{align}
        \abs{\set{ \abs{ A_k(t)- A(t)}>\delta}}&\leq \abs{\set{ A_k(t)- A(t)>\delta}}+\abs{\set{ A(t)- A_k(t)>\delta}}\\
        \overset{\textbox{Markov}}&{\leq}\frac{1}{\delta} \int_{\set{A_k>A+\delta}} A_k(t)-A(t) \dif t +\abs{\set{ A(t)- A_k(t)>\delta}}\\
        &\leq \frac{1}{\delta} \int_a^b A_k(t)-A(t) \dif t - \frac{1}{\delta} \int_{\set{A_k\leq A}} A_k(t)-A(t) \dif t+\abs{\set{ A(t)- A_k(t)>\delta}},
    \end{align}
    which, together with 
    \cref{eq:limit area_on_A>An}, \cref{eq:almost_measure_convergence}, \cref{eq:convergence_of_integrals}, yields
    \begin{equation}
        \limsup_{k \to +\infty}\abs{\set{ \abs{ A_k(t)- A(t)}>\delta}} =\liminf_{k\to +\infty}  \abs{\set{ \abs{ A_k(t)- A(t)}>\delta}}=0
    \end{equation}
    for every $\delta >0 $, proving the convergence in measure.
\end{proof}

\begin{corollary}\label{cor:area_convergence_under_W11_convergence}
Let $(M,g)$ be an $n$-dimensional Riemannian manifold, and let $D\subset M$ be an open subset.
Let $\phi_k,\phi:D\to\R$ be Lipschitz on any open bounded subset of $D$. Suppose that $\phi_k$ converges to $\phi$ in $W^{1,1}$ on any open bounded subset of $D$. Suppose that there exists $c\in\R$ and a compact set $K\subset \overline{D}$ such that $\mu(\set{\phi_k\le c}\smallsetminus K)=0$ for any $k \in \N$. Then, for any $a<b<c$ the functions $(a,b)\ni t\mapsto A_k(t) \coloneqq \Hff^{n-1}(\partial \set{\phi_k \le t } \cap D)$ 
converge in measure to $(a,b)\ni t\mapsto A(t) \coloneqq \Hff^{n-1}(\partial \set{\phi \le t }\cap D)$.\\
In particular, there exists a subsequence $(k_l)_l \subset \N$ such that
\begin{align}
     \lim_{l\to+\infty} \Hff^{n-1}(\partial \set{\phi_{k_l} \le t }\cap D) =\Hff^{n-1}(\partial \set{\phi \le t }\cap D),
\end{align}
for almost every $t \in(a,b)$.
\end{corollary}

\begin{proof}
The statement directly follows from \cref{lem:area_convergence_under_BV_convergence} recalling that for locally Lipschitz functions $\phi_k, \phi$ as in the assumptions, the sets $\partial \set{\phi_k \le t }$ and $\partial^* \set{\phi_k < t }$ (resp. $\partial \set{\phi \le t }$ and $\partial^* \set{\phi < t }$) are equivalent up to $\Hff^{n-1}$-negligible sets for almost every $t \in (a,b)$ (see, e.g., the proof of \cite[Theorem 18.1]{maggi_setsfiniteperimetergeometric_2012}).
\end{proof}

\begin{corollary}\label{cor:ConvergenzaLevelSetsVersusVarifolds}
Let $(M,g)$ be an $n$-dimensional Riemannian manifold.
\begin{enumerate}
    \item\label{item:areaconvergence1} Let $D\subset M$ be open and let $E_k\subset D$ be a sequence of uniformly bounded sets of relative finite perimeter, i.e., $\Hff^{n-1}(\partial^*E_k \cap D)<+\infty$, converging in $L^1(D)$ to a set of finite perimeter $E$. Let $\mu_{V_k}\coloneqq \Hff^{n-1}\res \partial^* E_k \cap D$ and suppose that $\mu_{V_k}$ converges to some varifold $\mu_V$ in the sense of varifolds.\\
    Then $\mu_V \ge \Hff^{n-1}\res \partial^* E \cap D$, and $\mu_V = \Hff^{n-1}\res \partial^* E \cap D$ if and only if $\lim_k \Hff^{n-1}(\partial^*E_k \cap D)=\Hff^{n-1}(\partial^*E \cap D)$.

    \item\label{item:areaconvergence2}  Let $\phi_k, \phi, D,K, a,b,c$ be as in \cref{cor:area_convergence_under_W11_convergence}. If
    \begin{align}
     \kern2cm \lim_{k\to+\infty} \Hff^{n-1}(\partial \set{\phi_{k} \le t } \cap D) =\Hff^{n-1}(\partial \set{\phi \le t } \cap D),
    \end{align}
    and $\mu_{k,t}\coloneqq\Hff^{n-1}\res\partial \set{\phi_{k} \le t }\cap D$ is a varifold with generalized mean curvature $\vec{\H}_{k,t}$ such that $\sup_k \norm{\vec{\H}_{k,t}}_{L^2(\mu_{k,t})} < +\infty$ for a.e. $t \in(a,b)$, then $\Hff^{n-1}\res\partial \set{\phi_{k} \le t }\cap D$ converges to $\Hff^{n-1}\res\partial \set{\phi \le t }\cap D$ in the sense of varifolds as $k\to+\infty$ for a.e. $t \in (a,b)$.
\end{enumerate}
\end{corollary}

\begin{proof}
For any $x\in M$ let $r_i\searrow0$ be such that $\mu_{V_k}(\partial B_{r_i}(x))  = \mu_V(\partial B_{r_i}(x))=0$ for any $k$. Hence,
\begin{equation}
\frac{(\partial^* E \cap B_{r_i}(x)) }{\abs{\B^{n-1}}r_i^{n-1}}
\Hff^{n-1}\le \liminf_{k\to+\infty} \frac{\mu_{V_k}(\overline{B}_{r_i}(x))}{\abs{\B^{n-1}}r_i^{n-1}} \le \frac{\mu_V(B_{r_i}(x))}{\abs{\B^{n-1}}r_i^{n-1}}.
\end{equation}
By rectifiability of $\partial^*E$, at $\Hff^{n-1}$-a.e. point $x \in \partial^* E \cap D$ there exists the limit 
\begin{equation}
    \lim_{r \to 0^+}\frac{\Hff^{n-1}(\partial^* E \cap B_{r}(x)) }{\abs{\B^{n-1}}r^{n-1}}=1.
\end{equation}
Thus, we get
\begin{equation}
    1 \le \liminf_{i\to+\infty}\frac{\mu_V(B_{r_i}(x))}{\abs{\B^{n-1}}r_i^{n-1}} \le \limsup_{r\to0} \frac{\mu_V(B_{r}(x))}{\abs{\B^{n-1}}r^{n-1}}.
\end{equation}
Hence $\mu_V \ge \Hff^{n-1}\res\partial^* E \cap D$ (see, e.g., \cite[Chapter 1, Section 3]{simon_lecturesgeometricmeasuretheory_1983}). Moreover
\begin{equation}
    \Hff^{n-1}(\partial^* E \cap D) \le \liminf_{k\to+\infty} \Hff^{n-1}(\partial^* E_k \cap D) \le  \limsup_{k\to+\infty} \Hff^{n-1}(\partial^* E_k \cap D)  = \mu_V(M),
\end{equation}
where in the last equality we used that since $E_k\subset K$ for any $k$ for some compact set $K$, then $\Hff^{n-1}(\partial^* E_k \cap D) = \mu_{V_k}(M) = \int_M \varphi \dif \mu_{V_k} \to  \int_M \varphi \dif \mu_{V} = \mu_V(M)$ where $\varphi \in \CS^0(M)$ is any function with compact support such that $\varphi \ge \chi_K$.
Therefore $\mu_V = \Hff^{n-1}\res \partial^* E \cap D$ if and only if $\lim_k \Hff^{n-1}(\partial^*E_k \cap D)=\Hff^{n-1}(\partial^*E \cap D)$, which proves \cref{item:areaconvergence1}.

To prove \cref{item:areaconvergence2}, applying \cref{thm:ProprietaConvergenzeVarifold} we have that for a.e. $t\in (a,b)$ there is a subsequence $k_l$ (a priori depending on $t$) such that $\mu_{k_l,t}$ converges to some $\mu_t$ in the sense of varifolds. Hence, \cref{item:areaconvergence1} implies that $\mu_t = \Hff^{n-1}\res \partial\set{ \phi\le t}\cap D$, for a.e. $t \in (a,b)$. The same argument shows that any subsequence of $\mu_{k,t}$ has a subsequence converging to $\Hff^{n-1}\res \partial\set{ \phi\le t}\cap D$, hence the full sequence $\mu_{k,t}$ converges to $\Hff^{n-1}\res \partial\set{\phi\le t}\cap D$.
\end{proof}

The next lemma contains a convergence result for the unit normals to level sets of a sequence of converging functions.

\begin{lemma}\label{lem:convergence_of_unit_normal} Let $(M,g)$ be a $n$-dimensional Riemannian manifold and let $D \subset M$ be an open subset. Let $\phi_k,\phi:D\to \R$ be Lipschitz on any open bounded subset of $D$. Suppose that $\phi_k$ converges in $L^1_{\loc}(A)$ to $\phi$ and $\nabla \phi_k\rightharpoonup \nabla \phi$ weakly in $L^1(A)$ as $k\to +\infty$, for any open bounded set $A\subset D$. Suppose that there exists $c\in\R$ and a compact set $K\subset \overline{D}$ such that $\mu(\set{\phi_k\le c}\smallsetminus K)=0$ for any $k \in \N$ and
\begin{equation}
    \int_K\abs{\nabla \phi_k} \to \int_K \abs{\nabla \phi}.
\end{equation}
Then, denoting by $\nu^k$ (resp. $\nu$) the $L^\infty$ vector field equal to $\nabla \phi_k/\abs{\nabla \phi_k}$ a.e. where $\nabla \phi_k\neq0$ (resp. $\nabla \phi/\abs{\nabla \phi}$ a.e. where $\nabla \phi\neq0$), and equal to $0$ a.e. where $\nabla\phi_k=0$ (resp. $\nabla\phi=0$), for any $a<b<c$ we have 
\begin{equation}
    \lim_{k\to +\infty}\int_a^b \int_{\partial \set{\phi_k \leq t}}\abs{\nu^{k} - \nu}^2 \dif \Hff^{n-1} \dif t =0,
\end{equation}
\end{lemma}

\begin{proof}
 Observe that $\abs{\nu^{k}-\nu}^2 \leq 2- 2 \ip{\nu^{k}| \nu}$. Hence
 \begin{equation}
 \begin{split}
     \lim_{k\to +\infty}&\int_a^b \int_{\partial \set{\phi_k \leq t}}\abs{\nu^{k} - \nu}^2 \dif \Hff^{n-1} \dif t \leq 
      \lim_{k\to +\infty} \int_a^b \int_{K \cap \partial \set{\phi_k \leq t}}2-2 \ip{\nu^{k}|\nu} \dif \Hff^{n-1} \dif t \\
      & \le  \lim_{k\to +\infty} \int_K 2 \abs{\nabla \phi^k} - 2\ip{\nabla \phi^k | \nu} 
      =\int_K 2 \abs{\nabla \phi} - 2\ip{\nabla \phi | \nu} = 0.
 \end{split}
 \end{equation}
\end{proof}

We conclude with a joint lower semicontinuity lemma.

\begin{lemma}\label{lem:JointConvergenceLSC}
Let $(M,g)$ be an $n$-dimensional Riemannian manifold, and let $D$ be an open subset. Let $\lambda_k$ be a sequence of nonnegative Radon measures on $D$ weakly*-converging to a Radon measure $\lambda$. For $p\in(1,\infty)$, let $Y_k$ be a sequence of vector fields in $L^p(\lambda_k)$ such that $\sup_k \norm{Y_k}_{L^p(\lambda_k)}<\infty$.

Then there exists a vector field $Y \in L^p(\lambda)$ such that, up to extracting a subsequence, there holds
\begin{equation}
    \lim_{k\to+\infty} \int \ip{ Y_k | X } \dif \lambda_k = \int \ip{ Y | X } \dif \lambda ,
    \qquad
    \qquad
    \liminf_{k\to+\infty} \norm{Y_k}_{L^p(\lambda_k)} \ge \norm{Y}_{L^p(\lambda)},
\end{equation}
for any vector field $X \in \CS^0_c(D)$.

If in addition there exists a nonnegative Radon measure $\mu$ such that $\lambda_k \le \kst\mu$ for any $k$, then
\begin{equation}
    \lim_{k\to+\infty} \int \ip{ Y_k | X } \dif \lambda_k = \int \ip{ Y | X } \dif \lambda,
\end{equation}
for any vector field $X \in L^{p'}(\mu)$.
\end{lemma}

\begin{proof}
    The first part of the proof follows by standard precompactness of measures and Riesz representation (see, e.g., the proof of \cite[Lemma 10.1]{ambrosio_lecturesoptimaltransport_2021}).

    As for the second part of the proof, since $\mu$ is Radon, for any $X \in L^{p'}(\mu)$ and any $\varepsilon>0$ there exists $X_\varepsilon \in \CS^0_c(D)$ such that $\norm{X - X_\varepsilon}_{L^{p'}(\mu)} < \varepsilon$. By assumption, $X \in L^{p'}(\lambda_k)$ for any $k$, $X \in L^{p'}(\lambda)$ as $\lambda\le \mu$, and $\norm{X - X_\varepsilon}_{L^{p'}(\lambda_k)} \le \norm{X - X_\varepsilon}_{L^{p'}(\mu)}< \varepsilon$ for any $k$. Hence
    \begin{equation}
        \begin{split}
            \limsup_k \abs{\int  \ip{Y_k | X} \dif \lambda_k - \int  \ip{Y | X} \dif \lambda } 
            & \le \limsup_k \abs{\int  \ip{Y_k | X_\varepsilon} \dif \lambda_k - \int  \ip{Y | X} \dif \lambda} +
            \abs{\int  \ip{Y_k | X-X_\varepsilon} \dif \lambda_k } \\
            &\le \abs{\int  \ip{Y | X_\varepsilon - X} \dif \lambda } +  \varepsilon\sup_k\norm{Y_k}_{L^p(\lambda_k)} \\
            &\le \kst\norm{Y}_{L^p(\lambda)} \, \norm{X_\varepsilon- X}_{L^{p'}(\mu)} +  \varepsilon\sup_k\norm{Y_k}_{L^p(\lambda_k)} . 
        \end{split}
    \end{equation}
    Letting $\varepsilon\to0$, the last claim follows.
\end{proof}

\section{Geometric evolution equations in nonlinear potential theory}
For ease of future reference, we record here the evolution equations in nonlinear potential theory. This is a version of \cref{thm:monotonicity-NPT}, tailored for computations.

\begin{theorem}
    Let $p\in(1,2]$, $\alpha > (n-p)/(n-1)$. Let $\Omega \subset M$ be a closed bounded set with $\CS^{1,\beta}$-boundary, for some $\beta>0$. Let $w_p$ be the solution to the problem \cref{eq:moser_potential} for some $\Phi \in \Lip(\overline{D})$, where $D$ is a precompact subset with smooth boundary containing $\Omega$, and $\inf_{\partial D}\Phi>0$. Fix $T<\inf_{\partial D} \Phi$. Assume that $\partial \Omega$ has generalized mean curvature $\H \in L^2(\partial \Omega)$. Then, the following statements hold
    \begin{enumerate}
        \item for almost every $t \in [0,T)$, $\partial \Omega^{\sml{p}}_t$ is smooth out of $\set{\abs{\nabla w_p}=0}\cap \partial \Omega^{\sml{p}}_t$ which is closed and $\Hff^{n-1}$-negligible;
        \item for almost every $t \in [0,T)$, $\partial \Omega^{\sml{p}}_t$ is a curvature varifold with square integrable second fundamental form;
        \item denoting $\H$ the mean curvature in the sense of varifolds, for almost every $t \in [0,T)$
        \begin{equation}
            \H= \abs{\nabla w_p} - (p-1) \frac{\ip{\nabla \abs{\nabla w_p}| \nabla w_p}}{\abs{\nabla w_p}^2}
        \end{equation}
        holds almost everywhere on $\partial \Omega^{\sml{p}}_t$;
        \item for almost every $t \in [0,T)$, the quantity
        \begin{equation}
            \sca^\top = \sca - 2 \Ric(\nu,\nu) + \H^2 - \abs{\h}^2
        \end{equation}
        is a well-defined integrable function on $\partial \Omega^{\sml{p}}_t$ which agrees with the classical induced scalar curvature almost everywhere on $\partial \Omega^{\sml{p}}_t$;
        \item for almost every $t \in [0,T)$, a weak Gauss--Bonnet theorem holds in dimension $n=3$ (\cref{cor:GB}).
    \end{enumerate}
    \smallskip
    
    Let $F_p,G_p:[0,T)\to \R$ be defined as
    \begin{align}
        G_p(t)&\coloneqq \int_{\partial \Omega^{\sml{p}}_t} \abs{\nabla w_p}^{\alpha+ p-1} \dif \Hff^{n-1},\\
        F_p(t)& \coloneqq \int_{\partial \Omega^{\sml{p}}_t} \abs{\nabla w_p}^{\alpha+ p-2} \H \dif \Hff^{n-1}.
    \end{align}
    Then,
    \begin{enumerate}[resume]
        \item $G_p \in W^{2,1}_{\loc}(0,T)$ and $F_p\in W^{1,1}_{\loc}(0,T)$;
        \item in the sense of Sobolev functions we have
        \begin{align}
            G'_p(t) &= \left(\frac{\alpha}{p-1} +1\right) G_p(t) - \frac{\alpha}{p-1} F_p(t)\\
            F'_p(t) &= - \int_{\partial \Omega^{\sml{p}}_t} \abs{\nabla w_p}^{\alpha +p-3} ( \SQ_p + \Ric(\nu,\nu)) \dif \Hff^{n-1} +\frac{\alpha(n-1)}{(n-p)(p-1)}\left(\frac{n-1}{n-p}-\frac{1}{\alpha} \right) G_p(t)\\
            & \qquad -\left(\frac{\alpha(n+p-2)}{(n-p)(p-1)}-\frac{p}{p-1}\right)F_p(t),
        \end{align}
        where
        \begin{equation}
            \SQ_p=\left[\alpha-(2-p)\right] \frac{\abs{\nabla^\top \abs{\nabla w_p}}^2}{\abs{\nabla w_p}^2} + \abs{ \mathring{\h}}^2 + \frac{1}{p-1} \left[ \alpha- \frac{n-p}{n-1}\right] \left[ {\H} - \frac{n-1}{n-p}\abs{ \nabla w_p}\right]^2;
        \end{equation}
    
        \item  $G_p$ is continuous and $F_p$ is upper semicontinuous at $t=0$.
    \end{enumerate}
\end{theorem}

\end{appendices}






\begingroup
\setlength{\emergencystretch}{1em}
\printbibliography
\endgroup

\end{document}